\documentclass{amsart}

\usepackage{amsmath,amssymb,amsthm}
\usepackage{hyperref}
\usepackage{xcolor}
\usepackage{bbm}

\newtheorem{thmintr}{Theorem}
\newtheorem{prop}{Proposition}[subsection]
\newtheorem{thm}[prop]{Theorem}
\newtheorem{lem}[prop]{Lemma}

\newtheorem{defn}[prop]{Definition}

\theoremstyle{definition}

\newtheorem{rem}[prop]{Remark}
\newtheorem{exwith}[prop]{Example}

\newtheorem*{ack}{Acknowledgement}
\newtheorem*{coi}{Conflict of interest}
\newtheorem{remintr}{Remark}

\def\co{\colon\thinspace}

\newcommand{\bfa}{\mathbf a}

\newcommand{\bfb}{\mathbf b}

\newcommand{\C}{\mathbb C}
\newcommand{\CC}{\mathcal C}

\newcommand{\rmd}{\mathrm d}
\newcommand{\D}{\mathbb D}

\newcommand{\rme}{\mathrm e}

\newcommand{\EE}{\mathcal E}

\newcommand{\GG}{\mathcal G}

\newcommand{\HH}{\mathcal H}

\newcommand{\rmi}{\mathrm i}

\newcommand{\JJ}{\mathcal J}

\newcommand{\KK}{\mathcal K}

\newcommand{\MM}{\mathcal M}

\newcommand{\N}{\mathbb N}
\newcommand{\NN}{\mathcal N}

\newcommand{\bfp}{\mathbf p}

\newcommand{\Q}{\mathbb Q}
\newcommand{\bfq}{\mathbf q}

\newcommand{\R}{\mathbb R}
\newcommand{\RR}{\mathcal R}

\renewcommand{\SS}{\mathcal S}

\newcommand{\UU}{\mathcal U}
\newcommand{\bfu}{\mathbf u}

\newcommand{\VV}{\mathcal V}

\newcommand{\WW}{\mathcal W}

\newcommand{\XX}{\mathcal X}

\newcommand{\Z}{\mathbb Z}
\newcommand{\ZZ}{\mathcal Z}

\newcommand{\DB}{\bar{\partial}}
\newcommand{\CR}[1]{\bar{\partial}_{#1}}

\newcommand{\lra}{\longrightarrow}
\newcommand{\ra}{\rightarrow}

\DeclareMathOperator{\cut}{\mathrm{cut}}

\DeclareMathOperator{\ev}{\mathrm{ev}}

\DeclareMathOperator{\Gl}{\mathrm{Gl}}

\DeclareMathOperator{\Hom}{\mathrm{Hom}}

\DeclareMathOperator{\id}{\mathrm{id}}
\DeclareMathOperator{\im}{\mathrm{Im}}
\DeclareMathOperator{\ind}{\mathrm{ind}}

\DeclareMathOperator{\Int}{\mathrm{Int}}

\DeclareMathOperator{\loc}{\mathrm{loc}}

\DeclareMathOperator{\ot}{\mathrm{ot}}

\DeclareMathOperator{\PSl}{\mathrm{PSl}}
\DeclareMathOperator{\PSO}{\mathrm{PSO}}

\DeclareMathOperator{\Sl}{\mathrm{Sl}}

\DeclareMathOperator{\SO}{\mathrm{SO}}

\newcommand{\boldxi}{\mbox{\boldmath $\xi$}}

\begin{document}

\author{Wolfgang Schmaltz, Stefan Suhr, and Kai Zehmisch}
\address{Fakult\"at f\"ur Mathematik, Ruhr-Universit\"at Bochum,
Universit\"atsstra{\ss}e 150, D-44801 Bochum, Germany}
\email{Wolfgang.Schmaltz@rub.de, Stefan.Suhr@rub.de, Kai.Zehmisch@rub.de}

\title[Non-fillability of overtwisted contact manifolds via polyfolds]
{Non-fillability of overtwisted contact manifolds via polyfolds}

\date{08.08.2024}

\begin{abstract}
  We prove that any weakly symplectically fillable contact manifold is tight.
  Furthermore we verify the strong Weinstein conjecture for contact manifolds
  that appear as the concave boundary of a directed symplectic cobordism
  whose positive boundary satisfies the weak-filling condition and is overtwisted.
  Similar results are obtained in the presence of bordered Legendrian open books
  whose binding--complement has vanishing second Stiefel--Whitney class.
  The results are obtained via polyfolds.
\end{abstract}

\subjclass[2010]{53D42; 53D40, 57R17, 53D45, 37J55, 34C25, 37C27}
\thanks{This research is supported by the {\it Institute for Advanced Study} (Princeton, NJ)
in 2015, by the {\it Mathematisches Forschungsinstitut Oberwolfach} in 2020, and is part of
a project in the SFB/TRR 191 {\it Symplectic Structures in Geometry, Algebra and Dynamics},
funded by the DFG}

\maketitle


\section{Introduction\label{sec:intro}}

In \cite{eli89} Eliashberg introduced a dichotomy of closed contact $3$-manifolds,
the tight and overtwisted contact structures.
He established in \cite{eli89} an $h$-principle
in the sense of Gromov \cite{gro86} for overtwisted contact structures.
The higher dimensional analogue
was defined by Borman--Eliashberg--Murphy \cite{bem15}.
One way to detect tight contact structures on a $3$-manifolds is
to find a weak symplectic filling.
In view of the filling--by--holomorphic--discs technique
such fillable contact manifolds cannot be overtwisted,
see \cite{eli90,gro85} and cf.\ \cite[Corollary 3.8]{gz13}.
In higher dimensions
obstructions to overtwistedness in terms of
semi-positive weak symplectic fillings
were obtained by Niederkr\"uger \cite{nie06} and
Massot--Niederkr{\"u}ger--Wendl \cite{mnw13}.
The aim of this work is to remove the assumption
of being semi-positive.

We consider not necessarily connected
$(2n-1)$-dimensional contact manifolds $(M,\xi)$ and assume
that there is a contact form $\alpha$ on $M$
defining $\xi$, i.e.\ $\xi$ is the kernel of $\alpha$.
The restriction of $\rmd\alpha$ to $\xi$ is a symplectic form
providing $\xi$ with the symplectic orientation via $(\rmd\alpha)^{n-1}$.
The contact manifold $(M,\xi)$ is oriented by $\alpha\wedge(\rmd\alpha)^{n-1}$.
These notions are independent of the choice of contact form
as long as the contact form equals $f\alpha$
for a positive smooth function $f$ on $M$.

A compact $2n$-dimensional symplectic manifold $(W,\Omega)$
provided with the symplectic orientation $\Omega^n$
is called a {\bf weak symplectic filling} of a given
$(2n-1)$-dimensional contact manifold $(M,\xi)$,
if $\partial W=M$ as oriented manifolds,
where $\partial W$ carries the boundary orientation,
such that the following condition is satisfied:
For all choices of positive contact forms $\alpha$ for $\xi$
the differential forms
\[
\alpha\wedge\omega^{n-1}
\quad\text{and}\quad
\alpha\wedge\big(\rmd\alpha+\omega\big)^{n-1}\,,
\quad\text{where}\quad
\omega:=\Omega|_{TM}\,,
\]
are positive volume forms on $M$,
see \cite{dg12,mnw13}.
Fixing a contact form $\alpha$ for $\xi$ the latter is equivalent to
\[
\alpha\wedge\big(f\rmd\alpha+\omega\big)^{n-1}>0
\]
for all non-negative smooth functions $f$ on $M$.
A contact manifold $(M,\xi)$ is {\bf weakly symplectically fillable}
if it admits a weak symplectic filling.

If $(M,\xi)$ contains an overtwisted disc,
then $(M,\xi)$ is called {\bf overtwisted};
otherwise $(M,\xi)$ is called {\bf tight},
see \cite{bem15} and Section \ref{subsec:overtwistedness}.

\begin{thmintr}
 \label{fillableimpliestight}
 Any weakly symplectically fillable contact manifold is tight.
\end{thmintr}

Potentially, Theorem \ref{fillableimpliestight} can be obtained
with Pardon's \cite{par19} rigorously defined contact homology.
An argument is indicated in Remark \ref{remintr:contacthomology}
below.
We will prove Theorem \ref{fillableimpliestight}
along the classical line of reasoning
due to Gromov \cite{gro85} and Eliashberg \cite{eli90},
cf.\ \cite{zehm03,nie06,mnw13}.
In fact, Theorem \ref{fillableimpliestight},
will follow as a special case of Theorem \ref{thm:maindirectsyplcoborthm} (ii).
For that we remark, that a contact manifold,
which contains an overtwisted disc,
also contains a parallelisable small plastikstufe
whose core is a torus,
see the discussion in Section \ref{subsec:overtwistedness}
and Theorem \ref{thm:bemthmonplastik}.
A plastikstufe is an example of a
bordered Legendrian open book
such that the book fibration is trivial,
the page is a product of an interval with the binding
and the binding is the core,
see Section \ref{subsec:legopenbooks}.
Whenever a bordered Legendrian open book
in an ambient closed contact manifold is {\bf small},
i.e.\ has a contractible neighbourhood,
it was shown in \cite[Theorem 4.4]{mnw13}
that no {\it semi-positive} weak symplectic filling can exists.
The way in which the theorem is formulated suggests
the conjecture that the statement should be true
even without the assumption of semi-positivity.
Here we prove:

\begin{thmintr}
 \label{fillableimpliesnospinblob}
 A contact manifold is not weakly symplectically fillable
 provided that it contains
 a small bordered Legendrian open book
 such that the complement of the binding has
 vanishing second Stiefel--Whitney class.
\end{thmintr}

Theorem \ref{fillableimpliesnospinblob} implies
Theorem \ref{fillableimpliestight} and directly follows from
Theorem \ref{thm:maindirectsyplcoborthm} (ii).
The examples of small bordered Legendrian open books
given in \cite[Proposition 5.9]{mnw13} all have
vanishing second Stiefel--Whitney class
though they are sometimes not orientable and, hence, are not spin,
see Example \ref{ex:kleinbottleexample}.
This leaves the question,
whether there are contact manifolds
that are (i) weakly symplectically fillable
(and therefore tight with Theorem \ref{fillableimpliestight})
and
(ii) that admit a bordered Legendrian open book,
whose binding--complement could be orientable, but is not spin.
Note that non of the potential weak symplectic fillings can be semi-positive. 

Restricting to weak symplectic fillings that are semi-positive for a moment
Theorem \ref{fillableimpliesnospinblob} holds true also
if the second Stiefel--Whitney class does not vanish.
The reason is, that a compact $1$-dimensional manifold
has an even number of boundary components
which is used in a typical Gromov--Witten--invariant type argument
performed in a potential weak symplectic filling.
Taking holomorphic discs with boundary
on the bordered Legendrian open book
that intersect a given path connecting
the binding with the boundary inside a page
yields a $1$-dimensional moduli space.
At the end of the path on the binding
there is a foliation by boundary circles of Bishop discs;
at the other end, which corresponds to the boundary of the page,
no holomorphic disc homotopic
to a Bishop disc does exist.
After perturbing the almost complex structure
no bubbling off takes place
for the relevant moduli space in a semi-positive setting.
In other words, the $1$-dimensional moduli space
is compact with an odd number of boundary components.
This is not possible thus contradicts the presence of a weak symplectic filling.

Note that this is in contrast to the $1$-dimensional branched manifolds
that appear in the non semi-positive setting.
Namely, in general, the solution space
of a perturbed Cauchy--Riemann operator branches,
because of the need of multisections
near nodal holomorphic discs
with multiply covered sphere bubbles of negative first Chern number.
The vanishing assumption for the second Stiefel--Whitney class
in the non semi-positive setting
allows to orient the solution space
(see Remark \ref{rem:homotopunitrivw2=0})
resulting in an oriented compact
$1$-dimensional weighted branched manifold,
which has an even number of boundary components,
i.e.\ yields the desired contradiction to weak symplectic fillability.
We remark, that moduli spaces of holomorphic discs
in general are not orientable in contrast to the case of spheres,
see \cite{fooo09b}.

As a consequence of Theorem \ref{fillableimpliesnospinblob}
we can verify the conjecture stated in \cite[Theorem 5.13 (a)]{mnw13}
involving circular contactisations of Liouville domains also
called {\bf Giroux domaines}:

\begin{thmintr}
 \label{fillableimpliesnodoublegd}
 A contact manifold is not weakly symplectically fillable
 provided that it contains a domain that is obtained from a Giroux domain
 with disconnected boundary,
 where one boundary component is blown down via a contact cut.
\end{thmintr}

The construction in Example \ref{ex:kleinbottleexample}
yields a small bordered Legendrian open book
whose binding--complement has trivial second Stiefel--Whitney class.
Hence, Theorem \ref{fillableimpliesnodoublegd} follows from
Theorem \ref{fillableimpliesnospinblob}.

Theorem \ref{thm:maindirectsyplcoborthm} also verifies instances
of the Weinstein conjecture,
which asks for the existence of periodic Reeb orbits
for all closed contact manifolds, see \cite{wein79}.
For a short historical review see \cite[Section 1]{dgz14}.
Our approach, besides the usage of polyfolds,
is based on the work of Hofer \cite{hofe93}, Albers--Hofer \cite{ah09}
and Niederkr{\"u}ger--Rechtman \cite{nr11},
and yields so-called Reeb links:
A {\bf Reeb link} is a finite collection of parametrised periodic Reeb orbits
each of which is oriented by the corresponding Reeb vector field
and possibly multiply covered, see \cite{ach05}.
A Reeb link is called {\bf null-homologous} if the link components counted multiplicity
add up to zero in homology.
The {\bf strong Weinstein conjecture}
as formulated by Abbas--Cieliebak--Hofer in \cite{ach05}
asserts the existence of a null-homologous
Reeb link for all contact forms on all closed contact manifolds.

\begin{thmintr}
 \label{strongWConnegend}
 The strong Weinstein conjecture holds true
 for all contact manifolds
 that appear as the concave boundary
 of a directed symplectic cobordism
 whose positive end satisfies the weak-filling condition and
 that are at least one of the following:
 \begin{enumerate}
 \item [(i)]
  an overtwisted contact manifold,
 \item [(ii)]
  a contact manifold that
  contains a small bordered Legendrian open book
  such that the complement of the binding has
  vanishing second Stiefel--Whitney class,
 \item [(iii)]
  a contact manifold that
  contains a domain that is obtained from a Giroux domain
  with disconnected boundary,
  where one boundary component is blown down via a contact cut.
 \end{enumerate}
\end{thmintr}

The relevant notions related to directed symplectic cobordisms
can be found in Section \ref{subsec:adirsymplcobor}.
Theorem \ref{strongWConnegend} follows from 
Theorem \ref{thm:maindirectsyplcoborthm}
together with the remarks made for 
Theorem \ref{fillableimpliestight} and
\ref{fillableimpliesnodoublegd} above.

For the proof of Theorem \ref{thm:maindirectsyplcoborthm},
which implies Theorem \ref{fillableimpliesnospinblob} and \ref{strongWConnegend},
we will use the following alternative characterisation of 
weak symplectic fillability from \cite{mnw13}:
A compact symplectic manifold $(W,\Omega)$ is a weak symplectic filling
of a contact manifold $(M,\xi=\ker\alpha)$ if $\partial W=M$ as oriented manifolds
and if there exists an $\Omega$-tamed almost complex structure $J$ on $W$
such that $\xi$ is $J$-invariant and the restriction of $\rmd\alpha$ to $\xi$ tames $J$.
In this situation $(W,\Omega,J)$ is called a {\bf tamed pseudo-convex}
manifold, see \cite{eli90} and cf.\ Section \ref{sec:tamedpseidpconv}.
This point of view allows the use of holomorphic disc fillings in the sense of Bishop,
see Section \ref{sec:holomdiscs}.

It turns out that fillability questions can be perfectly described
in the language of symplectic cobordisms,
see Section \ref{sec:symplcobord}.
Assuming non-existence of Reeb links of the Reeb flows
that appear on the negative ends of the symplectic cobordism
the Gromov--Witten--invariant type polyfolds
can be defined in the sense of Hofer--Wysocki--Zehnder
\cite{hwz98,hwz03,hwz-I07,hwz-II09,hwz-III09,hwz-int10,hwz-dm12,hwz-gw17}.
This was observed in \cite{sz17} in the context of holomorphic spheres.
Necessary modifications for the usage of holomorphic discs instead
are worked out in Section \ref{sec:adelignemumfordtypespace} and
\ref{sec:polyfoldperturbations}.
Special attention we pay to orientability questions.
Similar to the polyfold version of the Deligne--Mumford space
we introduce a Riemann moduli space of boundary un-noded stable discs
with $3$ ordered boundary marked points
in Section \ref{sec:adelignemumfordtypespace}.
In Section \ref{sec:polyfoldperturbations}
we define the relevant polyfold of stable boundary un-noded disc maps
motivated by the absence of boundary disc bubbling
in the Gromov compactification of the
appearing moduli space of holomorphic discs.

\begin{remintr}
\label{remintr:contacthomology}
Contact homology, as a formal concept, was introduced
by Eliashberg--Givental--Hofer in \cite{egh00}
as contact-manifold-invariant having functorial properties.
Symplectic cobordisms,
which are directed from the negative to the positive end,
induce structure preserving maps (e.g.\ unital) from the contact homology
at the positive end to the contact homology at the negative end.
To incorporate non-exact cobordisms and weak-filling boundary conditions
a change of coefficients to a Novikov completion of the group ring
of the second homology (or an adapted quotient thereof)
of the symplectic cobordism resp.\ contact manifold
is necessary by compactness reasons,
see Bourgeois--van Koert \cite[Section 1.1]{bvk10},
Latschev--Wendl \cite[Section 2]{lw11},
and Niederkr{\"u}ger--Wendl \cite[Section 2.5]{nw11}.
Applying this, one gets that
a contact manifold with vanishing contact homology
(i.e.\ with $1$ being a boundary)
cannot be weakly symplectically fillable
(i.e.\ symplectically null-cobordant with empty negative end),
as the contact homology of the empty contact manifold equals the coefficient ring
and, therefore, never vanishes meaning $1\neq0$.

Combining Casals--Murphy--Presas' \cite[Theorem 1.1]{cmp19} and
Bourgeois--van Koert's \cite[Theorem 1.3]{bvk10} shows
that on every overtwisted contact manifold
there exists a non-degenerate, defining contact form
that admits a periodic Reeb orbit
that bounds precisely one finite energy plane,
which additionally is Fredholm regular,
implying the vanishing of the contact homology.
As Pardon rigorously defined contact homology in \cite{par19},
this implies, as stated on \cite[p.~835/6]{par19}, the vanishing of the contact homology
of overtwisted contact manifolds.
Furthermore, after reworking \cite{par19} with group ring coefficients,
this yields symplectic non-fillability even in the weak sense,
i.e.\ Theorem \ref{fillableimpliestight}.

Similarly and removing the word `strong',
part (i) of Theorem \ref{strongWConnegend}
could follow along the same line of reasoning
because the vanishing of the contact homology 
implies the Weinstein conjecture for the underlying contact manifold.
To obtain the strong Weinstein conjecture
as verified in part (i) of Theorem \ref{strongWConnegend}
one could argue as above in the case of non-degenerate contact forms.
In order to handle degenerate contact forms one could use
an approximation argument as in \cite[Section 6]{sz17}.
A filtered version of contact homology might yield the required energy, resp.,
action bounds.

On \cite[p.~836]{par19} Pardon addresses the vanishing of contact homology
in the presence of a small bordered Legendrian open book
based on an idea of Bourgeois--Niederkr{\"u}ger \cite[p.\ 69]{bou09}.
It would be of interest,
whether Pardon's approach to
broken homolorphic discs with boundary
via orientation local systems could remove
the assumption on the second Stiefel--Whitney class,
which we made in Theorem \ref{fillableimpliesnospinblob} and
in part (ii) of Theorem \ref{strongWConnegend}
in order to make our approach via boundary un-noded
stable holomorphic discs feasible.
The Deligne--Mumford space elaborated
in Section \ref{sec:adelignemumfordtypespace}
might represent a first step in this direction.
\end{remintr}


\section{Singular Legendrian foliations}
\label{sec:singlegfol}


\subsection{Legendrian open books}
\label{subsec:legopenbooks}

Following \cite{mnw13,nie13} we define:

\begin{defn}
\label{defn:relativeOB}
A {\bf relative open book decomposition}
$(B,\vartheta)$ of a connected manifold $N$
with boundary $\partial N$ consists of
  \begin{itemize}
    \item
    a non-empty codimension $2$ submanifold $B$
    of $\Int N$, called the {\bf binding},
    \item
    and a smooth, locally trivial fibration
    $\vartheta\co N\setminus B\ra S^1$,
    whose fibres $\vartheta^{-1}(\theta)$,
    $\theta\in S^1$, are called the {\bf pages},
  \end{itemize}
such that the following conditions are satisfied:
  \begin{itemize}
    \item[(i)]
    All pages intersect $\partial N$ transversally.
    \item[(ii)]
    The binding $B$ has a
    trivial tubular neighbourhood $B\times D^2$ in $N$
    in which $\vartheta$ is given by the angular coordinate
    in the $D^2$-factor.
  \end{itemize}
\end{defn}

\noindent
The pages $\vartheta^{-1}(\theta)$ in $N\setminus B$
are co-oriented by the orientation of $S^1$,
i.e.\ the linearisation $T\vartheta$ maps positive normal vectors
to positive tangent vectors of $S^1$.

As in \cite[Section 4]{mnw13} and 
\cite[Section I.4]{nie13} we define:

\begin{defn}
\label{defn:bLOB}
  A connected compact $n$-dimensional submanifold $N$
  with boundary $\partial N$ 
  of a $(2n-1)$-dimensional contact manifold $(M,\xi)$
  carries a
  {\bf bordered Legendrian open book} $(B,\vartheta)$
  if $(B,\vartheta)$ is a relative open book decomposition of $N$
  such that
  \begin{itemize}
    \item
    the pages of $(B,\vartheta)$
    are Legendrian submanifolds of $(M,\xi)$ and
    \item
    the {\bf singular set} of $N\subset(M,\xi)$,
    i.e.\ the set of all points $p\in N$
    such that $T_pN\subset\xi_p$,
    is equal to $B\cup\partial N$.
  \end{itemize}
\end{defn}

\noindent
In particular, the binding of a bordered Legendrian open book
is an isotropic submanifold of $(M,\xi)$; the boundary $\partial N$
is Legendrian.
The complement
\[
N^*:=N\setminus\big(B\cup\partial N\big)
\]
of $B\cup\partial N$ in $N$
is the set of {\bf regular points} of $N\subset(M,\xi)$,
i.e.\ the set of all points $p\in N$
such that $T_pN$ and $\xi_p$ intersect transversally.
The {\bf characteristic distribution} $TN^*\cap\xi$ integrates by
the Frobenius theorem to the so-called {\bf characteristic foliation} on $N^*$.
The {\bf characteristic leaves}, which by definition
are the leaves of the characteristic foliation, coincide with the pages
of the open book $(B,\vartheta)$.

If, in addition, $\xi$ is co-oriented,
then $\xi|_{N^*}$ puts a co-orientation
to the pages $\vartheta^{-1}(\theta)$ in $N\setminus B$.
We will assume that this co-orientation coincides with
the co-orientation induced by $\vartheta$
by possibly composing $\vartheta$
with a reflection on $S^1=\partial D$.

It follows from
\cite[Theorem I.1.3]{nie13}
or \cite[Theorem 1.4]{hua15}
that the germ of a contact structure $(M,\xi)$ is unique
near a submanifold $N\subset(M,\xi)$
(with boundary $\partial N$)
that carries a bordered Legendrian open book $(B,\vartheta)$.
The germ is uniquely determined
by the {\bf singular characteristic distribution} $\xi\cap TN$
given by the open book decomposition on $N$
determined by $(B,\vartheta)$.

A bordered Legendrian open book $(B,\vartheta)$
is called {\bf small} if the supporting submanifold
$N\subset(M,\xi)$ is contained in a ball inside $M$.

\begin{exwith} {\bf (A non-spin bordered Legendrian open book)}
\label{ex:kleinbottleexample}
  In \cite[Proposition 5.9]{mnw13} examples of contact manifolds $(M,\xi)$
  are constructed that contain a submanifold $N$,
  which carries a small bordered Legendrian open book.
  Some of the in \cite[Proposition 5.9]{mnw13}
  constructed examples are indeed non-spin.
  In order to see this, we repeat the essential construction steps here.
    
  The construction starts with a cylindrical Lagrangian submanifold $L$
  of an ideal Liouville domain $V$ with disconnected boundary
  $\partial V=\partial_+ V\cup\partial_- V$
  (see \cite[Theorem C]{mnw13})
  such that $L$ has disconnected boundary
  $\partial L=\partial_+ L\cup\partial_- L$
  with $\partial_{\pm} L\subset\partial_{\pm} V$.
  A perturbation of $L\times S^1$ inside the interior of the
  circular contactisation $V\times S^1$ of the ideal Liouville domain $V$
  -- a so-called Giroux domain --
  followed by a contact cut along $\partial_-V\times S^1$
  (see \cite[Section 5.1]{mnw13}), say,
  yields a bordered Legendrian open book $N$.
  Gluing a Giroux domain along $\partial_+V\times S^1$ with \cite[Lemma 5.1]{mnw13}
  and eventually cutting remaining boundary components
  yields a contact embedding of $N$ into a closed contact manifold $(M,\xi)$. 
  
  In the process, $L$ is the result of Polterovich surgery (see \cite{pol91}) along,
  say, two transverse intersection points of a Hamiltonian deformation of
  two boundary parallel Lagrangian discs.
  If the dimension of $L$ is even,
  then the Polterovich surgery result $L$ necessarily is homotopy equivalent
  to a $n$-dimensional Klein bottle with two points removed,
  see \cite[Paragraph 7]{pol91}.
  Therefore, $N^*$ is not orientable with $w_2(N^*)=0$.
  If the dimension of $L$ is odd,
  one can choose orientations such that $L$
  is homotopy equivalent
  to a $n$-dimensional Klein bottle
  or to $S^1\times S^{n-1}$ each time with two points removed.
\end{exwith}


\subsection{Overtwistedness}
\label{subsec:overtwistedness}

A $(2n-1)$-dimensional contact manifold $(M,\xi)$
is called {\bf overtwisted},
if $(M,\xi)$ contains an overtwisted disc,
see \cite{bem15}.
For example 
$\R^3$ equipped with the contact structure
$\ker\alpha_{\ot}$,
\[
\alpha_{\ot}:=\cos r\,\rmd z+r\sin r\,\rmd\theta\;,
\]
is overtwisted,
as $D_{\ot}^2:=\{z=0,r\leq\pi\}$
is an overtwisted disc.

For a $(n-2)$-dimensional
closed smooth manifold $Q$
we consider the contact manifold
$\R^3\times T^*Q$ equipped with contact structure
$\xi_Q:=\ker(\alpha_{\ot}+\lambda_{T^*Q})$,
where we identify $Q$
with the zero section in $T^*Q$
and denote the Liouville $1$-form of
$T^*Q$ by $\lambda_{T^*Q}$.
Following \cite[Section 10]{bem15}
we define the {\bf model plastikstufe with core} $Q$
to be the subset
$P_Q:=D_{\ot}^2\times Q$ of $\big(\R^3\times T^*Q,\xi_Q\big)$.

We will say that $(P_Q,\xi_Q)$
admits a contact embedding into a
$(2n-1)$-dimensional contact manifold $(M,\xi)$
if a neighbourhood of $P_Q$
in $\big(\R^3\times T^*Q,\xi_Q\big)$ does.
In this case
the image $N$ of the model $P_Q$
is called a {\bf plastikstufe with core} $Q$
and
carries the structure of a bordered Legendrian open book
with binding $Q$ and
pages corresponding to $I_{\theta}\times Q$,
where $I_{\theta}$ is the straight line segment
in $D_{\ot}^2\subset\R^2$
connecting $0$ and $\pi\rme^{\rmi\theta}$.
For the original definition of a plastikstufe
and the relation to
bordered Legendrian open books
we refer to \cite{nie06,nie13}.

\begin{thm}
  [Borman--Eliashberg--Murphy \cite{bem15}]
  \label{thm:bemthmonplastik}
  Let $Q$ be a $(n-2)$-dimensional
  closed smooth manifold,
  whose complexified tangent bundle is trivial.
  Then any $(2n-1)$-dimensional
  overtwisted contact manifold
  admits a small plastikstufe with core $Q$.
\end{thm}

The assumption on $Q$ is satisfied
for any stably parallelisable manifold $Q$,
cf.\ \cite[Section 1.1]{alp94}.
For the converse of Theorem \ref{thm:bemthmonplastik} we note:

\begin{thm}
  [Huang \cite{hua17}]
  \label{thm:huangthmonplastik}
  If a contact manifold $(M,\xi)$ contains a plastikstufe,
  then $(M,\xi)$ is overtwisted.
\end{thm}

A forerunner version of this result, which is \cite[Theorem 1.2]{hua17},
was given in \cite{cmp19}.
Further,
it is shown in \cite[Theorem 1.3]{hua17}
that if $(M,\xi)$ admits a
bordered Legendrian open book and $\dim M=5$,
then $(M,\xi)$ is overtwisted.
In fact,
a contact manifold $(M,\xi)$ is overtwisted precisely
if $(M,\xi)$ contains a bordered Legendrian open book
with pages diffeomorphic to $P\times\Sigma$,
where $P$ is a closed manifold and
$\Sigma$ a compact surface with boundary,
see \cite[Corollary 1.4]{hua17}.


\subsection{Local model near the binding}
\label{subsec:locmodnearthebind}

Let $(M,\xi)$ be a contact manifold.
We consider a
bordered Legendrian open book decomposition
$(B,\vartheta)$ of a submanifold $N\subset(M,\xi)$.
By Definition \ref{defn:relativeOB}
the binding $B\subset N$
admits a tubular neighbourhood $B\times D^2$
on which the fibre projection $\vartheta$ is given by
$(b,z=r\rme^{\rmi\theta})\mapsto\theta$.

By \cite[Proposition 4]{nie06}
a neighbourhood of $B\times D^2$ in $(M,\xi)$
is contactomorphic to a neighbourhood of
$\{0\}\times D^2\times B$ in
$\big(\R\times\C\times T^*B,\ker\alpha_o\big)$,
where
\[
\alpha_o:=
\rmd t
+\tfrac12\big(x\rmd y-y\rmd x\big)
+\lambda_{T^*B}\;,
\]
denoting by $t,z=x+\rmi y$ the coordinates on $\R\times\C$
and by $\lambda_{T^*B}$ the Liouville $1$-form on $T^*B$.
The contactomorphism restricts to
$(b,z)\mapsto(0,z,b)$ on $B\times D^2$.
Again we identify $B$ with the zero section in $T^*B$.


\subsection{Local model near the boundary}
\label{subsec:locmodnearthebdary}

Consider a contact manifold $(M,\xi)$.
Let $N\subset(M,\xi)$ be a submanifold
that supports a
bordered Legendrian open book
$(B,\vartheta)$.
By Definition \ref{defn:relativeOB}
the restriction of $\vartheta$ to $\partial N$
induces a locally trivial fibration over $S^1$
with fibre $F$.
Denoting the monodromy diffeomorphism
by $\varphi\co F\ra F$
this fibration is equivalent to the mapping torus
\[
M(\varphi)=
\frac
{[0,2\pi]\times F}
{(2\pi,f)\sim\big(0,\varphi(f)\big)}
\]
of $\varphi$.
The induced diffeomorphism
\[
\varphi^*:=(T\varphi^{-1})^*\co T^*F\ra T^*F
\]
naturally preserves the Liouville $1$-form $\lambda_{T^*F}$
so that the mapping torus
\[
M(\varphi^*)=
\frac
{T^*[0,2\pi]\times T^*F}
{\big((r,2\pi),u\big)\sim\big((r,0),(T\varphi^{-1})^*(u)\big)}
\]
carries the Liouville form $r\rmd\theta+\lambda_{T^*F}$
and can be identified with $T^*\big(M(\varphi)\big)$.
The corresponding Liouville vector field is of the form
$r\partial_r+Y_{T^*F}$ and $M(\varphi^*)$ fibres naturally
over $(T^*S^1,r\rmd\theta)$
with fibre projection map
$[(r,\theta),u]\mapsto[(r,\theta)]$.

We equip $\R\times M(\varphi^*)$ with the contact form
$\alpha_{\varphi}$ induced by
\[
\alpha_{\varphi}\equiv\rmd t+r\rmd\theta+\lambda_{T^*F}
\;.
\]
By \cite[Lemma 4.6]{mnw13}
a neighbourhood of $\partial N\subset(M,\xi)$
is contactomorphic to a neighbourhood of
\[
\{0\}\times M(\varphi)
\subset
\big(\R\times M(\varphi^*),\ker\alpha_{\varphi}\big)
\;,
\]
so that
\[
\big(\{0\}\times\{r=0,u=0\}\big)
\equiv
\{0\}\times M(\varphi)
\]
corresponds to $\partial N$ and
a neighbourhood of $\partial N$ in $N\subset M$
corresponds to the quotient of the set
$\{0\}\times\{r\leq0,u=0\}$
in $\{0\}\times M(\varphi^*)$.
We orient $T^*S^1\equiv\R\times S^1$ by $\rmd r\wedge\rmd\theta$.
This matches the co-orientation conventions
for the singular distribution determined by $\xi$
and the pages of $(B,\vartheta)$ in Section \ref{subsec:legopenbooks}.


\section{Holomorphic discs}
\label{sec:holomdiscs}


\subsection{A germ of Bishop disc filling}
\label{subsec:agermofbishopdisfill}

Motivated by Section \ref{subsec:locmodnearthebind}
we define a natural almost complex structure $J$
on $\R\times(\R\times\C\times T^*B)$
that allows a lifting
of obvious holomorphic discs similar to
cf.\ \cite[Section 2]{gz16b}.

For that choose a Riemannian metric $g_B$ on $B$.
Denote by $J_{T^*B}$
the almost complex structure on $T^*B$
that is induced by the Levi-Civita connection of $g_B$,
see \cite[Appendix B]{nie06} or \cite[Section 5]{kwz22}.
Observe that $J_{T^*B}$ is compatible
with the symplectic form $\rmd\lambda_{T^*B}$.
Furthermore
denoting by $g_B^{\flat}$
the dual metric of $g_B$
the kinetic energy function on $T^*B$ is defined by
$k(u)=\frac12g^{\flat}(u,u)$, $u\in T^*B$,
and satisfies
$\lambda_{T^*B}=-\rmd k\circ J_{T^*B}$.
In other words,
$k$ is a strictly plurisubharmonic potential
in the sense of \cite[Section 3.1]{gz12}.

On the Liouville manifold
\[
  (V,\lambda_V):=
  \Big(
    \C\times T^*B,\tfrac12\big(x\rmd y-y\rmd x\big)
    +\lambda_{T^*B}
  \Big)
\]
we consider the almost complex structure
\[
J_V=\rmi\oplus J_{T^*B}
\;,
\]
which is compatible with the symplectic form $\rmd\lambda_V$.
The function
\[
\psi(z,u):=\tfrac14|z|^2+k(u)
\]
is a strictly plurisubharmonic potential $\psi$
on $(V,\lambda_V,J_V)$
satisfying
$\lambda_V=-\rmd\psi\circ J_V$.

The contactisation
$(\R\times V,\rmd t+\lambda_V)$
of $(V,\lambda_V)$ is given by
$(\R\times\C\times T^*B,\alpha_o)$
and the contact structure $\ker\alpha_o$
is spanned by vectors of the form
$v-\lambda_V(v)\partial_t$, $v\in TV$.
Using coordinates $s$ on the first $\R$-factor
of $\R\times(\R\times V)$
we define $J$ by requiring
$J(\partial_s)=\partial_t$ and
\[
J\big(v-\lambda_V(v)\partial_t\big)
=
J_Vv-\lambda_V(J_Vv)\partial_t
\]
for all $v\in TV$.
In other words,
$J$ is a $s$-translation invariant almost complex structure
on $\R\times(\R\times V)$
that preserves the contact distributions $\ker\alpha_o$
on all slices $\{s\}\times\R\times V$.

\begin{rem}
\label{rem:symplecticrem}
 The form $\rmd(s\alpha_o)=\rmd s\wedge\alpha_o+s\rmd\alpha_o$
 is symplectic on $\{s>0\}$ and compatible with $J$.
 Therefore,
 the function $(s,t,z,u)\mapsto\tfrac12s^2$
 is strictly plurisubharmonic on $\{s>0\}$
 because $\alpha_o=-\rmd s\circ J$.
\end{rem}

By \cite[Proposition 5]{nie06} the {\bf Niederkr\"uger map}
\[
\Phi(s,t,z,u)=\big(s-\psi(z,u)+\rmi t,z,u\big)
\]
is a biholomorphic map
\[
\Phi\co
(\R\times\R\times\C\times T^*B,J)
\lra
(\C^2\times T^*B,\rmi\oplus J_{T^*B})
\;,
\]
which maps the hypersurface $\{s=0\}$ onto $\{s\circ\Phi^{-1}=0\}=\{x_1=-\psi(z,u)\}$.
As in \cite[Proposition 3.2]{nie06} we consider a $(n-1)$-dimensional family of holomorphic
discs
\[
\{-\varepsilon^2\}\times\D_{2\varepsilon}\times\{b\}
\]
in $(\C^2\times T^*B,\rmi\oplus J_{T^*B})$ with parameters $\varepsilon\in\R^+$, $b\in B$.
Here, we denote by $\D_r\subset\C$ the closed disc with radius $r$ and centre $0$.
Writing $\D$ for the closed unit disc $\D_1$ the disc family can be parametrised by
\[
v_{\varepsilon,b}(z)=
\big(
-\varepsilon^2,2\varepsilon\cdot z,b
\big)
\]
for $z\in\D$.
The lifts $u_{\varepsilon,b}=\Phi^{-1}\circ v_{\varepsilon,b}$ via the Niederkr\"uger map
are holomorphic maps
\[
(\D,\partial\D)\lra
\Big(
  \R\times\R\times\C\times T^*B,
  \{0\}\times\{0\}\times\C^*\times B
\Big)
\]
given by
\[
u_{\varepsilon,b}(z)=
\Big(
\varepsilon^2\big(|z|^2-1\big),0,2\varepsilon\cdot z,b
\Big)
\;.
\]
We will refer to the $u_{\varepsilon,b}$ as {\bf local Bishop discs}.

From \cite[Proposition 6]{nie06} we get {\bf local uniqueness}:

\begin{lem}
\label{lem:locunique}
 For all simple $J$-holomorphic disc maps
 $u\co\D\ra\R\times\R\times\C\times T^*B$
 such that $u(\partial\D)\subset\{0\}\times\{0\}\times\C^*\times B$
 there exist $\varepsilon\in\R^+$, $b\in B$ and
 a M\"obius transformation $\varphi\co(\D,\partial\D)\ra(\D,\partial\D)$
 such that $u=u_{\varepsilon,b}\circ\varphi$.
\end{lem}

\begin{proof}
  Consider the Niederkr\"uger transform $v=\Phi(u)$ of $u$.
  The projection to the $T^*B$-factor is a $J_{T^*B}$-holomorphic disc with boundary on the zero
  section $B$ and is therefore constant as by Stokes theorem the symplectic energy vanishes.
  So we are left with a holomorphic disc $v=(f+\rmi g,v_2)$ in $\C^2$
  such that the restriction to $\partial\D$ satisfies
  $f+\tfrac14|v_2|^2=0$ and $g=0$,
  because
    \[
    \Phi\big(\{0\}\times\{0\}\times\C^*\times B\big)
    =\big\{x_1=-\tfrac14|z|^2,y_1=0,u=0\}
    \;.
    \]
  The maximum and minimum principle
  implies that the harmonic function $g$ vanishes identically
  so that $f$ must be constant
  according to the classical Cauchy--Riemann equations.
  Write $f=-\varepsilon^2$ for some $\varepsilon>0$.
  Observe, that $\varepsilon$ indeed cannot vanish
  as a constant holomorphic disc is never simple.
  Hence, $\frac{1}{2\varepsilon}v_2$ is a holomorphic
  self-map of $(\D,\partial\D)$.
  Again by the simplicity assumption
  the degree of the restriction of $\frac{1}{2\varepsilon}v_2$
  to the boundary must be $1$.
  The argument principle implies
  that $\frac{1}{2\varepsilon}v_2$ is an automorphism which is given by a M\"obius transformation.
\end{proof}

We remark that the boundary circles
\[
u_{\varepsilon,b}(\partial\D)=
\{0\}\times\{0\}\times\partial\D_{2\varepsilon}\times\{b\}
\]
of the local Bishop discs
foliate $\{0\}\times\{0\}\times\C^*\times B$.
Recall that a neighbourhood of
$\{0\}\times\{0\}\times\{0\}\times B$
corresponds to a neighbourhood of the binding $B$ in $N^*$.


\subsection{Pseudo-convexity}
\label{subsec:psicon}

We consider a $2n$-dimensional
almost complex manifold $(W,J)$
with non-empty boundary.
Denote by $M$ a boundary component $M\subset\partial W$
and by $\xi=TM\cap JTM$ the $J$-invariant hyperplane distribution
along $M$.
Assume that there exists a smooth $1$-form $\alpha$
on $M$ such that $\xi=\ker\alpha$ and that $\rmd\alpha$
is $J$-{\bf positive} on complex lines in $\xi$
in the sense that
$\rmd\alpha(v,Jv)>0$ for all $v\in\xi$, $v\neq0$.
In other words,
$(M,\xi)$ is a $J$-{\bf convex hypersurface} of $(W,J)$
as defined in \cite{eli90}.

In this situation,
$\xi$ is a contact structure with contact form $\alpha$,
cf. Remark \ref{positivoncpxlinesimpliessympl}.
Denoting by $R$ the Reeb vector field of $\alpha$
we additionally assume that $-JR$ is outward pointing.
In other words,
$(M,\xi)$ is a $J$-{\bf convex boundary component} of $(W,J)$.
Observe, that $(W,J)$ is naturally oriented by the $n$-th power of any
$J$-positive (and hence non-degenerate) $2$-form on $W$.
Therefore,
the contact orientation
$\alpha\wedge(\rmd\alpha)^{n-1}$ on $M$
and the boundary orientation
on $M\subset\partial W$ coincide.

\begin{exwith}
\label{ex:tobeshifted}
The almost complex manifold
$\big((0,1]\times\R\times\C\times T^*B,J\big)$
as constructed in Section \ref{subsec:agermofbishopdisfill}
has $J$-convex boundary with strictly plurisubharmonic function $\tfrac12s^2$,
see Remark \ref{rem:symplecticrem}.
\end{exwith}

In fact, if $M$ is compact, then $M\subset\partial W$ is the regular zero-level set
of a strictly plurisubharmonic function $\varrho$ defined near $M$ that is negative
on the complement of $M$, see \cite[Lemma 2.7]{ce12}.
The precomposition $\varrho\circ u$ with a holomorphic map
$u\co G\ra(W,J)$, $G\subset\C$ open domain,
is subharmonic where defined, and therefore satisfies the strong
{\bf maximum principle}.
So, for example,
non-constant holomorphic spheres in $(W,J)$
are uniformly bounded away from
the boundary component $M\subset\partial W$,
which we assumed to be compact.
Furthermore,
in the case of a non-constant $J$-holomorphic disc
$u\co(\D,\partial\D)\ra(W,M)$,
we get $u(\Int\D)\subset W\setminus M$ and
the radial derivative
\[
0<
\rmd\big(\varrho\circ u\big)(\partial_r)
=
-\big(\rmd\varrho\circ J\big)
\Big(Tu(\partial_{\theta})\Big)
\]
is positive along $\partial\D$
by the {\bf boundary lemma of E.\ Hopf}.
Because
\[
\xi=
\ker(\rmd\varrho)\cap
\ker\big(\!-\rmd\varrho\circ J\big)
\]
is co-oriented by the Reeb vector field
of any contact form defining $\xi$
as a co-oriented hyperplane distribution
this means that the curve $u(\partial\D)\subset M$
is an immersion {\bf positively transverse} to $\xi$.
In particular, any such holomorphic disc $u$
such that $u|_{\partial\D}$ is an embedding
must be simple,
see \cite[Lemma 4.5]{gz10}.


\subsection{Uniqueness of the germ near the binding}
\label{subsec:uniquegermnerrbind}

Consider a $2n$-dimensional almost complex manifold $(W,J)$
that has a compact $J$-convex boundary component $(M,\xi)$
as described in Section \ref{subsec:psicon}.
Assume that $(M,\xi)$ contains a submanifold $N$
supporting a bordered Legendrian open book $(B,\vartheta)$.
By Section \ref{subsec:locmodnearthebind}
a neighbourhood of the binding
$B\subset W$ is diffeomorphic to a neighbourhood
of $\{0\}\times\{0\}\times\{0\}\times B$
in $(-\infty,0]\times\R\times\C\times T^*B$
such that the restriction to the boundaries
$M$ and $\{0\}\times\R\times\C\times T^*B$
induces a contactomorphism.
In addition,
assume that the almost complex structure $J$ of $W$
corresponds to the one constructed in Section
\ref{subsec:agermofbishopdisfill} under the diffeomorphism.

From \cite[Proposition 7]{nie06}
we get {\bf semi-global uniqueness}:

\begin{lem}
\label{lem:semiglobunique}
 There exists a neighbourhood $U_B\subset W$ of $B$
 such that for all simple $J$-holomorphic disc maps
 $u\co(\D,\partial\D)\ra(W,N^*)$
 with $u(\D)\cap U_B\neq\emptyset$
 we have that $u(\D)$ is contained in $U_B$ and
 there exist $\varepsilon\in\R^+$, $b\in B$ and
 a M\"obius transformation $\varphi\co(\D,\partial\D)\ra(\D,\partial\D)$
 such that $u=u_{\varepsilon,b}\circ\varphi$.
\end{lem}

\begin{proof}
 Using the above diffeomorphism
 we describe such a neighbourhood $U_B$
 as subset of $(-\infty,0]\times\R\times\C\times T^*B$:
 For $x_1<0$, $y_1\in\R$ consider
 the complex hypersurfaces
 $\{x_1+\rmi y_1\}\times\C\times T^*B$
 in $\big(\C^2\times T^*B,\rmi\oplus J_{T^*B}\big)$.
 The intersection with the real hypersurface
 $\{x_1=-\psi(z,u)\}$ is the sphere bundle
 in $\underline{\C}\oplus T^*B$
 given by $|x_1|=\tfrac14|z|^2+\frac12g^{\flat}(u,u)$;
 the intersection with $\{x_1\leq-\psi(z,u)\}$,
 therefore, is the corresponding disc bundle in 
 $\underline{\C}\oplus T^*B$.
 Here $\underline{\C}$ denotes the trivial
 complex line bundle over $B$. 
 Hence, the complex hypersurfaces
   \[
   H_{x_1,y_1}:=
   \Phi^{-1}\Big(\{x_1+\rmi y_1\}\times\C\times T^*B\Big)
   \cap\{s\leq0\}
   \]
 foliate the complement of
 $\{0\}\times\R\times\{0\}\times B$
 in $\big((-\infty,0]\times\R\times\C\times T^*B,J\big)$
 including a foliation by their real boundaries.
 Then, by definition,
 $U_B$ corresponds to
 \[
 U_B:=
 \Big(\{0\}\times(-\delta,\delta)\times\{0\}\times B\Big)
 \cup\bigcup_{|x_1|,|y_1|<\delta}H_{x_1,y_1}
 \]
 for $\delta>0$ sufficiently small,
 under the above mentioned identifying diffeomorphism.
 
 Consider a simple $J$-holomorphic disc
 $u\co(\D,\partial\D)\ra(W,N^*)$
 and suppose that $G=u^{-1}(U_B)$ is not empty.
 Observe that $G\subset\D$ is open.
 Restricting to $G$ we write
 $\Phi(u)=(f+\rmi g, v_2,v_3)$.
 If $g$ is constant in a neighbourhood of a point in $G$,
 then so is the holomorphic function $f+\rmi g$.
 With the identity theorem
 this implies that $f+\rmi g$ is constant on $G\setminus\partial\D$
 and, hence, on $G$.
 Denoting the constant by $f_o+\rmi g_o$
 this translates into $G=u^{-1}(H_{f_o,g_o})$
 so that $G\subset\D$ is closed also. 
 Hence, $G=\D$ by an open--closed argument,
 i.e. $u(\D)\subset U_B$ and
 the claim follows with Lemma \ref{lem:locunique}.
 
 It remains to show
 that the complementary case,
 namely that $\{\rmd g=0\}$ has no interior points,
 cannot occur.
 Indeed, otherwise
 the holomorphic disc $u(\D)$ and
 the complex hyperplane $H_{x_1^*,y_1^*}$ intersect
 along finitely many points
 for $x_1^*<0$, $y_1^*\neq0$.
 For that observe with Section \ref{subsec:psicon}
 and Example \ref{ex:tobeshifted} (suitably shifted in the $s$-direction)
 that the intersection is along interior points of $u(\D)$.
 This follows because
 $H_{x_1^*,y_1^*}$ and the boundary condition $N^*$
 for the holomorphic disc $u$ are disjoint
 as
 \[
 \Phi\Big(\{0\}\times\{0\}\times\C^*\times B\Big)
 \subset
 \R_-\times\{y_1=0\}\times\C^*\times B
 \;.
 \]
 By positivity of local intersection numbers
 the total intersection number
 $u\bullet H_{x_1^*,y_1^*}$ is positive.
 On the other hand
 this total intersection number
 is equal to the homological intersection
 of $[u]\in H_2(W,M)$ and $[\partial c]\in H_{2n-2}W$,
 where the $(2n-1)$-chain $c$ is given by
 \[
 c=\bigcup_{x\in[x_1^*,0]}H_{x,y_1^*}
 \;.
 \]
 Indeed, the maximum principle implies
 that $u(\Int\D)$ does not intersect $M$
 so that $u(\D)$ and $\partial c\cap M$ are disjoint.
 Hence,
 \[
 u\bullet H_{x_1^*,y_1^*}=[u]\cdot [\partial c]=0
 \;.
 \]
 This is a contradiction.
 In other words,
 $\{\rmd g=0\}$ has to have an interior point.
\end{proof}


\subsection{Holomorphic model near the boundary}
\label{subsec:holommodelnearbound}

As in Section \ref{subsec:agermofbishopdisfill}
we define an almost complex structure on $M(\varphi^*)$:
Choose a Riemannian metric $g_{\theta}$ on $F$
that smoothly depends on $\theta\in\R$
such that $g_{2\pi}=g_{\theta}=\varphi^*g_{\theta-2\pi}$
for all $\theta\in(2\pi-\varepsilon,2\pi+\varepsilon)$
and $\varepsilon>0$ small.
For each $\theta\in\R$
we define the kinetic energy on $T^*F$ by
\[
k_{\theta}(u)=\frac12g_{\theta}^{\flat}(u,u)
\;,\quad u\in T^*F\;,
\]
denoting by $g_{\theta}^{\flat}$
the dual metric of $g_{\theta}$.
For each $\theta\in[0,2\pi]$ we
construct an almost complex structure $J_{\theta}$
on $T^*F$ as on \cite[Section 5.3]{kwz22}
that is compatible with $\rmd\lambda_{T^*F}$ and
turns $k_{\theta}$ into a strictly plurisubharmonic potential
on $T^*F$ in the sense of \cite[Section 3.1]{gz12}
meaning that
\[
\lambda_{T^*F}=-\rmd k_{\theta}\circ J_{\theta}\;.
\]
The almost complex structure $J_{\theta}$
is uniquely determined by $g_{\theta}$ and $\rmd\lambda_{T^*F}$;
therefore, $J_{\theta}$ depends smoothly on $\theta$
and satisfies $J_{2\pi}=J_{\theta}=\varphi^*J_{\theta-2\pi}$
for $\theta\in(2\pi-\varepsilon,2\pi+\varepsilon)$.

The metric on $[0,2\pi]\times F$
obtained by taking the sum of the Euclidean metric
and $g_{\theta}$ for each $\theta\in[0,2\pi]$
descents to a metric $g$ on the mapping torus $M(\varphi)$.
The induced kinetic energy function on $M(\varphi^*)$
is denoted by
\[
\psi(r,\theta,u)=\tfrac12r^2+k_{\theta}(u)
\;.
\]
The almost complex structure
$\rmi\oplus J_{\theta}$ on $T^*[0,2\pi]\times T^*F$
descents to a compatible almost complex structure
$J_g$ on $M(\varphi^*)$ and
\[
r\rmd\theta+\lambda_{T^*F}\equiv-\rmd\psi\circ J_g
\]
modulo second order terms in $|u|_{\theta}$, $u\in T^*F$,
in the sense of \cite[Appendix E.5]{mcsa04}.
Indeed, the derivative of $k_{\theta}(u)$ in $\theta$-direction
contributes a term that locally is a quadratic form in the
coordinates of $u\in T^*F$.
Consequently,
the restriction of $\rmd r\wedge\rmd\theta+\rmd\lambda_{T^*F}$
to $\{u=0\}$ in $M(\varphi^*)$ is equal to
$-\rmd\big(\rmd\psi\circ J_g\big)$.
As 
$\rmd r\wedge\rmd\theta+\rmd\lambda_{T^*F}$
is positive on $(\rmi\oplus J_{\theta})$-complex lines
and equals $-\rmd\big(\rmd\psi\circ J_g\big)$
modulo first order terms in $|u|_{\theta}$, $u\in T^*F$,
we conclude that $\psi$ is strictly plurisubharmonic
in a neighbourhood of $\{u=0\}$ in $M(\varphi^*)$.

We consider the almost complex manifold
$(\C\times M(\varphi^*),\rmi\oplus J_g)$
provided with the Liouville form
$s\rmd t+r\rmd\theta+\lambda_{T^*F}$,
where we denote the coordinates on $\C$ by $s+\rmi t$.
Observe that the almost complex structure
\[
J=\rmi\oplus J_g
\]
is compatible with
$\rmd s\wedge\rmd t+\rmd r\wedge\rmd\theta+\rmd\lambda_{T^*F}$
and that the function
\[
\Psi(s+\rmi t,r,\theta,u)=\tfrac12s^2+\psi(r,\theta,u)
\]
is strictly plurisubharmonic
in a neighbourhood of $\{u=0\}$ in $\C\times M(\varphi^*)$.
This holds because
\[
s\rmd t+r\rmd\theta+\lambda_{T^*F}\equiv-\rmd\Psi\circ J
\]
modulo second order terms in $|u|_{\theta}$, $u\in T^*F$.
Therefore, as above, the restriction of
$\rmd s\wedge\rmd t+\rmd r\wedge\rmd\theta+\rmd\lambda_{T^*F}$
to $\{u=0\}$ in $\C\times M(\varphi^*)$ is equal to
$-\rmd\big(\rmd\Psi\circ J\big)$.

By rescaling the metric $g_{\theta}$ on $F$ by a constant
if necessary
we can assume that $\Psi$ is strictly plurisubharmonic
in a neighbourhood of $\{\Psi=\tfrac12\}$.
The hypersurfaces
$\{\Psi=\tfrac12\}$ and $\{s=1\}$
are transverse to $s\partial_s+r\partial_r+Y_{T^*F}$,
which is a Liouville vector field w.r.t.\ the symplectic form
$\rmd s\wedge\rmd t+\rmd r\wedge\rmd\theta+\rmd\lambda_{T^*F}$.
Therefore,
the contraction into the symplectic form induces contact forms on both
hypersurfaces.
The induced contact form on 
$\{s=1\}=\{1\}\times\rmi\R\times M(\varphi^*)$
is $\alpha_{\varphi}$; the one on $\{\Psi=\tfrac12\}$ is given by
$-\rmd\Psi\circ J$ along $\{u=0\}$.

A reparametrisation of the flow of the Liouville vector field
$s\partial_s+r\partial_r+Y_{T^*F}$ yields a contact embedding
of $\{\Psi=\tfrac12,s>0\}$ onto $\{s=1\}$ w.r.t.\ the induced contact structures,
see \cite[Appendix A.1]{bschz19}.
The hypersurfaces $\{\Psi=\tfrac12\}$ and $\{s=1\}$
intersect along $\{s=1,r=0,u=0\}$, on which the flow is stationary.
Moreover, as we flow along the Liouville vector field
$s\partial_s+r\partial_r+Y_{T^*F}$
we observe that the multi level set
$\{\Psi=\tfrac12,s>0, t=0, r\leq0, u=0\}$
corresponds to $\{1\}\times\{0\}\times\{r\leq0, u=0\}$
in $\{1\}\times\rmi\R\times M(\varphi^*)$
under the contact embedding.
The latter set was used in Section \ref{subsec:locmodnearthebdary}
to describe the germ of contact structure near the boundary
of a bordered Legendrian open book;
the boundary being $\{\Psi=\tfrac12,s>0, t=0, r=0, u=0\}$,
which corresponds to $\{1\}\times\{0\}\times\{r=0, u=0\}$.

The hypersurface $\{\Psi=\tfrac12\}$ carries a second contact structure
given by
\[
\ker\big(\rmd\Psi\big)\cap\ker\big(\!-\rmd\Psi\circ J\big)
\,,
\]
which turns $\{\Psi=\tfrac12\}$ into a $J$-convex boundary of
$\{\Psi\leq\tfrac12\}$.
The induced singular characteristic foliation on 
$\{\Psi=\tfrac12,s>0, t=0, r\leq0, u=0\}$
coincides with the one described in the preceding paragraph,
where the contact distribution this time is taken w.r.t.\
$s\rmd t+r\rmd\theta+\lambda_{T^*F}\equiv-\rmd\Psi\circ J$
along $\{u=0\}$. 
By uniqueness of the germ of a contact structure
formulated in \cite[Lemma 4.6]{mnw13}, a neighbourhood of
$\{\Psi=\tfrac12,s>0, t=0, r\leq0, u=0\}$ is contactomorphic to
the contact structure we considered first.
In other words,
given a contact manifold $(M,\xi)$ containing a submanifold $N$
that supports a bordered Legendrian open book,
we obtain an alternative description
of the germ of contact structure near the boundary $\partial N$
presented in Section \ref{subsec:locmodnearthebdary}.


\subsection{Holomorphically blocking boundary}
\label{subsec:holomblockbound}

Let $(W,J)$ be a $2n$-dimensional almost complex manifold 
so that a given contact manifold $(M,\xi)$
is a compact $J$-convex boundary component of $W$,
see Section \ref{subsec:psicon}.
Let $N$ be a submanifold of $(M,\xi)$ 
that supports a bordered Legendrian open book $(B,\vartheta)$.
In view of Section \ref{subsec:holommodelnearbound} we assume
that a neighbourhood of the boundary $\partial N\subset W$
is diffeomorphic to a neighbourhood
of $\{1\}\times\{0\}\times\{r=0, u=0\}$
in $\{\Psi\leq\tfrac12, s>0\}$
such that $N\subset W$ and
$\{\Psi=\tfrac12,s>0, t=0, r\leq0, u=0\}$
correspond to each other diffeomorphically.
Furthermore
we assume that under the diffeomorphism
the almost complex structure $J$ of $W$
corresponds to $\rmi\oplus J_g$ on $\{\Psi\leq\tfrac12\}$
in $\C\times M(\varphi^*)$
inducing contactomorphisms on the boundary.

Similarly to \cite[Lemma 4.7]{mnw13}
we obtain the {\bf blocking lemma}:

\begin{lem}
\label{lem:blockinglemma}
 There exists a neighbourhood $U_{\partial N}\subset W$ of $\partial N$
 such that for all $J$-holomorphic disc maps
 $u\co(\D,\partial\D)\ra(W,N^*)$
 with $u(\D)\cap U_{\partial N}\neq\emptyset$
 are constant.
\end{lem}

\begin{proof}
 For $|s|<1$, $t\in\R$ consider
 the complex hypersurfaces
   \[
   H_{s,t}:=
   \Big(\{s+\rmi t\}\times M(\varphi^*)\Big)
   \cap\{\Psi\leq\tfrac12\}
   \]
 of $\{\Psi\leq\tfrac12\}$ in
 $(\C\times M(\varphi^*),\rmi\oplus J_g)$.
 Modulo the above identifying diffeomorphism
 $U_{\partial N}$ corresponds to
 \[
 U_{\partial N}:=
 \Big(\{1\}\times(-\delta,\delta)\times M(\varphi)\Big)
 \cup\bigcup_{1-s,|t|<\delta}H_{s,t}
 \]
 for $\delta>0$ sufficiently small.
 
 Let $u\co(\D,\partial\D)\ra(W,N^*)$ be a $J$-holomorphic disc
 such that the open set $G=u^{-1}(U_{\partial N})$ is not empty.
 Write the restriction of $u$ to $G$ as
 $u=(u_1=f+\rmi g, u_2)$ w.r.t.\ $(\C\times M(\varphi^*),\rmi\oplus J_g)$.
 An argument similar to the last paragraph
 of the proof of Lemma \ref{lem:semiglobunique}
 (that utilises $J$-convexity and
 positivity of intersections with the complex hyperplanes $H_{s,t}$)
 shows that $\{\rmd g=0\}$ has an interior point.
 As in the second paragraph
 of the proof of Lemma \ref{lem:semiglobunique}
 this shows that $u_1=f_o+\rmi g_o$ is constant
 and that $G=\D$, i.e. $u(\D)\subset U_{\partial N}$.
 
 The projection map $M(\varphi^*)\ra T^*S^1$
 sends $u_2$ to a smooth map $v\co\D\ra T^*S^1$
 such that $v(\partial\D)\subset\{r\leq0\}\simeq S^1$.
 In particular,
 the degree of $v|_{\partial\D}$ must be zero
 so that $v(\partial\D)$ is tangent to a fibre of $T^*S^1$.
 Therefore,
 $u(\partial\D)$ admits a point of tangency
 with a page of the Legendrian open book on $N$.
 In view of the maximum principle by E.\ Hopf,
 which implies the positive transversality property
 formulated in Section \ref{subsec:psicon},
 this implies that $u$ must be constant.
\end{proof}


\section{Tamed pseudo-convexity}
\label{sec:tamedpseidpconv}

Let $(W,J)$ be a $2n$-dimensional almost complex manifold
that admits a compact $J$-convex boundary component $(M,\xi)$,
see Section \ref{subsec:psicon}.
We assume that there exists a symplectic form $\Omega$ on $W$
that is $J$-positive on complex lines in $TW$,
i.e.\ $J$ is tamed by $\Omega$, cf.\ \cite{eli90}.
Define an odd-symplectic form on $M$ by setting
$\omega:=\Omega|_{TM}$.


\subsection{Magnetic symplectisation}
\label{subsec:magsympl}

Let $\alpha$ be a defining contact form for $\xi$ on $M$ such that
$\rmd\alpha$ is positive on complex lines in $(\xi,J)$.

\begin{rem}
 \label{positivoncpxlinesimpliessympl}
 As $\omega$ is positive on complex lines in $(\xi,J)$, it follows that for all
 non-zero $v\in\xi$ the contraction $\iota_v\omega$ does not vanish.
 In other words, $\omega$ restricts to a symplectic form on $\xi$.
 Choosing a symplectic basis for $(\xi,\omega)$ of the form
 $v_1,Jv_1,\ldots,v_{n-1},Jv_{n-1}$ yields
 that $\omega^{n-1}$ is a positive volume form on $(\xi,J)$.
 Furthermore the contact form $\alpha$ does not vanish
 on the characteristic foliation given by $\ker\omega$,
 i.e.\ $\xi$ and $\ker\omega$ intersect transversally.
 In fact, according to our orientation conventions,
 $\alpha\wedge\omega^{n-1}$ is a positive volume form on $M$.
 
 Observe, that the same reasoning applies to 
 all $2$-forms that are positive on complex lines in $(\xi,J)$. 
 This shows positivity of
 $\alpha\wedge(\rmd\alpha)^{n-1}$ which was used implicitly
 in Section \ref{subsec:psicon}.
\end{rem}

The symplectic neighbourhood theorem for hypersurfaces implies
that there exist $\varepsilon>0$ and a symplectic embedding
\[
\Big((-\varepsilon,0]\times M,\rmd(s\alpha)+\omega\Big)
\lra(W,\Omega)
\]
whose restriction to $\{0\}\times M$ is the inclusion of $M=\partial W$
into $W$,
cf.\ \cite[Exercise 3.36]{mcsa98} and \cite[Remark 2.7]{mnw13}.
Observe that
\[
\Big(\rmd(s\alpha)+\omega\Big)^n=
\Big(\rmd s\wedge\alpha+s\rmd\alpha+\omega\Big)^n=
n\rmd s\wedge\alpha\wedge\big(s\rmd\alpha+\omega\big)^{n-1}
\,.
\]
Therefore, we find $\varepsilon>0$ and a large positive constant $s_o$
such that, considered on $(-\varepsilon,\infty)\times M$,
$\rmd(s\alpha)+\omega$ is symplectic on $(-\varepsilon,\varepsilon)\times M$
and $(s_o,\infty)\times M$.
In fact, using Remark \ref{positivoncpxlinesimpliessympl},
$\alpha\wedge\big(s\rmd\alpha+\omega\big)^{n-1}$
is a positive volume form on $M$ for all positive $s$
because $\rmd\alpha$ and $\omega$ are positive on complex lines in $(\xi,J)$.
Therefore,
$\rmd(s\alpha)+\omega$ is a symplectic form on $(-\varepsilon,\infty)\times M$.
Via gluing along $\{0\}\times M\equiv M\subset\partial W$
using the above symplectic embedding
we build a symplectic manifold,
the so-called {\bf magnetic completion},
\[
\big(\widetilde{W},\widetilde{\Omega}\big):=
(W,\Omega)\cup
\Big([0,\infty)\times M,\rmd(s\alpha)+\omega\Big)
\,.
\]


\subsection{Truncating the magnetic completion}
\label{subsec:truncatiingmagcompl}

We continue the considerations in Section \ref{subsec:magsympl}.
In addition, let $N\subset(M,\xi)$ be a submanifold that carries the structure
of a bordered Legendrian open book decomposition $(B,\vartheta)$.
Choose contact embeddings of the model neighbourhood
of the binding $B$ (see Section \ref{subsec:locmodnearthebind})
and of the alternative model neighbourhood of the boundary $\partial N$
(see Section \ref{subsec:holommodelnearbound}) into $(M,\xi=\ker\alpha)$.
The push forward of the respective contact forms $\alpha_o$ and
the restriction of $-\rmd\Psi\circ J$ to the tangent spaces of $\{\Psi=\tfrac12\}$
are equal to $\rme^h\alpha$ for a smooth function $h$
on the images of the model neighbourhoods. 
Alter the contact form $\alpha$ by cutting off $h$ to $0$
near the boundary of these images,
so that the considered contact embeddings -- 
after shrinking the model neighbourhoods a bit -- 
are in fact strict.
Observe that this conformal change of the contact form $\alpha$
does not effect the property of $\rmd\alpha$
to be positive on complex lines in $(\xi,J)$.

Let $s_o$ be a positive real number and consider the truncation
\[
\big(\hat{W},\hat{\Omega}\big):=
(W,\Omega)\cup
\Big([0,2s_o]\times M,\rmd(s\alpha)+\omega\Big)
\;.
\]
The resulting contact embeddings of the model neighbourhoods
into $\{2s_o\}\times M$,
which are strict up to conformal factor $2s_o$
w.r.t.\ the contact form $2s_o\alpha$,
extend along collar directions to embeddings
of the neighbourhood $U_B$ constructed in
Lemma \ref{lem:semiglobunique} and of the neighbourhood
$U_{\partial N}$ constructed in Lemma \ref{lem:blockinglemma}
into $(-\infty,2s_o]\times M$ in the following way:
For the embedding of $U_B$ simply cross with the identity in
the $s$-direction;
for the embedding of $U_{\partial N}$ denote by $Y$ the vector field
on $\{\Psi\leq\tfrac12\}$ obtained by multiplying the Reeb vector field
of the contact form $-\rmd\Psi\circ J$ on the level sets of $\Psi$ with $-J$
and follow the flow lines of $Y$ in backward time
(taking time logarithmically).
Observe that $Y$ coincides with the Liouville vector field
$s\partial_s+r\partial_r+Y_{T^*F}$ along $\{u=0\}$,
see Section \ref{subsec:holommodelnearbound}.

The images of the neighbourhoods are again denoted by $U_B$ and
$U_{\partial N}$.
Taking $s_o$ sufficiently large we achieve that $U_B\cup U_{\partial N}$
fit into $[s_o,2s_o]\times M$.
Push forward yields an almost complex structure $J_o$ on
$U_B\cup U_{\partial N}$ that allows the conclusions
of Lemmata \ref{lem:semiglobunique} and \ref{lem:blockinglemma}.
We remark that
$J_o$ is invariant under translations in $s$-direction that shift off $U_B$;
along $[s_o,2s_o]\times\partial N$ the almost complex structure $J_o$ is
independent of $s$.

Moreover, the symplectic form $\rmd(s\alpha)$ is compatible
with $J_o$ on $U_B$ and on $U_{\partial N}$,
see Remark \ref{rem:symplecticrem} and Section \ref{subsec:holommodelnearbound},
resp.
Furthermore $J_o$ sends $\partial_s$ to the Reeb vector field and
preserves the contact distribution.
Applying \cite[Proposition 2.63(i)]{mcsa98} to the symplectic bundle
\[
\big(\xi,\rmd(s\alpha)\big)\lra[s_o,2s_o]\times M
\]
we extend $J_o$ to an almost complex structure on
$[s_o,2s_o]\times M$ keeping the properties just listed.
In particular, $\{2s_o\}\times M$ is a $J_o$-convex boundary
independently of the conformal factor $2s_o$,
see Section \ref{subsec:psicon}.

Placing the whole scenario to $[s_o,2s_o]\times M$ for $s_o$
sufficiently large we can additionally assume that $J_o$ is also tamed by
the symplectic form $\rmd(s\alpha)+\omega$ on $[s_o,2s_o]\times M$.
This is possible because for $s_o$ sufficiently large
$\rmd(s\alpha)=\rmd s\wedge\alpha+s\rmd\alpha$
dominates
$\rmd(s\alpha)+\omega=\rmd s\wedge\alpha+(s\rmd\alpha+\omega)$
on $J_o$-complex lines.
With \cite[Proposition 2.51]{mcsa98}
we extend $J_o$ to a tamed almost complex structure
$\hat{J}$ on $\big(\hat{W},\hat{\Omega}\big)$ that restricts to $J$ on $W$
and to $J_o$ on $[s_o,2s_o]\times M$.
Moreover, $\{2s_o\}\times M\subset\partial\hat{W}$ is a $\hat{J}$-convex
boundary component.
We will refer to the construction of $\big(\hat{W},\hat{\Omega},\hat{J}\big)$
as {\bf magnetic collar extension}.


\subsection{Deforming the Truncation}
\label{subsec:deformthetrunc}

Assuming exactness of $\omega|_{TN}$ in the situation of Sections
\ref{subsec:magsympl} and \ref{subsec:truncatiingmagcompl} the
symplectic form in the magnetic collar extension
$\big(\hat{W},\hat{\Omega},\hat{J}\big)$ can be assumed to satisfy
\[
\hat{\Omega}=2s_o\rmd\alpha
\qquad\text{on}\quad
T\big(\{2s_o\}\times N\big)
\]
after deformation:

Write $\omega=\rmd\beta$ in a neighbourhood of $N$
taking a neighbourhood that strongly deformation retracts to $N$. 
Extend $\beta$ to a $1$-form on $M$ that vanishes outside
a larger neighbourhood.
Define a $2$-form $\eta:=\omega-\rmd(\varrho\beta)$
on $[s_o,2s_o]\times M$, where $\varrho$ is a smooth function
that vanishes on $\{s\leq s_o\}$, equals $1$ on $\{s\geq2s_o\}$,
and satisfies $0\leq\varrho'\leq2/s_o$.
We claim that for $s_o$ sufficiently large
\[
\Big([s_o,2s_o]\times M,\rmd(s\alpha)+\eta\Big)
\]
is symplectic.
Indeed, spelling out the $n$-th power of $\rmd(s\alpha)+\eta$ we find
\[
n\rmd s\wedge\big(\alpha-\varrho'\beta\big)
\wedge\Big(s\rmd\alpha+\omega-\varrho\rmd\beta\Big)^{n-1}
\,.
\]
For $s_o$ sufficiently large $\alpha-\varrho'\beta$ will be a contact form
on all slices $\{s\}\times M$, $s\in[s_o,2s_o]$, so that, restricted to the
related contact structures, $(s\rmd\alpha)^{n-1}$ is a positive volume form.
Making $s_o$ even larger $(s\rmd\alpha)^{n-1}$ dominates lower order terms
in the $(n-1)$-st power of $s\rmd\alpha+\omega-\varrho\rmd\beta$
restricted to the mentioned contact structures.

Starting with the deformed symplectic structure corresponding
to $\rmd(s\alpha)+\eta$ the tamed complex structure $\hat{J}$
can be constructed as in Section \ref{subsec:truncatiingmagcompl}.
In order to achieve that $\rmd(s\alpha)$ dominates
\[
\rmd(s\alpha)+\eta=
\rmd s\wedge\big(\alpha-\varrho'\beta\big)
+\Big(s\rmd\alpha+\omega-\varrho\rmd\beta\Big)
\]
on $J_o$-complex lines on $[s_o,2s_o]\times M$
simply choose $s_o$ sufficiently large.

We will refer to the construction of
$\big(\hat{W},\hat{\Omega},2s_o\alpha,\hat{J}\big)$
as {\bf deformed magnetic collar extension}.


\subsection{Gromov compactness}
\label{subsec:gromovcomp}

We consider a deformed magnetic collar extension of the $\Omega$-tamed
almost complex structure $(W,J)$ as in Section \ref{subsec:deformthetrunc}.
The resulting tamed almost complex manifold together with the choice
of contact form on the $J$-convex boundary component is denoted by
$(W,\Omega,\alpha,J)$.
In particular, the restriction of the $2$-forms
$\Omega$ and $\rmd\alpha$ to the tangent spaces
of the bordered Legendrian open book
$(N,B,\vartheta)\subset(M,\xi=\ker\alpha)$
are equal as $2$-form on $N$.

Notice, that the regular set $N^*$ of $N$
(see Section \ref{subsec:legopenbooks})
is a {\bf totally real} submanifold of $(W,J)$,
i.e.\ $TN^*\cap JTN^*$ is the zero section.
Indeed, denoting by $E$ the real linear span of $v,Jv$
for a given tangent vector $v\in T_pN^*$
the only possibility for the complex line $E$ to be tangent to $N^*\subset M$
is to be tangent to the page of $(B,\vartheta)$ through $p$,
which is Legendrian w.r.t.\ $\xi=TM\cap JTM$.
Positivity of $\rmd\alpha$ on complex lines in $(\xi,J)$
implies the vanishing of $v$ because the pages of $(B,\vartheta)$
are Legendrian submanifolds, and, hence, have Lagrangian tangent spaces
inside $(\xi,\rmd\alpha)$.

Via the neighbourhood $U_B\subset W$ of the binding $B\subset N$
constructed in Lemma \ref{lem:semiglobunique} we obtain an embedding
relative boundary, a so-called {\bf local Bishop filling},
\[
F\co
\Big(
(0,\delta)\times\D\times B,
(0,\delta)\times\partial\D\times B
\Big)\lra(W,N^*)
\]
for some $\delta>0$ such that for all $(\varepsilon,b)\in(0,\delta)\times B$
the maps $u_{\varepsilon,b}=F(\varepsilon,\,.\,,b)$ are $J$-holomorphic discs
$(\D,\partial\D)\ra(W,N^*)$.
Furthermore $F$ extends smoothly to a map defined on
$[0,\delta)\times\D\times B$ that maps $\{0\}\times\D\times B$ to $B$ via
$(\varepsilon,z,b)\mapsto b$.

We consider the {\bf moduli space} $\MM$ of $J$-holomorphic discs
$u\co(\D,\partial\D)\ra(W,N^*)$ that are homologous in $W$ relative $N^*$
to one of the Bishop discs $u_{\varepsilon,b}$.
For all $u\in\MM$ we have the following:

\begin{enumerate}
\item The degree of the map $\vartheta\circ u\co\partial\D\ra S^1$,
the so-called {\bf winding number} of $u$, is equal to $1$.
The boundary lemma of E.\ Hopf implies,
that $u(\partial\D)$ is an embedded curve in $N^*$
positively transverse to $\xi$.
Hence, $u$ is simple, see Section \ref{subsec:psicon}.
\item We find a $2$-chain $C$ in $N$ with boundary $-u(\partial\D)$
such $u(\D)+C$ is null-homologous in $W$.
Therefore, by applying Stokes theorem twice, we get that the
{\bf symplectic energy} of $u$ satisfies
\[
0<\int_{\D}u^*\Omega=-\int_{C}\Omega=\int_{\partial\D}u^*\alpha\,,
\]
as $\Omega=\rmd\alpha$ on $TN$.
Denote by $\rmd\vartheta$ the pullback of $\rmd\theta$ along
$\vartheta\co N\setminus B\ra S^1$.
According to our co-orientation convention in
Section \ref{subsec:legopenbooks} and the local models in
Sections \ref{subsec:locmodnearthebind} and
\ref{subsec:locmodnearthebdary}
we observe that $\alpha|_{TN}=f\rmd\vartheta$ for a non-negative function
$f$ on $N$, that vanishes precisely along the singular set $B\cup\partial N$.
Consequently, we get {\bf uniform energy bounds}
\[
0<\int_{\D}u^*\Omega=
\int_{\partial\D}u^*f\cdot(\vartheta\circ u)^*\rmd\theta\leq
2\pi\max f\,.
\]
\item The restriction of $u$ to $\partial\D$ is {\bf uniformly bounded away
from the boundary} $\partial N$ as the intersection of $u(\D)$ with
$U_{\partial N}$ is empty by Lemma \ref{lem:blockinglemma}.
Recall that if $u(\D)$ intersects the neighbourhood $U_B$ of the binding
then $u$ is a reparametrisation of a local Bishop disc $u_{\varepsilon,b}$,
see Lemma \ref{lem:semiglobunique}.
We {\bf truncate} the moduli space $\MM$ (keeping the notation) by removing
all holomorphic discs that are reparametrisations of $u_{\varepsilon,b}$ for
$(\varepsilon,b)\in(0,\delta/2)\times B$.
\end{enumerate}

\begin{rem}
 \label{rem:gromovcomp}
 Under the assumption of {\bf uniform} $C^0${\bf -bounds} for $\MM$, i.e.\
 that there exists a compact subset of $W$ that contains all holomorphic discs
 $u(\D)$, $u\in\MM$, we obtain:
 Any sequence of holomorphic discs in $\MM$ admits a Gromov converging
 subsequence, see \cite{fz15}.
 Observe that by $J$-convexity (see Section \ref{subsec:psicon}) no sequence of
 holomorphic discs $u(\D)$, $u\in\MM$, can escape the boundary component $M$ of $W$.

 The total winding number of the Gromov limit $\bfu$ must be $1$
 according to the properties of Gromov convergence, see \cite{fz15}.
 Moreover, all individual winding numbers of non-constant disc bubbles of $\bfu$
 are positive by positive transversality, see the boundary lemma to E.\ Hopf
 in Section \ref{subsec:psicon}.
 Therefore, $\bfu$ contains precisely one disc component
 $u_0$ that is necessarily simple and of winding number $1$.
 In particular, $u_0|_{\partial\D}$ is an embedding positively transverse to $\xi$.
 
 If the image of $\bfu$ intersects $U_B$, then $\bfu$ does it along
 the disc component $u_0$ by the maximum principle,
 see Section \ref{subsec:psicon}.
 This implies that $u_0$ is one of the Bishop discs $u_{\varepsilon,b}$
 by Lemma \ref{lem:semiglobunique}, so that there are in fact no bubbles
 in this situation.
 Indeed, potential sphere bubbles must be null-homologous or subject to
 an argument using the maximum principle.
 Therefore, any sequence of holomorphic discs in $\MM$ that
 Gromov converges to $\bfu$ with $u_0(\D)$ intersecting $U_B$
 converges in $C^{\infty}$ to $u_0$ up to reparametrisation.
 
 Consequently, $\MM$ can be compactified to $\overline{\MM}$ by adding all
 Gromov limits to $\MM$.
 The resulting moduli space $\overline{\MM}$ is compact in the sense that any
 sequence in $\overline{\MM}$ admits a Gromov converging subsequence.
 The resulting limit objects $\bfu$ share the properties mentioned in the proceeding
 remark.
 In particular, we obtain {\bf uniform gradient bounds near the boundary} w.r.t.\ a
 given background metric:
 There exists constants $\rho\in(0,1)$ and $C>0$ such that for all $\bfu$ in
 $\overline{\MM}$ we have that $|\nabla u_0|<C$ restricted to $\D\cap\{|z|\geq\rho\}$,
 where $u_0$ is the disc component of $\bfu$ parametrised such that
 $\vartheta\circ u_0(\rmi^k)=\rmi^k$ for $k=0,1,2$.
\end{rem}


\section{Symplectic cobordisms}
\label{sec:symplcobord}

A {\bf symplectic cobordism} is a compact $2n$-dimensional symplectic manifold
$(W,\Omega)$ with boundary $M:=\partial W$.
We assume that $(W,\Omega)$ is {\it connected} and oriented via $\Omega^n$. 
The {\bf odd-symplectic} form $\omega:=\Omega|_{TM}$ is closed and
maximally non-degenerate with $1$-dimensional kernel $\ker\omega$.
The line bundle $\ker\omega$ on $\partial W$ is trivialised by the restriction
of the Hamiltonian vector field of a collar neighbourhood parameter to
$M=\partial W\subset W$.
Therefore, one finds a $1$-form $\alpha$ on $M$ that does not vanish on $\ker\omega$.
The sign of $\alpha$ defines an orientation on each component of
$\partial W$ via the volume form $\alpha\wedge\omega^{n-1}$.
If the boundary orientation (induced by the outward pointing normal)
coincides with the one given by $\alpha$ we call the
components of $M=\partial W$ {\bf positive}, otherwise {\bf negative}.
We will write
\[
\partial(W,\Omega)=
-(M_-,\omega_-,\alpha_-)
+(M_+,\omega_+,\alpha_+)
\]
accordingly.
Due to the symplectic neighbourhood theorem there exist collar neighbourhoods
$[0,\varepsilon)\times M_-$, resp., $(-\varepsilon,0]\times M_+$,
such that the symplectic form $\Omega$ can be written as
$\rmd(s\alpha_{\mp})+\omega_{\mp}$
with $s\in[0,\varepsilon)$, resp., $s\in(-\varepsilon,0]$.
This allows gluing of symplectic cobordisms along boundary components
of opposite sign that are orientation reversing odd-symplectomorphic.
The most prominent examples of symplectic cobordisms have {\bf contact type} boundary,
i.e.\ the $1$-form $\alpha$ on $M$ can be chosen such that $\omega=\rmd\alpha$,
see \cite{mcd91}.
In the contact type context positive boundary components are called {\bf convex};
negative ones {\bf concave}.


\subsection{A directed symplectic cobordism}
\label{subsec:adirsymplcobor}

We consider a symplectic cobordism $(W,\Omega)$.
We assume that the boundary $\partial(W,\Omega)$
decomposes into concave boundary components,
whose union we denote by $(M_-,\omega_-,\alpha_-)$,
and positive boundary components,
whose union we denote by $(M_+,\omega_+,\alpha_+)$.
In addition, we require the {\bf weak-filling condition}:
$\alpha_+$ can be chosen to be a contact form,
which according to our conventions implies that
$\alpha_+\wedge(\rmd\alpha_+)^{n-1}$ and
$\alpha_+\wedge\omega_+^{n-1}$ are positive,
such that
\[
\alpha_+\wedge\big(f_+\rmd\alpha_++\omega_+\big)^{n-1}
>0
\]
for all non-negative smooth functions $f_+$ on $M_+$.
We call $(W,\Omega)$ a {\bf directed symplectic cobordism},
cf.\ \cite{egh00}.

\begin{rem}
\label{rem:sympcobvsfilling}
 Observe that if $M_-$ is empty, $(W,\Omega)$ will be a
 weak symplectic filling of the contact manifold $(M_+,\xi_+)$
 with contact structure $\xi_+:=\ker\alpha_+$.
 If additionally $(M_+,\xi_+)$ is the convex boundary
 of the symplectic cobordism $(W,\Omega)$ with $M_-=\emptyset$
 then $(W,\Omega)$ is a {\bf strong symplectic filling} of $(M_+,\xi_+)$.
\end{rem}

As in \cite{gz12,sz17} such symplectic cobordisms $(W,\Omega)$ can be used to verify
the strong Weinstein conjecture for the contact manifold that appears
as the concave boundary component of $(W,\Omega)$.
The $3$-dimensional variant is related to the ball theorem,
see \cite[Theorem 2.2 and Corollary 3.8]{gz13}:

\begin{thm}
 \label{thm:maindirectsyplcoborthm}
 Let $(W,\Omega)$ be a directed symplectic cobordism. 
 Assume that the contact manifold $(M_+,\xi_+)$ contains a submanifold $N$
 with non-empty boundary such that $N$ supports a bordered Legendrian open book
 $(B,\vartheta)$ and such that $\omega_+|_{TN}$ is exact.
 Furthermore assume that one of the following conditions is satisfied:
 \begin{enumerate}
 \item [(i)]
  The symplectic cobordism $(W,\Omega)$ is semi-positive.
 \item [(ii)]
 The second Stiefel--Whitney class $w_2(TN^*)$ of $N^*=N\setminus(B\cup\partial N)$
 vanishes; or $N^*$ is orientable and there exists a class in $H^2(W;\Z_2)$ that restricts to
 $w_2(TN^*)$.
 \end{enumerate}
 Then $M_-$ necessarily is non-empty and for any $\xi_-$-defining
 contact form there exists a null-homologous Reeb link.
\end{thm}

Theorem \ref{thm:maindirectsyplcoborthm} part (i) is contained in \cite{mnw13}.
We include a variant of the argument in Section \ref{subsec:thesemiposcase} as guideline for
the polyfold proof of part (ii).
Following \cite{mcsa04} we call a $2n$-dimensional symplectic manifold $(W,\Omega)$
{\bf semi-positive} if the first Chern class $c_1$ of $(W,\Omega)$ satisfies the following
condition: $c_1(A)\geq0$ for all spherical homology $2$-classes $A$ of $W$ with
$c_1(A)\geq3-n$ and $\Omega(A)>0$.
If $2n\leq6$ the condition is automatic.
The proof of part (ii) of Theorem \ref{thm:maindirectsyplcoborthm}
is postponed to Section \ref{sec:polyfoldperturbations}.

We remark that the inclusion $N^*\subset W$ induces in cohomology
the restriction map $H^2(W;\Z_2)\ra H^2(N^*;\Z_2)$.
The long exact cohomology sequence with $\Z_2$-coefficients yields,
that the restriction map is surjective if and only if the co-boundary operator
$H^2(N^*;\Z_2)\ra H^3(W,N^*;\Z_2)$ vanishes.
Furthermore the choice of an orientation on $N^*$
induces an orientation on each page $\vartheta^{-1}(\theta)$, $\theta\in S^1$,
via the co-orientation induced by $\vartheta$.
The binding is oriented via the boundary orientation of a page.
Therefore, $N$ is orientable if and only if $N^*$ is.

Observe that the inclusion map $f\co N^*\subset N$
induces $TN^*=f^*TN$, so that naturality of
the Stiefel--Whitney classes implies
$w_2(TN^*)=f^*w_2(TN)$.
In particular, if $w_2(TN)$ vanishes so does $w_2(TN^*)$.
Further, if the fibration $\vartheta$ on $N^*$ is trivial
(as it is the case for the plastikstufe),
e.g.\ a product of the page $\vartheta^{-1}(1)$ with $S^1$,
the second Stiefel--Whitney class $w_2(TN^*)$ is equal to
$w_2\big(T\vartheta^{-1}(1)\big)$
according to the Whitney cross product formula for the total Stiefel--Whitney class,
cf.\ \cite[Chapter 17]{bred93}.
In general, taking a connection on the fibration $\vartheta$
we obtain a splitting of $TN^*$ into the vertical $\ker T\vartheta$
and horizontal subbundle.
As the horizontal subbundle is isomorphic to the trivial bundle $\vartheta^*TS^1$
Whitney sum formula for the total Stiefel--Whitney class implies
$w_2(TN^*)=w_2\big(\!\ker T\vartheta\big)$.
The square product of the $2$-dimensional Klein bottle shows
that the latter not always equals the second Stiefel--Whitney class of a fibre.

\begin{rem}
\label{rem:weakfillingexcludedbysmallblob}
Consider a symplectic cobordism $(W,\Omega)$ that satisfies all the requirements
of Theorem \ref{thm:maindirectsyplcoborthm}.
Notice that if $M_-=\emptyset$, then $(W,\Omega)$ would be a weak symplectic filling
of $(M_+,\xi_+)$.
Hence, by Theorem \ref{thm:maindirectsyplcoborthm} no such weak symplectic filling
can exist.
\end{rem}

\begin{rem}
\label{rem:exactnessalongblob}
The exactness requirement for $\omega_+|_{TN}$
in Theorem \ref{thm:maindirectsyplcoborthm}
is fulfilled e.g.\ if $(M_+,\xi_+)$ is the convex boundary of $(W,\Omega)$
or if the Legendrian open book $N$ is small.
\end{rem}


\subsection{Completing the cobordism}
\label{subsec:complthecob}

We consider the directed symplectic cobordism $(W,\Omega)$ from
Theorem \ref{thm:maindirectsyplcoborthm}.
According to the contact type boundary condition along the negative 
boundary components of $(W,\Omega)$ the symplectic neighbourhood
theorem allows a description of a collar neighbourhood of $M_-\subset W$
as $[0,\varepsilon)\times M_-$ such that the symplectic form equals 
$\Omega=\rmd(\rme^s\alpha_-)$ with $s\in[0,\varepsilon)$.
Therefore, a partial completion over the concave boundary of $(W,\Omega)$
can be given by
\[
\Big((-\infty,0]\times M_-,\rmd(\rme^s\alpha_-)\Big)
\cup
(W,\Omega)
\]
via gluing along $\{0\}\times M_-\equiv M_-$.
Any other contact form defining $\xi_-=\ker\alpha_-$ on $M_-$
can be realised as the restriction to the tangent spaces of a graph over $M_-$
inside $(-\infty,\varepsilon)\times M_-$ up to a positive constant factor,
cf.\ \cite[Section 3.3]{gz12}.
A change of $(W,\Omega)$ by adding the super-level set of the graph
to the directed symplectic cobordism will allow the verification of the strong
Weinstein conjecture for a particular choice of contact form as announced in
Theorem \ref{thm:maindirectsyplcoborthm}.
In fact,
we will assume that $\alpha_-$ is a generic perturbation of a given $\xi_-$-defining
contact form which allows an application of the Gromov--Hofer compactness theorem
as formulated in \cite{behwz03}.
This is justifiable with the Arzel\`a--Ascoli theorem,
cf.\ \cite[Section 6.4]{gz12} and \cite[Section 6]{sz17}.

On the {\bf negative end} $\big((-\infty,0]\times M_-,\rmd(\rme^s\alpha_-)\big)$
we choose a shift invariant almost complex structure $J_-$
that sends the Liouville vector field $\partial_s$ to the Reeb vector field of $\alpha_-$
and leaves $\xi_-$ invariant such that $J_-$ is compatible with the symplectic structure
induced by $\rmd\alpha_-$ on $\xi_-$.
Extend $J_-$ to a tamed almost complex structure $J$ on $(W,\Omega)$ such that
$\xi_+=TM_+\cap JTM_+$ and the positive boundary $(M_+,\xi_+)$ of $(W,\Omega)$
is $J$-convex, see \cite[Theorem D]{mnw13}.
Further, we glue the deformed magnetic collar extension constructed in Sections
\ref{subsec:truncatiingmagcompl} and \ref{subsec:deformthetrunc} along
$M_+\equiv\{0\}\times M_+$ to build (keeping the notation) a symplectic manifold 
\[
\big(\hat{W},\hat{\Omega}\big):=
\Big((-\infty,0]\times M_-,\rmd(\rme^s\alpha_-)\Big)
\cup
(W,\Omega)\cup
\Big([0,2s_o]\times M_+,\rmd(s\alpha_+)+\eta\Big)
\,.
\]
The resulting almost complex structure is denoted by $\hat{J}$.
Notice that $\hat{J}$ equals $J_o$ on the neighbourhoods
$U_B$ and $U_{\partial N}$ of $B$ and $\partial N$, resp.
Here we think of $B$ and $\partial N$ as subsets of $M_+\equiv\{2s_o\}\times M_+$.
In particular, the results from Section \ref{subsec:gromovcomp} are available.


\subsection{The semi-positivity case}
\label{subsec:thesemiposcase}

We prove Theorem \ref{thm:maindirectsyplcoborthm} under assumption (i).
Recall the moduli space $\MM$ introduced in Section \ref{subsec:gromovcomp}.
To cut out the M\"obius reparametrisation group geometrically we define the moduli
space $\MM_{1,\rmi,-1}$ to be the set of all holomorphic discs $u\in\MM$
such that  $\vartheta\circ u(\rmi^k)=\rmi^k$ for $k=0,1,2$.
This allows to fix the disc reparametrisations for sequences in $\MM_{1,\rmi,-1}$
in the compactness formulation in Remark \ref{rem:gromovcomp}.

We choose a base point $b_o$ of $B$ and an embedded curve
$\gamma$ inside the page $\vartheta^{-1}(1)$ that connects $B$ and $\partial N$
such that $\gamma$ is given by $\{b_o\}\times\big(D^2\cap\R_{\geq0}\big)$ in the
model description in Section \ref{subsec:locmodnearthebind}.
We define the moduli space $\MM_{\gamma}$ to be the set of all holomorphic discs
$u\in\MM_{1,\rmi,-1}$ with $u(1)\in\gamma$.
In other words $\MM_{\gamma}$ consists of all $u\in\MM$ such that
\[
u(1)\in\gamma
\,,\quad
\vartheta\circ u(\rmi)=\rmi
\,,\quad
\vartheta\circ u(-1)=-1
\,.
\]

The Maslov index of the Fredholm problem defined by $\MM$ equals $2$
(see \cite[Proposition 8]{nie06} or \cite{gz10}).
With \cite[Theorem 4.4]{zeh15} there is a generic choice of $\hat{J}$
such that with the first dimension formula in \cite[Theorem 3.7]{zeh15}
(successively taking relations $R=\emptyset$,
$R=\vartheta^{-1}(1)\times\vartheta^{-1}(\rmi)\times\vartheta^{-1}(-1)$,
and $R=\gamma\times\vartheta^{-1}(\rmi)\times\vartheta^{-1}(-1)$)
the moduli spaces $\MM$, $\MM_{1,\rmi,-1}$, and $\MM_{\gamma}$
are smooth manifolds of dimension $n+2$, $n-1$, and $1$, resp.
By \cite[Remark 3.6]{zeh15} we can assume
that the generic perturbation of $\hat{J}$ is supported
in the complement of the union of the negative end $(-\infty,0]\times M_-$ of $\hat{W}$
and $U_B\cup U_{\partial N}$.
This is because all local Bishop discs are Fredholm regular by
\cite[Proposition 9]{nie06} and because all holomorphic discs that are contained
completely inside $\big((-\infty,0]\times M_-\big)\cup U_B\cup U_{\partial N}$ are
the local Bishop discs.

Moreover, the boundary component of $\MM_{1,\rmi,-1}$
that corresponds to the local Bishop filling $F$
has a collar neighbourhood in $\MM_{1,\rmi,-1}$
diffeomorphic to $[0,1)\times B$ via $F$.
This results in a collar neighbourhood
$[0,1)\times\{b_o\}$ in $\MM_{\gamma}$,
see Section \ref{subsec:gromovcomp}.
By the blocking property of $U_{\partial N}$ (see Section \ref{subsec:holomblockbound})
the evaluation map $\MM_{\gamma}\ra\gamma$, $u\mapsto u(1)$, is not surjective. 
Invariance of the mod-$2$ degree for proper maps, which counts the number of preimages
modulo $2$, implies that $\MM_{\gamma}$ cannot be compact.

\begin{proof}[{\bf Proof of Theorem \ref{thm:maindirectsyplcoborthm} part (i)}]
Arguing by contradiction we suppose that there is a compact subset $K$ of $\hat{W}$
such that the holomorphic discs $u(\D)$ are contained in $K$ for all
$u\in\MM_{\gamma}$.
In this situation $\MM_{\gamma}$ can be compactified in the sense of Gromov,
see Remark \ref{rem:gromovcomp}.
Observe that Gromov limiting stable holomorphic discs are contained in $K$ also
and have precisely one disc component that in addition must be simple.
With \cite{zeh15} we can assume by an additional {\it a priori} perturbation of $\hat{J}$
that the moduli spaces of simple stable maps that cover the stable maps in the Gromov
compactification of $\MM$ are cut out transversally.
This perturbation can be supported in the complement of the union of
$(-\infty,0]\times M_-$ and $U_B\cup U_{\partial N}$ because no holomorphic
sphere can stay inside $\big((-\infty,0]\times M_-\big)\cup U_B\cup U_{\partial N}$
completely by the maximum principle.
But, similarly to the computations in \cite[Section 6.3]{gz12}, the moduli spaces of the
covering simple stable holomorphic discs are of negative dimensions, hence, empty.
This argument uses the semi-positivity assumption.
Consequently, there is no bubbling off for the moduli space $\MM_{\gamma}$; in other
words $\MM_{\gamma}$ is compact.
This is a contradiction and therefore a compact subset $K$ of $\hat{W}$ that contains all
holomorphic discs that belong to $\MM_{\gamma}$ cannot exists.

Consequently, $M_-$ is necessarily non-empty.
Moreover, for any choice of $\xi_-$-defining contact form $\alpha_-$ there exists a
null-homologous Reeb link by the remarks made in Section \ref{subsec:complthecob}.
The relevant formulation of Gromov--Hofer convergence is obtained by combining
the convergences statements in \cite{hwz98,hwz03} with \cite{fz15}.
\end{proof}


\section{A Deligne--Mumford type space}
\label{sec:adelignemumfordtypespace}

In Section \ref{sec:polyfoldperturbations} the proof of Theorem \ref{thm:maindirectsyplcoborthm}
under assumption (ii) will be given.
In preparation we discuss moduli spaces of stable nodal boundary un-noded discs.
We follow \cite[Section 2.1]{hwz-gw17}, \cite{cie18,hwz-dm12} and indicate modifications necessary
in the presence of boundaries.


\subsection{Boundary un-noded nodal discs}
\label{subsec:boundunnodnoddisc}

Let $S$ be an oriented surface that is equal to the disjoint union of one closed disc
and a (possibly empty) finite collection of spheres.
All connected components of $S$ are provided with the standard orientation.
Let $j$ be an orientation preserving complex structure on $S$ turning $(S,j)$
into a Riemann surface with boundary, i.e.\ $(S,j)$ admits a holomorphic atlas
whose charts are given by open subsets of the closed upper half-plane.

We call a subset of $\Int(S)$ consisting of two distinct points a {\bf nodal pair}.
Each finite collection $D$ of pair-wise disjoint nodal pairs
defines an equivalence relation on $S$
calling two points equivalent if and only if they from a nodal pair.
The set $S/D$ of equivalence classes is provided with the quotient space topology.

Let $D$ be a finite collection of pair-wise disjoint nodal pairs such the quotient space
$S/D$ is simply connected.
We call $(S,j,D)$ a {\bf boundary un-noded nodal disc}.
A point of $S$ that belongs to a nodal pair in $D$ is called a {\bf nodal point}.
The set of nodal points is denoted by $|D|$.
Observe that $|D|$ and $\partial S$ are disjoint.

Let $m_0,m_1,m_2$ be pair-wise distinct {\bf marked points} on the boundary
$\partial S$ ordered according to the boundary orientation of $\partial S$.
We call $\big(S,j,D,\{m_0,m_1,m_2\}\big)$ a {\bf marked boundary un-noded nodal disc}.
Two marked boundary un-noded nodal discs
\[
\big(S,j,D,\{m_0,m_1,m_2\}\big)
\quad
\text{and}
\quad
\big(S',j',D',\{m'_0,m'_1,m'_2\}\big)
\]
are {\bf equivalent} if there exists a diffeomorphism $\varphi\co S\ra S'$
such that $\varphi^*j'=j$,
the injection $D\ra D'$ defined by
$\{\varphi(x),\varphi(y)\}\in D'$ for all $\{x,y\}\in D$
is onto, and
$\varphi(m_k)=m'_k$ for $k=0,1,2$.
Observe that $\varphi$ necessarily preserves orientations.


\subsection{Domain stabilisation\label{subsec:domainstabilisation}}

In Section \ref{subsec:boundunnodstabdiscs}
boundary un-noded nodal discs will appear
as the domain of stable maps.
If the domain nodal discs have sphere components,
the nodal discs will be unstable. 
In order to obtain a natural groupoidal structure on the space of marked boundary
un-noded nodal discs we have to stabilise these by adding marked points.
A point that is a marked point or a nodal point is called {\bf special}.
We call a connected component $C$ of $S$ {\bf stable} if the number of
special points on $C$ is greater or equal than $3$.
In particular the disc component of $S$ is stable.

We consider equivalence classes of stable discs $\big[S,j,D,\{m_0,m_1,m_2\},A\big]$
where $\big(S,j,D,\{m_0,m_1,m_2\}\big)$ is a marked boundary un-noded nodal disc
as in Section \ref{subsec:boundunnodnoddisc} that we provide with an additional finite
set of {\bf auxiliary marked points} $A\subset S\setminus\partial S$ in the
complement of $|D|$ so that $\#\big((A\cup|D|)\cap C\big)\geq3$ for each sphere
component $C$ of $S$.
In particular, all components $C$ of $S$ are stable.
The {\bf equivalence relation} is given by diffeomorphisms $\varphi\co S\ra S'$
as in Section \ref{subsec:boundunnodnoddisc} such that in addition $\varphi$ maps
$A$ bijectively onto $A'$.

The set of all equivalence classes $\RR$ is a {\bf nodal Riemann moduli space}.
Given a non-negative integer $N$ we denote by $\RR_N\subset\RR$ the subset
of stable nodal discs that are equipped with precisely $N=\#A$ auxiliary marked points
so that $\RR$ is the disjoint union over all $\RR_N$, $N\geq0$.
The elements in $\RR_N$ can be represented by stable nodal discs 
$\big(S,j,D,\{m_0,m_1,m_2\},A\big)$ of different {\bf stable nodal type} $\tau$, which is an
isomorphism class of weighted rooted trees.
The vertices are given by the components of $S$, where the root corresponds to the disc
component.
All vertices are weighted by the number of (auxiliary) marked points on the corresponding
component of $S$.
The edge relation is given by the nodes in $D$.
Observe that for given $N$ the number of stable nodal types corresponding to stable discs
$\big(S,j,D,\{m_0,m_1,m_2\},A\big)$ with $N=\#A$ is finite so that $\RR_N$ is a finite disjoint
union of subsets of stable discs $\RR_{\tau}\subset\RR$ of the same stable nodal type $\tau$.


\subsection{Groupoid as an orbit space\label{subsec:groupoidasorbitspace}}

Let $\tau$ be a stable nodal type.
In order to rewrite $\RR_{\tau}$ as an orbit space we choose natural representatives
of the stable nodal marked discs $\big[S,j,D,\{m_0,m_1,m_2\},A\big]$ in $\RR_{\tau}$ as
follows:
We fix the oriented diffeomorphism types of $\C P^1$ and $\D$ for the components of $S$
including the choice $\{1,\rmi,-1\}$ for the marked points $\{m_0,m_1,m_2\}$, i.e.\
$m_k=\rmi^k$, $k=0,1,2$.
We denote by $\sigma$ the {\bf area form} on $S$ that is the Fubini--Study form on
the $\C P^1$ components and equals $\rmd x\wedge\rmd y$ on $\D$
taking conformal coordinates $x+\rmi y$.
Let $\JJ\equiv\JJ_S$ be the space of orientation preserving complex structures on $S$,
which equals the space of almost complex structures on $S$ tamed by $\sigma$,
cf.\ \cite{abb14,sal11}.
Furthermore we fix a selection of special points $|D|$ and $A$.
By $\GG\equiv\GG(S,D,\{1,\rmi,-1\},A)$ we denote the group of orientation preserving
diffeomorphisms of $S$ that preserve $\{1,\rmi,-1\}$ point-wise and inject $A$ onto $A$
and $D$ onto $D$, resp.

We identify the nodal Riemann moduli space $\RR_{\tau}$ with the orbit space
\[
\RR_{\tau}=\JJ/\GG\,
\qquad\text{via}\quad
\big[S,j,D,\{1,\rmi,-1\},A\big]\equiv [j]\,,
\]
cf.\ \cite[p.~612]{rs06} or \cite[Section 4.2]{wen10}.
The action
\[
\GG\times\JJ\lra\JJ
\,,\quad(\varphi,j)\longmapsto\varphi^*j\,,
\]
of $\GG$ on $\JJ$ is given by the pull back
\[
\varphi^*j:=T_{\varphi}\varphi^{-1}\circ j_{\varphi}\circ T\varphi
\]
for $\varphi\in\GG$ and $j\in\JJ$.
By \cite[Lemma 7.5]{rs06} or \cite[Lemma 4.2.8]{wen10} this action is proper, see also
Remark \ref{rem:topologyfixedstabletype} below.
The action is free if and only if all isotropy subgroups $\GG_j$, $j\in\JJ$, of
$j$-holomorphic maps in $\GG$ are trivial.
The action is locally free because all isotropy subgroups $\GG_j$ are finite by the
following Remark \ref{rem:isotropyisfinite}:

\begin{rem}
 \label{rem:isotropyisfinite}
 Each connected component of $S$ is provided with at least $3$ special points,
 see Section \ref{subsec:domainstabilisation}.
 Hence, all {\bf isotropy subgroups} $\GG_j$ {\bf are finite}:
 
 To see this fix a biholomorphic identification of $(S,j)$ with the surface given by the
 disjoint union of $(\D,\rmi)$ and an at most finite number of copies of $(\C P^1,\rmi)$.
 This is possible by uniformisation, cf.\ \cite{abb14} and \cite[Theorem C.5.1]{mcsa04},
 \cite[Satz 5.33]{sal11} for boundary regularity. 
 For the marked points we can assume that $m_k=\rmi^k$, $k=0,1,2$.
 Then any automorphism in $\GG_j$ conjugates to an $\rmi$-holomorphic map that
 restricts to the identity on $(\D,\rmi)$ and defines M\"obius transformations of
 $(\C P^1,\rmi)$ corresponding to the maps induced between the not necessarily identical
 sphere components.
 We get a finite number of possibilities to obtain those M\"obius transformations each of
 which is permuting the set of special points that admits at least $3$ points by the stability
 condition.
\end{rem}

\begin{rem}
 \label{rem:topologyfixedstabletype}
 In order to describe the {\bf topology for fixed stable nodal type} provide $\JJ$ with the
 $C^{\infty}$-topology, which is metrisable, complete and locally compact by the
 Arzel\`a--Ascoli theorem.
 
 We provide $\RR_{\tau}=\JJ/\GG$ with the {\bf quotient topology} meaning that the open sets
 in $\JJ/\GG$ are precisely those, whose preimage under the quotient map
 $[\,.\,]\co\JJ\ra\JJ/\GG$ is open.
 In particular, $[\,.\,]$ is continuous by definition.
 The quotient map $[\,.\,]$ is open because for any open subset $\KK$ of $\JJ$ the
 $[\,.\,]$-preimage of $[\KK]$ is equal to the union of all $g\KK$, $g\in\GG$, which is open.
 Hence, a neighbourhood base of the topology on $\JJ/\GG$ is given by the family of subsets
 whose elements $[j]$ can be represented by complex structures $j$ belonging to an open
 subset of $\JJ$.
 Therefore, $\JJ/\GG$ is a second countable locally compact and,
 hence, paracompact topological space.
 
 In fact, $\RR_{\tau}=\JJ/\GG$ is Hausdorff.
 This follows from the {\bf properness argument}
 as follows:
 Consider a sequence $u_{\nu}$ of equivalences
 \[
 u_{\nu}\co
 \big(S,j_{\nu},D,\{1,\rmi,-1\},A\big)
 \lra
 \big(S,k_{\nu},D,\{1,\rmi,-1\},A\big)
 \]
 with $j_{\nu}\ra j$ and $k_{\nu}\ra k$ in $\JJ$.
 We claim that $u_{\nu}$ has a $C^{\infty}$-convergent subsequence whose limit $u$
 will be an equivalence
 \[
 u\co
 \big(S,j,D,\{1,\rmi,-1\},A\big)
 \lra
 \big(S,k,D,\{1,\rmi,-1\},A\big)
 \]
 also, cf.\ \cite[Proposition 3.25]{hwz-dm12}.
 
 It suffices to proof $C^{\infty}$-convergence because the limit $u$ will be automatically an
 equivalence as $u$ restricts to a degree $1$ map on each component of $S$.
 Restricting $u_{\nu}$ to the components of $S$ we obtain sequences of
 $k_{\nu}$-holomorphic diffeomorphisms $v_{\nu}$ one sequence for each component
 of $(S,j_{\nu})$.
 As the degree of each $v_{\nu}$ equals $1$ viewed as map onto its image,
 the area $\int_Cv_{\nu}^*\sigma=\pi$,
 $C$ being $\D$ or $\C P^1$,
 is uniformly bounded via the transformation formula.
 Hence, we find a subsequence of $u_{\nu}$ so that all corresponding sequences $v_{\nu}$
 converge in the sense of Gromov, see \cite{fz15}.
 Again using that the degree of all $v_{\nu}$ is $1$ we see that only one of the potential
 bubbles of each Gromov limit can intersect $0\in C$; the remaining bubbles would be
 necessarily constant as their area vanish.
 In other words, the chosen subsequence of $u_{\nu}$ converges in $C^{\infty}$, because
 the reparametrisations by $j_{\nu}$-holomorphic diffeomorphisms are fixed,
 i.e.\ equal $\id$,
 due to the stability condition.
 
 We verify the Hausdorff property, cf.\ \cite[Proposition 3.19]{hwz-dm12}:
 Assume that any pair of neighbourhoods of given points $[j]$ and $[k]$, resp.,
 intersects non-trivially.
 Taking shrinking neighbourhoods of the corresponding points $j$ and $k$ in $\JJ$
 we find $j_{\nu}\ra j$ and $k_{\nu}\ra k$ in $\JJ$ such that $[j_{\nu}]=[k_{\nu}]$ for all $\nu$.
 Then the above properness argument yields $[j]=[k]$.
\end{rem}


\subsection{Infinitesimal action\label{subsec:infinitesimalaction}}

We assume the situation of Section \ref{subsec:groupoidasorbitspace}.
Denote by $\Omega^0$ the Lie algebra given by the tangent space of $\GG$
at the identity diffeomorphism, which is the space of vector fields on $(S,\partial S)$
that are stationary on the set of all special points.
In other words, $\Omega^0$ is the set of all smooth sections of $TS$ that are
tangent to $\partial S$ and vanish on $|D|\cup\{1,\rmi,-1\}\cup A$.
We denote by $\Omega^{0,1}_j$ the space of all endomorphism fields of the
tangent bundle of $S$ that anti-commute with $j$.
Linearising the equation $j^2=-1$ we see that $\Omega^{0,1}_j$ is
the tangent space of $\JJ$ at $j$.
Viewing $\Omega^{0,1}_j$ as the set of all
$j$-complex anti-linear $TS$-valued differential forms on $S$,
the infinitesimal action is
\[
\Omega^0\times\Omega^{0,1}_j\lra\Omega^{0,1}_j
\,,\quad
(X,y)\longmapsto L_Xj+y
\,.
\]

A {\bf complex linear Cauchy--Riemann operator} $D$ is a $j$-complex linear operator
that maps smooth vector fields of $(S,j)$ to $\Omega^{0,1}_j$ such that
$D(fX)=\DB f\cdot X+fDX$ for all smooth vector fields $X$ and smooth functions $f$,
where $\DB$ is the composition of the exterior derivative with the projection of the space of
smooth $1$-forms onto those that anti-commute with $j$, see \cite[Appendix C.1]{mcsa04}.
Complex linearity can be expressed via $D(jX)=jDX$ for all vector fields $X$ of $S$.

In order to compute the Lie derivative $L_Xj$ we denote by $\CR j$ the uniquely
determined complex linear Cauchy--Riemann operator of the holomorphic line bundle
$(TS,j)$ (ignoring boundary points) that agrees with
\[
\DB X=\tfrac12\big(TX+\rmi\circ TX\circ\rmi\big)
\]
in local holomorphic coordinates $(\C,\rmi)$, cf.\ \cite[Remark C.1.1]{mcsa04}.
The local holomorphic representation $\DB$ of $\CR j$ implies that $\CR j$ induces a
{\bf real linear Cauchy--Riemann operator} on the bundle pair $\big((TS,T\partial S),j\big)$
taking local holomorphic coordinates in the closed upper half-plane.
Nevertheless the operator $\CR j\co\Omega^0\ra\Omega^{0,1}_j$
{\bf is not complex linear} because $j$ does not induce a complex linear vector space
structure on $\Omega^0$.
Indeed, $jX$ is not tangent to $\partial S$ for all boundary points of $S$
for which $X\in\Omega^0$ does not vanish.

With this preparation we compute
\[
L_Xj=\frac{\rmd}{\rmd t}\Big|_{t=0}\varphi_t^*j\,,
\]
where $\varphi_t$ is a smooth path in $\GG$ through $\varphi_0=\id$ with
\[
\frac{\rmd}{\rmd t}\Big|_{t=0}\varphi_t=X\,.
\]
In local holomorphic coordinates $\varphi_t^*j$ reads as
\[
T_{\varphi_t}\varphi_t^{-1}\circ\rmi\circ T\varphi_t\,.
\]
Taking the time derivative at $t=0$ and using $\varphi_t^{-1}=\varphi_{-t}$ yields
\[
-TX\circ\rmi+\rmi\circ TX=
\rmi\big(\rmi\circ TX\circ\rmi+TX\big)\,,
\]
so that the Lie derivative is equal to
\[
L_Xj=2j\CR jX
\,.
\]

\begin{rem}
 \label{rem:kernelistangentspace}
 Taking a path $\varphi_t$ in the isotropy subgroup $\GG_j$
 through $\varphi_0=\id$,
 meaning that the path $\varphi_t$ in $\GG$ satisfies $\varphi_t^*j=j$ for all $t$,
 we obtain $X\in\ker\CR j$ by taking time derivative.
 Conversely,
 assuming $\CR jX=0$ for a path $\varphi_t$ in $\GG$ through $\varphi_0=\id$
 yields
 \[
  \frac{\rmd}{\rmd t}\big(\varphi_t^*j\big)=
  \varphi_t^*\Big(2j\CR jX\Big)=0
  \,,
 \]
 i.e.\ that the path $\varphi_t^*j$ in $\JJ$ is constant, hence, equals $j$.
 In other words,
 the tangent space at the identity
 of the group of all $j$-holomorphic maps in $\GG$
 \[
 T_{\id}\GG_j=\ker\CR j
 \]
 equals the space of all $j$-holomorphic vector fields in $\Omega^0$.
\end{rem}

\begin{rem}
 \label{rem:automphiandbdcommute}
 The Cauchy--Riemann operator is {\bf conformally invariant}:
 Let $\varphi\co(S,j)\ra(S,k)$ be a holomorphic diffeomorphism.
 Then $\varphi^*\circ\CR k=\CR j\circ\varphi^*$.
 In particular, the Cauchy--Riemann operator $\CR j$
 commutes with the automorphisms
 of the Riemann surface with boundary $(S,j)$.
\end{rem}

\begin{rem}
 \label{rem:crvialc}
 Define a Riemannian metric $g_j$ on $(S,j)$ by setting
 \[
 g_j(v,w)=\tfrac12\big(\sigma(v,jw)+\sigma(w,jv)\big)
 \,,\quad v,w\in TS\,,
 \]
 where the area form $\sigma$ is the one chosen in
 Section \ref{subsec:groupoidasorbitspace}.
 In particular, the complex structure $j$ is orthogonal,
 so that the bilinear form $g_j$ is a {\bf Hermitian metric}.
 As any $2$-form on a surface is closed the $2$-form
 \[
 \sigma_j:=\tfrac12(\sigma+j^*\sigma)
 \]
 is symplectic and compatible with $j$.
 The latter means that $g_j(v,w)=\sigma_j(v,jw)$ for all $v,w\in TS$,
 so that the symplectic
 form $\sigma_j$ is compatible with $j$.
 
 Denoting the {\bf Levi-Civita connection} of $g_j$ by $\nabla\equiv\nabla^{g_j}$,
 \cite[Lemma 4.15]{mcsa98} says that closedness of $\sigma_j$ and integrability of $j$
 together are equivalent to $\nabla j=0$,
 cf.\ Remark \ref{rem:jorthogonalthenparallel} below.
 In particular, $\nabla(jX)=j\nabla X$ for all vector fields $X$ of $S$, so that $\nabla$ is a
 {\bf Hermitian connection}, i.e.\  $\nabla$ is a complex linear metric connection.
 Hence,
 \[
 \big(\nabla X\big)^{0,1}:=\tfrac12\big(\nabla X+j\circ\nabla X\circ j\big)
 \]
 defines a complex linear Cauchy--Riemann operator on $(TS,j)$,
 cf.\ \cite[Remark C.1.2]{mcsa04}.
 In local holomorphic coordinates we write $\nabla X=TX+\Gamma(\,.\,,X)$ with help
 of the Christoffel symbols, so that the map $X\mapsto\Gamma(\,.\,,X)$ is
 $\rmi$-complex linear.
 With symmetry of $\Gamma$ we get
 $
 \rmi\circ\Gamma(\rmi\,.\,,X)=
 \rmi\circ\Gamma(X,\rmi\,.\,)=
 -\Gamma(\,.\,,X)
 $,
 so that $\big(\nabla X\big)^{0,1}=\DB X$.
 Therefore,
 \[
 \big(\nabla X\big)^{0,1}=\CR jX\,.
 \]
 Furthermore we remark that the Hermitian connection $\nabla$ is uniquely determined
 by this equation, see \cite[Remark C.1.2]{mcsa04}.
 Consequently, $L_Xj=2j\CR jX$ can be obtained with the computations on
 \cite[Theorem C.5.1]{mcsa04} or \cite[p.~631]{rs06} as well.
\end{rem}

\begin{rem}
 \label{rem:jorthogonalthenparallel}
 We give an alternative argument for the fact that $j$ is parallel, which we used in Remark
 \ref{rem:crvialc}:
 For any non-vanishing tangent vectors $v,w$ at any given point of $S$ consider
 a curve $c$ tangent to $v$ and extend $w$ to a parallel vector field $X$ along $c$.
 As $X$ and $jX$ are orthogonal and as the length of $jX$ is constant $\nabla$ being metric
 implies that $\nabla_{\dot c}(jX)$ is perpendicular to the span of $\{X,jX\}$,
 and hence vanishes.
 Therefore, with the Leibnitz rule and parallelity of $X$ we get $(\nabla_{\dot c}j)X=0$.
 Consequently, $\nabla j=0$.
 
 Observe that the argument works for all metrics (and corresponding Levi-Civita connections)
 for which $j$ is orthogonal.
\end{rem}

\begin{rem}
 \label{rem:jautoareallisometries}
 All elements of $\GG_j$ are {\bf isometries} of $g_j$.
 Indeed, by finiteness of $\GG_j$ one finds for each $\psi\in\GG_j$ a natural number
 $k$ such that $\psi^k=\id$.
 On the other hand $\psi$ pulls $g_j$ back to $fg_j$ for some positive function $f$
 on $S$ by the description in Remark \ref{rem:crvialc} and the fact that all positive
 area forms on a surface are positively proportional.
 Therefore, the conformal factor of the pull-back of $g_j$ by $\psi^k$ becomes $f^k$,
 which necessarily is $1$.
 Hence, $f=1$.
\end{rem}

\begin{rem}
 \label{rem:incoisotropy}
 Let $(\varphi_t,j_t)$ be a smooth path in $\GG\times\JJ$ through
 $(\varphi_0,j_0)=(\psi,j)$ and denote the velocity vector field by
 \[
 \frac{\rmd}{\rmd t}(\varphi_t,j_t)=\big(X_t,y_t\big)
 \in T_{\varphi_t}\GG\times T_{j_t}\JJ
 \,.
 \]
 The corresponding velocity vector field of $\varphi_t^*j_t$ equals
 \[
 \frac{\rmd}{\rmd t}\big(\varphi_t^*j_t\big)=
 \varphi_t^*\Big(2j_t\CR {j_t}X_t+y_t\Big)
 \in T_{\varphi_t^*j_t}\JJ
 \,.
 \]
 If $\psi\in\GG_j$ we obtain with $(X,y)=(X_0,y_0)$ that
 \[
 \frac{\rmd}{\rmd t}\Big|_{t=0}\big(\varphi_t^*j_t\big)=
 \psi^*\Big(2j\CR jX+y\Big)
 \in T_j\JJ
 \,,
 \]
 which is equal to
 \[
 2j\CR j\big(\psi^*X\big)+\psi^*y
 \,.
 \]
 Indeed, this follows with Remark \ref{rem:automphiandbdcommute} or with Gau{\ss}'s
 {\it theorema egregium} which gives $\psi^*\circ\CR j=\CR j\circ\psi^*$ because all
 $\psi\in\GG_j$ are isometries of $g_j$, see Remark \ref{rem:jautoareallisometries}.
 Consequently, the corresponding infinitesimal action reads as
 \[
 T_{\psi}\GG\times\Omega^{0,1}_j\lra\Omega^{0,1}_j
 \,,\quad
 (X,y)\longmapsto2j\CR j\big(\psi^*X\big)+\psi^*y
 \]
 for all $\psi\in\GG_j$.
 Observe that the infinitesimal action $2j\CR j\oplus\mathbbm{1}$ at $(\id,j)$
 sends $\big(\psi^*X,\psi^*y\big)\in\Omega^0\times\Omega^{0,1}_j$
 to the same element in $\Omega^{0,1}_j$.
\end{rem}


\subsection{A Fredholm index\label{subsec:fredindofinfiact}}

We compute the Fredholm index of the $W^{1,3}$-Sobolev completed
Cauchy--Riemann operator $\CR j\co W^{1,3}\ra L^3$ induced by
the Cauchy--Riemann operator from Section \ref{subsec:infinitesimalaction}.

Ignoring zeros for the moment,
for each component $C$ of $S$
the Fredholm index is given by the Riemann--Roch \cite[Theorem C.1.10]{mcsa04}
applied to the $j$-complex $1$-dimensional bundle pair $(TC,T\partial C)$.
Namely, the Fredholm index
is the sum of the Euler characteristic of $C$
and the Maslov index of $(TC,T\partial C)$.
The Maslov index for $C=\D$ is $2$ by normalisation;
for $C=\C P^1$ twice the first Chern number,
i.e.\ twice the Euler characteristic, which gives $4$,
see \cite[Chapter C.3]{mcsa04}.
Hence, the Fredholm index is $3$ restricted to the disc component
and $6$ on the spheres.

Now we take the zeros into account.
Component-wise,
for each boundary zero we have to subtract $1$
from the computed Fredholm index;
for each interior zero we subtract $2$.
Adding up, we obtain that $\ind\CR j$ equals
\[
3-\#\{1,\rmi,-1\}-2\#\big((|D|\cup A)\cap\D\big)
+2\sum_C\Big(3-\#\big((|D|\cup A)\cap C\big)\Big)
\,,
\]
where the sum is taken over all sphere components $C$ of $S$.
This gives
\[
\ind\CR j=
-2\#|D|-2\#A
+6\big(\#\{C\}-1\big)
\,,
\]
where $\#\{C\}=\#D+1$ is the number of all components of $S$.
Using $\#|D|=2\#D$ we finally obtain
\[
\ind\CR j=2\big(\#D-\#A\big)
\,.
\]

On the other hand,
by a boundary version of the argument principle
(see \cite[Theorem A.5.4]{ah19}),
the Maslov index of $(TC,T\partial C)$
for each component $C$ of $S$
is the weighted sum of the
number of zeros counted multiplicities
of a non-zero element in the kernel of $\CR j$,
where interior zeros are counted twice.
As the corresponding Maslov index is $2$ on the disc component
and $4$ on the spheres, the kernel of $\CR j$ is trivial
due to the stability condition that the number of special points on each
component $C$ of $S$ is at least $3$.
This argument uses elliptic regularity
saying that the vector fields in $\ker\CR j$ are smooth,
see \cite{mcsa04}.
Therefore, one can show triviality of $\ker\CR j=T_{\id}\GG_j$
(see Remark \ref{rem:kernelistangentspace})
alternatively using finiteness of $\GG_j$
under the stability condition, see Remark \ref{rem:isotropyisfinite}. 

Using elliptic regularity as in \cite{mcsa04}
we see that the Cauchy--Riemann operator
$\CR j\co\Omega^0\ra\Omega^{0,1}_j$
is injective.
The image is closed and has codimension $2\big(\#A-\#D\big)$,
cf.\ \cite[Proposition 2.5]{hwz-gw17}.


\subsection{Interlude:\ Cayley transformation\label{subsec:cayleytransform}}

We provide $\Omega^{0,1}_j$ with the norm $|y|_j$
given by the maximum of point-wise
operator norms of $y$ w.r.t.\ the metric $g_j$.
Because $y$ is self-adjoint w.r.t.\ $g_j$
the norm $|y|_j$ is equal to the maximum of the square-root of the point-wise
eigenvalues of the $j$-complex linear endomorphism field $y\circ y$.
In particular, the resolvent $(1-y)^{-1}$ of $y$ at $1$ is defined provided $|y|_j<1$.
We obtain a homeomorphism
\[
\mathfrak{i}\co
\Omega^{0,1}_j\cap B_1(0)\lra\JJ
\,,\quad
0\longmapsto j
\,,
\]
defined on the open unit ball about zero via the conjugation
\[
y\longmapsto(1-y)j(1-y)^{-1}
\,,
\]
whose inverse is given by the
\[
k\longmapsto(k+j)^{-1}(k-j)
\,,
\]
cf.\ \cite[Proposition 1.1.6]{aud94}, \cite[Proposition 2.6.4]{mcsa17} or
\cite[p.~634]{rs06}.
Because $\GG_j$ acts by isometries of $g_j$
(see Remark \ref{rem:jautoareallisometries})
the conjugation map $\mathfrak{i}$ is $\GG_j$-equivariant.

We claim that $\JJ$ is a submanifold of the space $\Omega^1$
of all endomorphism fields of the tangent bundle of $S$
and that $\mathfrak{i}$ is a global chart.
The case of almost complex structures compatible with $\sigma$
essentially follows with the expositions in the above cited literature
\cite{aud94, mcsa17, rs06}.
We will follow \cite[Chapter I.7.3]{gz23}
taking the modifications for the case of almost complex structures
only tamed by $\sigma$ into account:

We call a not necessarily symmetric endomorphism field
$x\in\Omega^1$ positive and write $x>0$
provided that for all non-zero tangent vectors $v\in TS$
the quadratic form $g_j(v,xv)$ is positive.
As the kernel of a positive endomorphism field
$x\in\Omega^1$ is trivial
we see that the inverse $x^{-1}\in\Omega^1$ exists.
Therefore, for positive $x\in\Omega^1$
the endomorphism field $1+x$ is positive as well
such that the inverse $(1+x)^{-1}$ exists.
In fact, the half space of all positive
endomorphism fields $\Omega^1\cap\{x>0\}$
is an open cone in $\Omega^1$ closed under taking inverses.
As above we provide $\Omega^1$ with the norm $|x|_j$
given by the maximum of point-wise
operator norms of $x\in\Omega^1$ w.r.t.\ the metric $g_j$.

The {\bf Cayley transform}
\[
\CC(x):=(1-x)(1+x)^{-1}=:\tilde{y}
\]
defines a map
\[
\Omega^1\cap\{x>0\}\lra\Omega^1\cap B_1(0)
\,.
\]
Indeed, setting $v=(1+x)^{-1}w$ polarisation yields
$4g_j(v,xv)=|w|_j^2-|\tilde{y}w|_j^2$, so that $x$ is positive
if and only if $|\tilde{y}|_j<1$.
We remark that there is an alternative formula
\[
\CC(x)=(1+x)^{-1}(1-x)
\]
for the Cayley transform because $1+x$ and $1-x$ commute.
Setting
\[
\CC(\tilde{y}):=(1-\tilde{y})(1+\tilde{y})^{-1}
\]
we obtain a map
\[
\CC\co\Omega^1\cap B_1(0)\lra\Omega^1\cap\{x>0\}
\]
in the converse direction.
Again, this uses polarisation and
that $|\tilde{y}|_j<1$ implies triviality of the kernel of $1+\tilde{y}$.
Because the Cayley transform $\CC$ is involutive the map
$\CC\co\Omega^1\cap\{x>0\}\ra\Omega^1\cap B_1(0)$
is a diffeomorphism.

Extending the conjugation map
$\mathfrak{i}\co\tilde{y}\mapsto (1-\tilde{y})j(1-\tilde{y})^{-1}$
to the open set $\Omega^1\cap B_1(0)$
yields a smooth map
$\Omega^1\cap B_1(0)\ra\Omega^1$.
Observe, that $\mathfrak{i}(\tilde{y})$ is a complex structure
potentially reversing orientation.
Restricting $\mathfrak{i}$ to the $j$-complex anti-linear part
$\Omega^{0,1}_j$ of $\Omega^1$ we get
$\mathfrak{i}(y)=j\CC(-y)$ for all $y\in\Omega^{0,1}_j\cap B_1(0)$.
This defines an injection
\[
\mathfrak{i}=j\circ\CC\circ(-1)\co
\Omega^{0,1}_j\cap B_1(0)\lra\Omega^1
\,.
\]
In order to describe the image we observe that
\[
\CC\circ(-1)\co
\Omega^{0,1}_j\cap B_1(0)\lra\Omega^1\cap\{x>0\}\cap\{jx=x^{-1}j\}
\]
is a well defined homeomorphism with inverse $(-1)\circ\CC$.
Additionally, the multiplication map $j\co\Omega^1\ra\Omega^1$
is a diffeomorphism with inverse $-j$ and restricts to the homeomorphism
\[
j\co
\Omega^1\cap\{x>0\}\cap\{jx=x^{-1}j\}\lra\JJ
\,.
\]
Indeed, positivity of $x$ is equivalent to the positivity of
$\sigma_j(v,jxv)$ for all $v\in TS$.
Moreover, we find a smooth function $f$ on the surface $S$
such that the $2$-forms $\sigma$ and $\sigma_j$ satisfy $\sigma=f\sigma_j$.
Because $\sigma(w,jw)=f|w|_j^2$ is positive for all non-zero $w\in TS$
the function $f$ must be positive.
Hence, positivity of $x$ is equivalent to the positivity of
$\sigma(v,jxv)$ for all $v\in TS$,
as $\sigma_j$ and $\sigma$ are positively proportional.
It follows that the complex structure $jx$ is tamed by $\sigma$,
i.e.\ preserves the orientation of the Riemann surface $(S,j)$.
In total, the conjugation map
\[
\mathfrak{i}=j\circ\CC\circ(-1)\co
\Omega^{0,1}_j\cap B_1(0)\lra\JJ
\]
is a well defined homeomorphism.
The inverse is $k\mapsto-\CC(-jk)=(k+j)^{-1}(k-j)$,
where equality follows with the above alternative commuted formula
for the Cayley transform.
Furthermore $\mathfrak{i}=j\circ\CC\circ(-1)$
is the restriction of the diffeomorphism
obtained as the composite of
$\CC\circ(-1)\co\Omega^1\cap B_1(0)\ra\Omega^1\cap\{x>0\}$
with $j\co\Omega^1\ra\Omega^1$ where defined.
In other words,
the inverse $(-1)\circ\CC\circ(-j)$ of $j\circ\CC\circ(-1)$
serves as a global submanifold chart
of $\JJ\subset\Omega^1$ the model being
$\Omega^{0,1}_j\cap B_1(0)\subset\Omega^1$. 

In fact, $\JJ$ is a complex manifold:
A complex structure on the vector space $\Omega^{0,1}_j$
is given by $y\mapsto jy$.
An almost complex structure on $\JJ$
is given by
\[\mathfrak{i}(y)=(1-y)j(1-y)^{-1}\]
on the tangent space
$T_{\mathfrak{i}(y)}\JJ=\Omega^{0,1}_{\mathfrak{i}(y)}$
for all $y\in\Omega^{0,1}_j\cap B_1(0)$.
Taking derivative of the equation
\[\mathfrak{i}(y)(1-y)=(1-y)j\]
w.r.t.\ $y$ we get
$
T_y\mathfrak{i}(\dot y)(1-y)-\mathfrak{i}(y)\dot y=
-\dot y j
$
and, using $-\dot y j=j\dot y$, that
\[
T_y\mathfrak{i}(\dot y)=
\big(j+\mathfrak{i}(y)\big)\dot y(1-y)^{-1}
\]
for the linearisation of $\mathfrak{i}$ at $y\in B_1(0)$
for all tangent vectors $\dot y$ of
$B_1(0)\subset\Omega^{0,1}_j$.
Using
$\mathfrak{i}(y)\big(j+\mathfrak{i}(y)\big)=\big(j+\mathfrak{i}(y)\big)j$
we obtain
\[
\mathfrak{i}(y)\circ T_y\mathfrak{i}=T_y\mathfrak{i}\circ j
\,.
\]
In other words, the linearisation $T_y\mathfrak{i}$
is complex linear for all $y\in\Omega^{0,1}_j\cap B_1(0)$.
Therefore,
$\mathfrak{i}\co\Omega^{0,1}_j\cap B_1(0)\lra\JJ$
is biholomorphism.

\begin{rem}
 \label{rem:viahomospace}
 With \cite[Corollary 6.4]{abklr94} one finds a further action by conjugation via
 $y\mapsto\rme^yj\rme^{-y}$ for all $y\in\Omega^{0,1}_j$ defining a
 $\GG_j$-equivariant diffeomorphism $\Omega^{0,1}_j\ra\JJ$:
 To see this it is enough to argue fibre-wise.
 Identifying any given tangent space of $S$ with $\R^2$ we claim that the space of all
 orientation preserving complex multiplications $\JJ$ on $\R^2$
 is the homogeneous space $\PSl_2(\R)/\PSO_2$.
 Indeed, the group $\Gl_2^+(\R)$ of orientation preserving invertible linear maps on
 $\R^2$ acts transitively on $\JJ$ by conjugation $A\mapsto A\rmi A^{-1}$ and the isotropy
 subgroup at $\rmi$ is isomorphic to $\Gl_1(\C)$, cf.\ \cite[Proposition 2.48]{mcsa98}.
 Normalising via the determinant and dividing out $\pm 1$ this action descents to a
 transitive and faithful action of $\PSl_2(\R)$ with isotropy subgroup $\PSO_2$, as a
 conformal linear map that preserves the area necessarily preserves the metric.
 In particular, we see that $\JJ$ is the hyperbolic upper half-plane $\PSl_2(\R)/\PSO_2$,
 cf.\ \cite[Exercise 4.17]{mcsa98}.
 
 The Lie algebra of $\PSl_2(\R)$ decomposes as a vector space into the Lie algebra of
 $\PSO_2$ and the set $\Omega^{0,1}_{\rmi}$ of all $2\times 2$ matrices that
 anti-commute with $\rmi$.
 Observe that the elements of $\Omega^{0,1}_{\rmi}$ are symmetric and trace-free.
 Therefore, the exponential map of the tangent space of $\PSl_2(\R)/\PSO_2$ at $[1]$
 can be written as
 \[
 \Omega^{0,1}_{\rmi}\lra\PSl_2(\R)/\PSO_2
 \,,\quad\,
 Y\longmapsto\big[\rme^Y\big]
 \,.
 \]
 This map is a diffeomorphism because any $A\in\Sl_2(\R)$ can be written uniquely as
 $A=\rme^SR$ for $Y\in\Omega^{0,1}_{\rmi}$ and $R\in\SO_2$ by the polar form theorem.
 Therefore, the composition with $[A]\mapsto A\rmi A^{-1}$
 \[
 \Omega^{0,1}_{\rmi}\lra\JJ
 \,,\quad\,
 Y\longmapsto\rme^Y\rmi\rme^{-Y}
 \,,
 \]
 is the diffeomorphism we wanted.
 
 We remark that the composition with $[A]\mapsto A\cdot\rmi$, where $A\cdot\rmi$ denotes
 the action of $\PSl_2(\R)$ on the upper half-plane by M\"obius transformations (which
 preserve the hyperbolic metric) is the exponential map $Y\mapsto \rme^Y\cdot\rmi$ of the
 hyperbolic upper half-plane at $\rmi$.
\end{rem}


\subsection{The Kodaira differential\label{subsec:kodairadiff}}

We continue the considerations from Section \ref{subsec:fredindofinfiact}.
For all $j\in\JJ$ the operator $j\CR j\co\Omega^0\ra\Omega^{0,1}_j$ is injective and
\[
H^1_j:=\Omega^{0,1}_j/\im(j\CR j)
\]
has real dimension $2\big(\#A-\#D\big)$.
Moreover, the operator $j\CR j$ is $\GG_j$-equivariant
by Remark \ref{rem:automphiandbdcommute}.
Alternatively, argue with the {\it theorema egregium}
and with Remark \ref{rem:jautoareallisometries} 
as done in Remark \ref{rem:incoisotropy}.
Therefore, conjugation
\[
\psi^*y:=T_{\psi}\psi^{-1}\circ y_{\psi}\circ T\psi
\]
by elements $\psi\in\GG_j$ leaves $\im(j\CR j)$ invariant,
so that $\GG_j$ induces an action on $H^1_j$
via $\psi^*[y]:=[\psi^*y]$.

Observe, that the complex structure $y\mapsto jy$ on $\Omega^{0,1}_j$
introduced in Section \ref{subsec:cayleytransform}
does {\bf not} descent to a complex structure on the quotient $H^1_j$.
The reason is that the subspace $\im(j\CR j)$ of $\Omega^{0,1}_j$ is {\bf not}
$j$-invariant as the operator $j\CR j\co\Omega^0\ra\Omega^{0,1}_j$ is {\bf not} complex
linear, see Section \ref{subsec:infinitesimalaction}.

Choose a $\GG_j$-invariant complementary subspace $E_j\subset\Omega^{0,1}_j$
of $\im(j\CR j)$ so that the quotient map $y\mapsto[y]$ of
$\Omega^{0,1}_j$ onto $H^1_j$ restricts to a
$\GG_j$-equivariant linear isomorphism $E_j\ra H^1_j$.

\begin{exwith}
 \label{ex:l2othogonal}
 A $\GG_j$-invariant complementary subspace $E_j$
 can be defined as the orthogonal complement of $\im(j\CR j)$
 w.r.t.\ to the $L^2$-inner product on $\Omega^{0,1}_j$
 induced by a $\GG_j$-invariant metric on $S$.
 Such a metric can be obtained
 via averaging any given metric over the finite set $\GG_j$.
 The obtained $\GG_j$-invariant $L^2$-inner product on $\Omega^{0,1}_j$
 can be symmetrised as in Remark \ref{rem:crvialc}
 so that the action of $j$ will be orthogonal in addition.
 Alternatively,
 a natural choice for $\GG_j$-invariant metric on $S$
 would be the following:
 The area form $\sigma_j$ and the metric $g_j$,
 both being $\GG_j$-invariant,
 together determine a Hodge star operator
 on $TS$-valued differential forms on $S$,
 namely, via $y\mapsto-y\circ j$.
 Restricting to elements $y\in\Omega^{0,1}_j$
 this yields $y\mapsto jy$
 so that we obtain a $\GG_j$-invariant $L^2$-inner product
 on $\Omega^{0,1}_j$ by
 \[
 \langle y_1,y_2\rangle_j:=
 \frac12\int_Sy_1\wedge jy_2
 \,.
 \]
 The wedge product of two $TS$-valued differential forms
 is defined component-wise w.r.t.\ local $g_j$-orthonormal frames.
 Using conformal coordinates one shows that
 the integrand $y_1\wedge jy_2$ equals
 $2\Re(y_1\circ y_2)\,\sigma_j$,
 where, point-wise, $\Re(y_1\circ y_2)$ denotes the real part
 of the complex eigenvalue
 of the $j$-complex linear map $y_1\circ y_2\in\Omega^1$
 between complex lines.
 This shows $j$- and $\GG_j$-invariance of the metric
 \[
 \langle y_1,y_2\rangle_j=
 \int_S\Re(y_1\circ y_2)\,\sigma_j
 \,,
 \]
 whose induced norm $\|y\|_j$ is given by the square-root of
 the integral of the point-wise eigenvalues
 of the $j$-complex linear endomorphism field $y\circ y$
 over $S$ against $\sigma_j$.
 Nevertheless,
 the orthogonal complement $E_j$ of $\im(j\CR j)$ in $\Omega^{0,1}_j$
 of any $j$-invariant metric will not be invariant under $j$
 as $\im(j\CR j)$ is not invariant under $j$.
 In Section \ref{subsec:skyscraperdeformation} we will construct a complex linear
 complement $E_j$.
\end{exwith}

Consider a so-called {\bf deformation} of $\big(S,j,D,\{1,\rmi,-1\},A\big)$, which is a map
\[
\mathfrak{j}\co
(V_j,0)\lra(\JJ,j)
\,,\qquad
y\longmapsto j(y)
\,,
\]
defined on an open neighbourhood $V_j\subset E_j$ of $0$.
If $\mathfrak{j}$ is an embedding whose image $\mathfrak{j}(V_j)$ is transverse to the
orbits of $\GG$, then the deformation $\mathfrak{j}$ is called {\bf effective}, cf.\ the
\cite[Definition 4.2.13]{wen10} of a {\bf local slice} through $j$.
Transversality can be expressed via the invertibility of the so-called {\bf Kodaira differential}
\[
[T_y\mathfrak{j}]\co E_j\lra\Omega^{0,1}_{j(y)}\lra H^1_{j(y)}
\]
for all $y\in V_j$, which is the composition of the linearisation $T_y\mathfrak{j}$ with the
quotient map $[\,.\,]$.
Furthermore
we call the deformation $\mathfrak{j}$ {\bf symmetric}
if $V_j$ is $\GG_j$-invariant
(e.g.\ taking metric balls about $0\in E_j$
w.r.t.\ to the $L^{\infty}$-norm $|\,.\,|_j$
from the beginning of Section \ref{subsec:cayleytransform}
or the $L^2$-norm $\|\,.\,\|_j$ induced by $\langle \,.\,,\,.\,\rangle_j$
from Example \ref{ex:l2othogonal})
and $\mathfrak{j}$ is $\GG_j$-equivariant.
The latter means that for all $y\in V_j$ and $\psi\in\GG_j$ we have that
$\psi^*\big(\mathfrak{j}(y)\big)=\mathfrak{j}(\psi^*y)$ so that
\[
\psi\co
\big(S,j(\psi^*y),D,\{1,\rmi,-1\},A\big)
\lra
\big(S,j(y),D,\{1,\rmi,-1\},A\big)
\]
is a holomorphic isomorphism.
Finally, the deformation $\mathfrak{j}$ is called {\bf complex}
provided that the complementary subspace $E_j\subset\Omega^{0,1}_j$
is invariant under $j$, i.e.\ is a complex linear subspace,
and that $\mathfrak{j}$ is holomorphic
in the sense that the differential $T\mathfrak{j}$ is complex linear,
i.e.\
$\mathfrak{j}(y)\circ T_y\mathfrak{j}=T_y\mathfrak{j}\circ j$
for all $y\in V_j\subset E_j$.

\begin{exwith}
\label{ex:symmetricandeffective}
 The Cayley transformation from Section \ref{subsec:cayleytransform}
 or the conjugation by the exponential map studied in Remark \ref{rem:viahomospace}
 yield symmetric effective deformations $\mathfrak{j}$ restricting
 \[
 y\longmapsto(1+jy)j(1+jy)^{-1}
 \quad\text{or}\quad
 y\longmapsto\rme^{jy}j\rme^{-jy}
 \]
 to $V_j=E_j\cap\Omega^{0,1}_j\cap B_1(0)$
 or to $V_j=E_j$, resp.
 Indeed, the differentials at $0\in V_j$ in direction of $\dot y\in E_j$ are given by
 $\dot y\mapsto2\dot y$.
 For complex linear $E_j$ the corresponding deformation will be complex.
\end{exwith}

\begin{rem}
 \label{rem:naturalityofthekd}
 The {\bf Kodaira differential is natural} in the following sense:
 Consider deformations $\mathfrak{j}\co V_j\ni y\mapsto j(y)$ and
 $\mathfrak{k}\co V_k\ni z\mapsto k(z)$ of $\big(S,j,D,\{1,\rmi,-1\},A\big)$ and
 $\big(S,k,D,\{1,\rmi,-1\},A\big)$, resp.
 Choose a linear isomorphism $\zeta\co E_j\ra E_k$, which we will write as
 $\zeta\co y\mapsto z(y)$.
 Let $y\mapsto\varphi(y)$ be a smooth family of diffeomorphisms
 \[
 \varphi(y)\co
 \big(S,j(y),D,\{1,\rmi,-1\},A\big)
 \lra
 \Big(S,k\big(z(y)\big),D,\{1,\rmi,-1\},A\Big)
 \]
 defined on $V_j\cap\zeta^{-1}(V_k)$ and deforming the equivalence
 \[
 \varphi=\varphi(0)\co
 \big(S,j,D,\{1,\rmi,-1\},A\big)
 \lra
 \big(S,k,D,\{1,\rmi,-1\},A\big)
 \,.
 \]
 Then the Kodaira differentials at $0$ satisfy
 \[
 [T_0\mathfrak{j}]=
 \varphi^*\circ[T_0\mathfrak{k}]\circ\zeta
 \,,
 \]
 cf.\ \cite[Proposition 1.6]{hwz-dm12}.
 In particular, $\mathfrak{j}$ is effective if and only if $\mathfrak{k}$ is.
 
 Indeed, choose $\dot y\in E_j$ and write $\mathfrak{j}_t=\mathfrak{j}(t\dot y)$, 
 $\varphi_t=\varphi(t\dot y)$, and $\mathfrak{k}_t=\mathfrak{k}\big(\zeta(t\dot y)\big)$ for 
 $t\in(-1,1)$ and take the Lie derivative of $\mathfrak{j}_t=\varphi_t^*\mathfrak{k}_t$.
 By conformal invariance of the Cauchy--Riemann operator
 $\varphi^*\circ k\CR k=j\CR j\circ\varphi^*$ (see Remark \ref{rem:automphiandbdcommute})
 we obtain
 \[
 T_0\mathfrak{j}(\dot y)=
 2j\CR j(\varphi^*X)+\varphi^*\Big(T_0\mathfrak{k}\big(\zeta(\dot y)\big)\Big)
 \]
 similarly to Remark \ref{rem:incoisotropy}, where $X$ is the velocity vector of the path
 $t\mapsto\varphi_t$ in $\GG$ at $0$.
 As the restriction of $\varphi_t$ to the special points in $D\cup\{1,\rmi,-1\}\cup A$ is
 constant by continuity of $t\mapsto\varphi_t$ we obtain that $X\in\Omega^0$.
 Hence, $\varphi^*X\in\Omega^0$ and therefore $2j\CR j(\varphi^*X)$ represents the zero class.
 The claim follows now because $\varphi$ defines a well--defined isomorphism
 $\varphi^*\co H^1_k\ra H^1_j$ via $\varphi^*[y]:=[\varphi^*y]$,
 \[
 \varphi^*y:=T_{\varphi}\varphi^{-1}\circ y_{\varphi}\circ T\varphi
 \,,
 \]
 by the previous argument.
 Hence,
 \[
 [T_0\mathfrak{j}(\dot y)]=
 \varphi^*\Big[T_0\mathfrak{k}\big(\zeta(\dot y)\big)\Big]
 \]
 for all $\dot y\in E_j$.
 
 If both deformations $\mathfrak{j}$ and $\mathfrak{k}$ are effective,
 then there is a {\it tautological} choice for an isomorphism $\zeta$
 for the given diffeomorphism $\varphi$, namely
 $\zeta=[T_0\mathfrak{k}]^{-1}\circ(\varphi^*)^{-1}\circ[T_0\mathfrak{j}]$.
 On the other hand, in the situation of Example \ref{ex:l2othogonal}, where the local slices are
 constructed via orthogonal complements, we get $\varphi^*E_k=E_j$ because $g_j$ and
 $\varphi^*g_k$ are conformally equivalent.
 This allows the choice $(\varphi^*)^{-1}\co E_j\ra E_k$ for $\zeta$ and yields
 \[
 [T_0\mathfrak{j}]=
 \varphi^*\circ[T_0\mathfrak{k}]\circ(\varphi^*)^{-1}
 \,.
 \]
\end{rem}


\subsection{Orbifold structure -- fixed stable nodal type\label{subsec:orbifoldstrfixdstabtype}}

In the situation of Section \ref{subsec:kodairadiff} we assume for the moment that the
isotropy group $\GG_j$ is trivial.
Invertibility of the Kodaira differential $[T_y\mathfrak{j}]$ implies that the linear subspace
$T_y\mathfrak{j}(E_j)$ in $\Omega^{0,1}_{j(y)}$ is complementary to $\im\big(j(y)\CR {j(y)}\big)$.
With the arguments from \cite[p.~634/5]{rs06} and \cite[Theorem 4.2.14]{wen10}
we obtain that
\[
\GG\times V_j\lra\JJ
\,,\qquad
(\varphi,y)\longmapsto\varphi^*\big(j(y)\big)
\,,
\]
is a diffeomorphism onto a neighbourhood of the $\GG$-orbit of $j=\id^*\!\big(j(0)\big)$
for a sufficiently small open neighbourhood $V_j\subset E_j$ of $0$.
In other words, the map $(\id,y)\mapsto j(y)$ induces a homeomorphism
\[
(\id,y)\longmapsto\big[S,j(y),D,\{1,\rmi,-1\},A\big]=\big[j(y)\big]
\]
defined on $V_j$ onto a neighbourhood of $\big[S,j,D,\{1,\rmi,-1\},A\big]=[j]$ in
$\RR_{\tau}=\JJ/\GG$.
The inverse serves as a chart of a smooth manifold structure on $\RR_{\tau}$.

For $\GG_j$ non-trivial we give an equivariant version of the above construction:
Assuming the situation of Section \ref{subsec:kodairadiff} we consider a symmetric
effective deformation $\mathfrak{j}\co(V_j,0)\ra(\JJ,j)$, which we also denote by
$y\mapsto j(y)$.
We would like to find a $\GG$-equivariant diffeomorphism
\[
\GG\times_{\GG_j} V_j\lra\JJ
\,,\qquad
(\varphi,y)\longmapsto\varphi^*\big(j(y)\big)
\,,
\]
onto a neighbourhood of the $\GG$-orbit of $j=\id^*\!\big(j(0)\big)$, where
$\GG\times_{\GG_j} V_j$ is the quotient of $\GG\times V_j$ by the action
$(\varphi,y)\mapsto\big(\psi^{-1}\circ\varphi,\psi^*y\big)$, $\psi\in\GG_j$.

In order to do so we would like to find an isomorphism
\[
\varphi\co
\big(S,j(y),D,\{1,\rmi,-1\},A\big)
\lra
\big(S,k,D,\{1,\rmi,-1\},A\big)
\]
for given $k\in\JJ$ sufficiently close to $j\in\JJ$ and for some $y\in V_j$.
Holomorphicity of $\varphi$ translates to $\varphi^*k=j(y)$.
In other words, $(\varphi,y,k)$ is a zero of the {\bf non-linear Cauchy-Riemann operator}
\[
F(\varphi,y,k):=\frac12\Big(T\varphi+k\circ T\varphi\circ j(y)\Big)
\,,
\]
which is a section into the bundle $\EE\ra\GG\times V_j\times\JJ$, whose fibre
$\EE_{(\varphi,y,k)}$ is the vector space of sections of the bundle of complex anti-linear
bundle homomorphisms from $\big(TS,j(y)\big)$ to $(TS,k)$.

Setting $F_k:=F(\,.\,,\,.\,,k)$ we will write the solutions $(\varphi,y)$ of $F_k(\varphi,y)=0$ as
a function of $k$ via the implicit function theorem.
The linearisation
\[
T_{(\id,0)}F_j:\Omega^0\times E_j\lra\Omega^{0,1}_j
\]
of $(\varphi,y)\mapsto F_j(\varphi,y)$ at $(\id,0)$ equals
\[
(X,\dot y)\longmapsto
\CR jX+\tfrac12j\cdot\big(T_0\mathfrak{j}\big)(\dot y)
\,.
\]
With Sections \ref{subsec:fredindofinfiact} and \ref{subsec:kodairadiff}
the operator
\[
-2j\cdot T_{(\id,0)}F_j=
-2j\CR j\oplus T_0\mathfrak{j}
\]
is an isomorphism.

The {\bf implicit function theorem} combined with an intermediate Sobolev completion and a
subsequent elliptic regularity argument as in \cite[p.~634/5]{rs06} or in
\cite[Theorem 4.2.14]{wen10} implies:
There exists an open neighbourhood $\KK\subset\JJ$ of $j$, a possibly smaller
$\GG_j$-invariant open neighbourhood $V_j\subset E_j$ of $0$, an open neighbourhood
$\HH\subset\GG$ of $\id$ such that the sets $\psi^*\HH$ are pair-wise disjoint for all
$\psi\in\GG_j$, and a unique map
\[
\Phi\co
\big(\KK,j\big)
\lra
\Big(\HH\times V_j,(\id,0)\Big)
\]
such that
\[
F_k(\varphi,y)=0
\quad\Longleftrightarrow\quad
(\varphi,y)=\Phi(k)
\,,
\]
whenever $(\varphi,y,k)\in\HH\times V_j\times\KK$.
Notice, that uniqueness implies $\Phi\big(j(y)\big)=(\id,y)$
for all $y\in V_j$.

For all $\psi\in\GG_j$ define a map
\[
\Phi_{\psi}\co
\big(\KK,j\big)
\lra\Big(\psi^*\HH\times V_j, (\psi,0)\Big)
\]
setting
\[
\Phi_{\psi}(k):=
\psi^*\big(\Phi(k)\big)
\,.
\]
Observe that $\Phi_{\psi}\big(j(y)\big)=\big(\psi,\psi^*y\big)$
for all $y\in V_j$.
Moreover, by symmetry of the deformation $\mathfrak{j}$, which reads as
$\psi^*\circ\mathfrak{j}=\mathfrak{j}\circ\psi^*$ for all $\psi\in\GG_j$, we get
$\psi^*\circ F_k=F_k\circ\psi^*$, and hence
$F_k\big(\Phi_{\psi}(k)\big)=0$ for all $k\in\JJ$ and $\psi\in\GG_j$.
In other words, for all $\psi\in\GG_j$ there exists a unique map
$\Phi_{\psi}$ with the above properties such that
\[
F_k(\varphi,y)=0
\quad\Longleftrightarrow\quad
(\varphi,y)=\Phi_{\psi}(k)
\,,
\]
whenever $(\varphi,y,k)\in\psi^*\HH\times V_j\times\KK$.

We get the following {\bf global uniqueness statement}:
There exists potentially smaller neighbourhoods $V_j$ and
$\KK$ such that for all solutions $(\varphi,y,k)\in\GG\times V_j\times\KK$
of $F(\varphi,y,k)=0$ there exists a unique $\psi\in\GG_j$ such that
$\varphi\in\psi^*\HH$ and $(\varphi,y)=\Phi_{\psi}(k)$.
This follows arguing by contradiction with properness of the action
$\GG\times\JJ\ra\JJ$, $(\phi,j)\mapsto\phi^*j$, see Remark \ref{rem:topologyfixedstabletype}.

Based on the current results of the implicit function theorem an {\bf orbifold chart} about
$j\in\JJ$ can be obtained as follows:
Let $\UU_j$ be the image of $\KK$ under the projection
$[\,.\,]\co\JJ\ra\RR_{\tau}=\JJ/\GG$, i.e.\
\[\UU_j=[\KK]\,,\]
so that, in particular, $\UU_j$ is open according to the
quotient topology described in Remark \ref{rem:topologyfixedstabletype}.
With help of $\Phi$ we find $\varphi\in\HH$ and $y\in V_j$ for each given $k\in\KK$ such that
$\varphi^*k=j(y)$.
Hence, $[k]=\big[j(y)\big]$, so that $\UU_j$ is the set of all isomorphism classes
$\big[S,j(y),D,\{1,\rmi,-1\},A\big]=\big[j(y)\big]$ with $y\in V_j$.

The isotropy group $\GG_j$ acts linearly on $V_j$ by conjugation. In view of the metric
obtained by restriction of the metric described in Example \ref{ex:l2othogonal} this action is
orthogonal.
Hence, the action is {\bf effective}, i.e.\ only for $\id\in\GG_j$ all points of $V_j$ are fixed points.
The map
\[
\mathfrak{p}_j\co V_j/\GG_j\lra\UU_j
\,,\qquad
[y]\longmapsto\big[j(y)\big]
\]
is well-defined by symmetry of the deformation $\mathfrak{j}$; $\mathfrak{p}_j$ is continuous
because the deformation $\mathfrak{j}$ is and the respective quotient maps are open and
continuous as explained in Remark \ref{rem:topologyfixedstabletype}.

We claim that
\[
\UU_j\lra V_j/\GG_j
\,,\qquad
[k]\longmapsto\big[\Phi^2(k)\big]
\,,
\] 
is the inverse map of $\mathfrak{p}_j$, where $\Phi^2(k)$ denotes the second component of
$\Phi(k)$.
First of all the map is well-defined by the following {\bf compatibility condition for uniformisers}:
Write $\big[j(y_1)\big]=[k]=\big[j(y_2)\big]$ for $y_1,y_2\in V_j$ and choose an isomorphism
\[
\phi\co
\big(S,j(y_1),D,\{1,\rmi,-1\},A\big)
\lra
\big(S,j(y_2),D,\{1,\rmi,-1\},A\big)
\,,
\]
whose existence is guaranteed by the definition of the equivalence relation.
We claim that
\[
\phi\in\GG_j
\quad
\text{and}
\quad
y_1=\phi^*y_2
\,.
\]
Indeed, we get $F\big(\phi,y_1,j(y_2)\big)=0$, so that we find a unique $\psi\in\GG_j$ such that
$\phi\in\psi^*\HH$ and $(\phi,y_1)=\Phi_{\phi}\big(j(y_2)\big)$ by the above global uniqueness
statement.
Because of $\Phi_{\phi}\big(j(y_2)\big)=(\psi,\psi^*y_2)$ we obtain $\phi=\psi\in\GG_j$ and
$y_1=\psi^*y_2$.
Being well-defined follows now with the equation $\Phi\big(j(y)\big)=(\id,y)$ for $y\in\{y_1,y_2\}$.
Similarly, in order to verify the two-sided inverse property we obtain
\[
[y]\longmapsto
\big[j(y)\big]\longmapsto
\Big[\Phi^2\big(j(y)\big)\Big]=[y]
\]
and, writing $\varphi^*k=j(y)$ for $y\in V_j$,
\[
[k]\longmapsto
\big[\Phi^2(k)\big]=[y]\longmapsto
\big[j(y)\big]=[k]
\,.
\]
Therefore, the assignment $[k]\mapsto\big[\Phi^2(k)\big]$ is the inverse map
$\mathfrak{p}_j^{-1}\co\UU_j\lra V_j/\GG_j$.
The inverse $\mathfrak{p}_j^{-1}$ is continuous as well.
This follows because $k\mapsto\Phi^2(k)$ is is continuous and the involved quotient maps
are open and continuous, cf.\ Remark \ref{rem:topologyfixedstabletype}.
Consequently, the $\GG_j$-invariant map
\[
[\mathfrak{j}]\co V_j\lra\UU_j
\]
induces a homeomorphism $\mathfrak{p}_j$ form $V_j/\GG_j$ onto $\UU_j$.
In other words, $\big(V_j,\GG_j,\mathfrak{p}_j^{-1}\big)$ is an {\bf orbifold chart for}
$\RR_{\tau}$ {\bf about} $[j]=\big[S,j,D,\{1,\rmi,-1\},A\big]$ and
$\big(E_j,\GG_j,V_j,\UU_j,[\mathfrak{j}]\big)$ is a $\tau$-{\bf uniformiser}
by definition.

\begin{rem}
 \label{rem:anisoofosotropygroups}
 By the above implicit function theorem we find $k$-holomorphic maps
 $\varphi\in\GG$ defined on $\big(S,j(y)\big)$, where $k\in\KK$ and $y\in V_j$, so that
 $(\varphi,y)$ is a solution of $F_k=0$.
 By global uniqueness and a potential precomposition with the inverse of an element in $\GG_j$
 making use of the symmetry property of $\mathfrak{j}$ we can assume that $\varphi\in\HH$
 and write $\Phi(k)=(\varphi,y)$.
 By local uniqueness and $\GG_j$-invariance of the solution set we obtain that $\{F_k=0\}$
 is equal to the set of all $\big(\psi^*\varphi,\psi^*y\big)$, $\psi\in\GG_j$.
 Hence, for $y=\Phi^2(k)$ the solution set $\big\{F_k(\,.\,,y)=0\big\}$ is given by
 $\{\psi^*\varphi\,|\,\psi\in\GG_{j,y}\}$, where $\GG_{j,y}$ denotes the {\bf stabiliser at} $y\in V_j$
 of the induced action by $\GG_j$ on $E_j$.
 For that we will also write 
 \[
 \GG_{j,y}:=\big(\GG_j|_{E_j}\big)_y
 \]
 Notice that $\GG_{j,0}=\GG_j$.
 
 For all $\hat{\psi}\in\GG_k$ we have that
 $F_k(\hat{\psi}\circ\varphi,y)=T\hat{\psi}\circ F_k(\varphi,y)$, so that the set
 $\big\{\hat{\psi}\circ\varphi\,|\,\hat{\psi}\in\GG_k\big\}$ is contained in
 $\{\psi^*\varphi\,|\,\psi\in\GG_{j,y}\}$.
 For the converse observe that $\varphi\circ\psi\circ\varphi^{-1}$, $\psi\in\GG_{j,y}$, is
 an element of $\GG_k$ and $\big(\varphi\circ\psi\circ\varphi^{-1}\big)\circ\varphi=\psi^*\varphi$.
 Therefore, we obtain two descriptions
 \[
 \Big\{\hat{\psi}\circ\varphi\,\,\big|\,\,\hat{\psi}\in\GG_k\Big\}=
 \big\{\psi^*\varphi\,\,|\,\,\psi\in\GG_{j,y}\big\}
 \]
 for the solution set $\big\{F_k(\,.\,,y)=0\big\}$ and
 \[
\GG_{j,y}\lra\GG_k
 \,,\qquad
 \psi\longmapsto\varphi\circ\psi\circ\varphi^{-1}
 \]
 is an {\bf isomorphism of isotropy groups} with inverse
 $\hat{\psi}\mapsto\varphi^{-1}\circ\hat{\psi}\circ\varphi$.
\end{rem}

In order to describe the {\bf transformation behaviour of orbifold charts} we
consider symmetric effective deformations $\mathfrak{j}\co V_j\ni y\mapsto j(y)$ and
$\mathfrak{k}\co V_k\ni z\mapsto k(z)$ of $\big(S,j,D,\{1,\rmi,-1\},A\big)$ and
$\big(S,k,D,\{1,\rmi,-1\},A\big)$, resp.
About the respective $\tau$-uniformiser $\big(E_j,\GG_j,V_j,\UU_j,[\mathfrak{j}]\big)$ and
$\big(E_k,\GG_k,V_k,\UU_k,[\mathfrak{k}]\big)$ we assume that $\UU_j\cap\UU_k=\emptyset$.
Hence, we find $\varphi\in\GG$ such that $\varphi^*k=j$.

We claim that we can assume that $j=j(0)=k(0)=k$.
Indeed, define a deformation $\mathfrak{k'}\co V_{k'}\ni z\mapsto k'(z)$,
\[
k'(z)=\varphi^*\Big(k\big((\varphi^{-1})^*z\big)\Big)
\,,
\]
of $\big(S,j,D,\{1,\rmi,-1\},A\big)$, whose domain is the subset $V_{k'}:=\varphi^*V_k$ of
the complementary space $E_{k'}:=\varphi^*E_k$.
For the latter use that $g_{k'}$ and $\varphi^*g_k$ are conformally equivalent.
By the naturality of the Kodaira differential we have
\[
[T_0\mathfrak{k'}]=
\varphi^*\circ[T_0\mathfrak{k}]\circ(\varphi^{-1})^*
\,,
\]
so that $\mathfrak{k'}$ is effective, see Remark \ref{rem:naturalityofthekd}.
In view of Remark \ref{rem:anisoofosotropygroups} the deformation $\mathfrak{k'}$ is
symmetric because for all $\psi\in\GG_k$, for which we have
$\psi^*\circ\mathfrak{k}=\mathfrak{k}\circ\psi^*$, it follows that
$\big(\varphi^{-1}\circ\psi\circ\varphi\big)^*\circ\mathfrak{k'}=
\mathfrak{k'}\circ\big(\varphi^{-1}\circ\psi\circ\varphi\big)^*$.

Therefore, we consider deformations $\mathfrak{j}$ and $\mathfrak{k}$ such that
$j=j(0)=k(0)=k$.
In view of the implicit function theorem above there exists a $k(z)$-holomorphic map
$\varphi(z)\in\GG$ close to $\id\in\GG_j$ defined on $\big(S,j\big(y(z)\big)\big)$ via the locally
unique smooth map $\Phi\big(k(z)\big)=\big(\varphi(z),y(z)\big)$ such that $\Phi(j)=(\id,0)$.
Switching the roles results into a $j(y)$-holomorphic map $\hat{\varphi}(y)\in\GG$ close to
$\id\in\GG_j$ defined on $\big(S,k\big(z(y)\big)\big)$ via the locally unique smooth map
$\Phi\big(j(y)\big)=\big(\hat{\varphi}(y),z(y)\big)$ such that $\Phi(j)=(\id,0)$.
Comparing both solutions using uniqueness yields
\[
\hat{\varphi}\big(y(z)\big)=\big(\varphi(z)\big)^{-1}
\]
as well as $z=z\big(y(z)\big)$ and $y=y\big(z(y)\big)$.
Therefore, after shrinking the domains according to the implicit function theorem if necessary,
which results into $\UU_j=\UU_k$, we obtain maps
\[
V_j\lra V_k
\,,\quad
y\longmapsto z(y)
\qquad\text{and}\qquad
V_k\lra V_j
\,,\quad
z\longmapsto y(z)
\,,
\]
which are smooth and inverse to each other such that
$[\mathfrak{k}]\circ\big(y\mapsto z(y)\big)=\hat{\varphi}(y)^*[\mathfrak{j}]$.
This results into a coordinate change of an orbifold structure because the construction is done
in a $\GG_j$-equivariant fashion:
For that use Remark \ref{rem:anisoofosotropygroups} and observe that the solution set
$\big\{F_{k(z)}(\,.\,,y)=0\big\}$ is given by 
\[
\Big\{\hat{\psi}\circ\varphi(z)\,\,\big|\,\,\hat{\psi}\in\GG_{k(z)}\Big\}=
\big\{\psi^*\varphi(z)\,\,|\,\,\psi\in\GG_{j,y}\big\}
\,.
\]

As for manifolds we obtain:

\begin{prop}
 \label{prop:rtauisanorbifold}
 The above constructed orbifold charts provide the nodal Riemann moduli space $\RR_{\tau}$
 for the stable nodal type $\tau$ with the structure of an orbifold of real dimension
 $2\big(\#A-\#D\big)$, whose isotropy groups at $[j]\in\RR_{\tau}$ are given by $\GG_j$ up to
 conjugation.
\end{prop}

Referring to the current situation we define
\[
\mathbf{T}_{j,k}:=
\Big\{
 (\varphi,y,z)
 \,\,\big|\,\,
 F_{k(z)}(\varphi,y)=0
\Big\}
\subset\GG\times V_j\times V_k
\]
provided with the subspace topology and call the projection $s\co\mathbf{T}_{j,k}\ra V_j$ onto $V_j$ the {\bf source map}; the the projection $t\co\mathbf{T}_{j,k}\ra V_k$ onto $V_k$
the {\bf target map}.
These maps come with inverses $y\mapsto\big(\hat{\varphi}^{-1}(y),y,z(y)\big)$ and
$z\mapsto\big(\varphi(z),y(z),z\big)$, resp., where
$\hat{\varphi}\big(y(z)\big)=\big(\varphi(z)\big)^{-1}$
as above.
Hence, $s$ and $t$ are homeomorphisms
providing $\mathbf{T}_{j,k}$ with the structure of a smooth manifold of dimension
$2\big(\#A-\#D\big)$, and $t\circ s^{-1}$ and $s\circ t^{-1}$ correspond to the above transition
maps $y\mapsto z(y)$ and $z\mapsto y(z)$, resp.
By the properness argument in Remark \ref{rem:topologyfixedstabletype} the map
$s\times t\co\mathbf{T}_{j,k}\ra V_j\times V_k$ is proper.

In other words,
we obtain an {\bf \'etale proper Lie groupoid}
$(R_{\tau},\mathbf{R}_{\tau})$,
which means the following:
Take a sequence of $\tau$-uniformisers
$\big(E_{j_{\nu}},\GG_{j_{\nu}},V_{j_{\nu}},\UU_{j_{\nu}},[\mathfrak{j}_{\nu}]\big)$ such that
$\bigcup_{\nu}\UU_{j_{\nu}}$ covers $\RR_{\tau}$.
The {\bf objects} are given by
\[
R_{\tau}:=\bigsqcup_{\nu}V_{j_{\nu}}
\]
and the {\bf morphisms} are
\[
\mathbf{R}_{\tau}:=\bigsqcup_{\nu,\mu}\mathbf{T}_{j_{\nu},j_{\mu}}\,.
\]
Morphisms can be composed whenever the corresponding target and source coincide.
The resulting composition is smooth, has a unit and each morphism admits a smooth inverse.
Furthermore all mentioned structure maps are smooth.
The nodal Riemann moduli space $\RR_{\tau}$ for the stable nodal type $\tau$ appears now
as {\bf orbit space}
\[\RR_{\tau}=R_{\tau}/\!\sim\,,\]
where two objects between which there exists a morphism are considered to be equivalent.


\subsection{Skyscraper deformation\label{subsec:skyscraperdeformation}}

A symmetric effective deformation $\mathfrak{j}$ is called a {\bf skyscraper deformation}
if there exists a $\GG_j$-invariant neighbourhood $U\subset S$ of $\partial S$ together with
the special points $|D|\cup\{1,\rmi,-1\}\cup A$ on which the {\bf deformation is stationary},
i.e.\ if $j(y)=j$ restricted to $U$ for all $y\in V_j$.
In view of the examples of symmetric effective deformations in Section
\ref{subsec:cayleytransform} and Remark \ref{rem:viahomospace} skyscraper deformations
can be obtained by restriction of symmetric effective deformations to a complementary
$\GG_j$-invariant vector space $E_j$ of $\im(j\CR j)$ whose elements vanish on $U$.
 
In order to construct such a vector space $E_j$ we denote by $\XX$ the space of all
smooth vector fields on $S$ that are tangent to the boundary along $\partial S$ and
admit $3$ zeros on each connected component of $S$ equal to special points in
$|D|\cup\{1,\rmi,-1\}\cup A$; on the disc component the zeros are required to be
$1,\rmi,-1$.
The operator $j\CR j$ restricted to $\XX$ induces an isomorphism $\XX\ra\Omega^{0,1}_j$.
 
We begin with the un-isotropic situation $\GG_j=\{\id\}$.
In order to construct a complement of $\Omega^0$ in $\XX$ we write $z_k$ for the elements
of $|D|\cup A$ and choose local holomorphic charts $(\C,\rmi)$ for $(S,j)$ about the special
points $z_k$.
We require that the chart domains are mutually disjoint and contained in
$S\setminus\partial S$.
Let $f_k$ be smooth cut off functions on $S$ that have their supports in the interior
of $r$-disc neighbourhoods about $z_k$ w.r.t.\ $g_j$ contained in the chosen
chart domains; the $f_k$ are required to by constantly $1$ on the $r/2$-disc
neighbourhoods about $z_k$.
Given $X\in\XX$ we define vector fields $X_k$ on $S$.
We require that the $X_k$ are given by $X(z_k)f_k$ in the chosen charts extended by zero
to $S$.
Observe that the $X_k$ vanish for special points that correspond to the zeros defining $\XX$.
Moreover, the $X_k$ are holomorphic on the $r/2$-discs.

Let $P\co\XX\ra\XX$ be the projector, i.e.\ $P^2=P$, given by
 \[
 P(X):=X-\sum_kX_k
 \,.
 \]
Observe that $P$ restricts to the identity on $P(\XX)=\Omega^0$.
The desired complement of $\Omega^0$ in $\XX$ is $(1-P)(\XX)$ as $1-P$ is a projector
as well.
The elements of $j\CR j(1-P)(\XX)$ vanish on a neighbourhood of all special points
$|D|\cup\{1,\rmi,-1\}\cup A$ and of $\partial S$.
Moreover, the dimension of $(1-P)(\XX)$ equals
\[
2\big(\#A-\#D\big)
\]
by the result of the computation of Section \ref{subsec:fredindofinfiact}
multiplied by $-1$.

Finally we set $E_j:=j\CR j(1-P)(\XX)$.
The elements of $E_j$ vanish on a neighbourhood of the union of the special points
$|D|\cup\{1,\rmi,-1\}\cup A$ and of the boundary $\partial S$.
Furthermore the isomorphism $j\CR j\co\XX\ra\Omega^{0,1}_j$ sends the splitting
$(1-P)(\XX)\oplus\Omega^0$ of $\XX$ to the splitting $E_j\oplus\im(j\CR j)$ of
$\Omega^{0,1}_j$.

We treat the case of a non-vanishing isotropy group $\GG_j$, which acts by permutations
on $\{z_k\}=|D|\cup A$.
It suffices to change the above projector $P$ by replacing the vector fields $X_k$ by
$\GG_j$-invariant vector fields $\hat{X}_k$.
For that denote by $B_r(z_k)$, $r>0$, the interior of the above $r$-discs.
Observe that the disjoint union of the $B_r(z_k)$ is $\GG_j$-invariant as $\GG_j$ acts by
isometries on $(S,g_j)$.

For each $z_k$ we assign a {\bf partner point} $w_k\in B_{r/2}(z_k)\setminus\{z_k\}$
requiring that the $\psi(w_k)$ are pair-wise distinct for all $\psi\in\GG_j$ and all $k$.
Notice, that the distance between $z_k$ and its partner $w_k$ is $\GG_j$-invariant for
all $k$.
We choose $\varepsilon>0$ such that
$B_{\varepsilon}(w_k)\subset B_{r/2}(z_k)\setminus\{z_k\}$ for all $k$.
Furthermore, we reqiure that the $\psi\big(B_{\varepsilon}(w_k)\big)$ are pair-wise
disjoint for all $\psi\in\GG_j$ and for all $k$.
Modify the cut off functions $f_k$ so that $f_k$ has support in
\[
B_r(z_k)\,\,\setminus\quad
\bigsqcup_{\ell\neq k\,\,\text{and}\,\,\psi\in\GG_j}
\psi\big(\overline{B_{\varepsilon/2}(w_{\ell})}\big)
\]
and is equal to $1$ on
\[
B_{r/2}(z_k)\,\,\setminus\quad
\bigsqcup_{\ell\neq k\,\,\text{and}\,\,\psi\in\GG_j}
\psi\big(\overline{B_{\varepsilon}(w_{\ell})}\big)
\,.
\]
In particular, for all $k$, we get $f_k(w_k)=1$ and $f_k\big(\psi(w_{\ell})\big)=0$ for all
$\ell\neq k$ and $\psi\in\GG_j$.
With the cut off functions $f_k$ modified we define $X_k$ for given $X\in\XX$ as in the
un-isotropic case.

We define the symmetrisations via
\[
\hat{X}_k:=\sum_{\psi\in\GG_j}\psi^*X_k
\,.
\]
We have $\phi^*\hat{X}_k=\hat{X}_k$ for all $\phi\in\GG_j$ and for all $k$ because $\GG_j$
acts on itself via composition permuting $\GG_j$.
The $\hat{X}_k$ that are assigned to the zeros $z_k$ of $\XX$ vanish; the remaining
$\hat{X}_k$ span a $2\big(\#A-\#D\big)$-dimensional vector space because
\[
\hat{X}_k(w_{\ell})=X_{\ell}(w_{\ell})
\]
for all $k,\ell$.
A basis can be obtained by taking $X\in\XX$ with $X(z_k)$ non-zero, so that the corresponding
partners $X_k(w_k)$ do not vanish.

\begin{rem}
\label{rem:complexstructure}
 Observe that the elements of $(1-P)(\XX)$,
 which are linear combinations of the vector fields $\hat{X}_k$
 constructed above, are vector fields on $S$
 that vanish on the boundary $\partial S$.
 Therefore, the complex structure $j$ on $S$
 preserves $(1-P)(\XX)$ and defines a {\bf complex structure} on
 $E_j=\CR j(1-P)(\XX)$ as $\CR j$ commutes with $j$,
 so that $E_j$ is a complex vector space of complex dimension $\#A-\#D$.
 Consequently taking the complex deformations
 form Example \ref{ex:symmetricandeffective} w.r.t.\ $E_j$
 yields holomorphic skyscraper deformations.
\end{rem}

\begin{rem}
\label{rem:smalldiscstrandbasisoftop}
 In the above construction the radii $r_k$ of the discs
 \[
 D_{r_k/2}(z_k):=\overline{B_{r_k/2}(z_k)}
 \]
 are necessarily constant on the orbits of the $\GG_j$-action on the points $z_k\in |D|\cup A$
 because $\GG_j$ acts on the discs $D_{r_k/2}(z_k)$ by isometries of $(S,g_j)$; but the radii
 are allowed to vary on distinct orbits $\GG_jz_k$.
 For a selection of orbit-wise constant radii $r_k$ denoted by $\mathbf{r}$ and the disjoint union
 \[
 \mathbf{D}_{j,\mathbf{r}}:=\bigsqcup_{z_k\in |D|\cup A}D_{r_k/2}(z_k)
 \]
 a skyscraper deformation $\mathfrak{j}$ that is stationary on $\mathbf{D}_{j,\mathbf{r}}$ can be
 constructed by the above arguments.
 Given a neighbourhood $U$ of $|D|\cup A$ one can choose $\mathbf{r}$ so small such that
 $\mathbf{D}_{j,\mathbf{r}}\subset U$.
 This yields an example of a {\bf small disc structure} $\mathbf{D}_j$, which by definition is a
 $\GG_j$-invariant disjoint union of discs $D_z$, $z\in|D|\cup A$, contained in a given
 neighbourhood of $|D|\cup A$ such that $z\in D_z$ for all $z\in|D|\cup A$.
 Furthermore the $D_z$ are the image of a smooth embedding of the closed unit disc $\D$ into
 $S$.
 Observe that a $k$-holomorphic map $\varphi\in\GG$ defined on $(S,j)$ sends
 $\mathbf{D}_{j,\mathbf{r}}$ diffeomorphically onto a small disc structure $\mathbf{D}_k$ on
 which the $\varphi$-push-forward skyscraper deformation $\mathfrak{k}$ of $\mathfrak{j}$ is
 stationary.
 
 Using small disc structures $\mathbf{D}_j$ orbifold charts
 $\big(V_j,\GG_j,\mathfrak{p}_j^{-1}\big)$ and $\tau$-uniformiser
 $\big(E_j,\GG_j,V_j,\UU_j,[\mathfrak{j}]\big)$ for $\RR_{\tau}$ about
 $[j]=\big[S,j,D,\{1,\rmi,-1\},A\big]$ can be constructed as in Section
 \ref{subsec:orbifoldstrfixdstabtype} using skyscraper deformations
 $\mathfrak{j}$ that are stationary on $\mathbf{D}_j$ exclusively.
 The transformation behaviour encoded in the $\mathbf{T}_{j,k}$ is compatible with skyscraper
 deformations which are stationary on disc structures.
 Correspondingly, a neighbourhood base of the topology on $\RR_{\tau}=\JJ/\GG$ described
 in Remark \ref{rem:topologyfixedstabletype} can be given by the family of subsets
 whose elements $[k]$ can be represented by complex structures $k$ that belong to an open
 subset of $\JJ$ and that satisfy $k=j$ restricted to some small disc structure
 $\mathbf{D}_j$.
 This follows with the implicit function theorem formulated in Section
 \ref{subsec:orbifoldstrfixdstabtype}.
\end{rem}

\begin{rem}
\label{rem:complexsubatlas}
 In view of the proceeding Remarks \ref {rem:complexstructure} and
 \ref {rem:smalldiscstrandbasisoftop}
 the orbifold structure on $\RR_{\tau}$ ensured in Proposition \ref{prop:rtauisanorbifold}
 and the subsequently described \'etale proper Lie groupoid structure
 admit subatlases generated by complex skyscraper deformations, so that
 the respective substructures are complex.
 For that one needs to verify that the transition maps
 $t\circ s^{-1}$ and $s\circ t^{-1}$ are holomorphic.
 In terms of complex skyscraper deformations
 $\mathfrak{j}\co V_j\ni y\mapsto j(y)$ and
 $\mathfrak{k}\co V_k\ni z\mapsto k(z)$
 of $\big(S,j,D,\{1,\rmi,-1\},A\big)$ and
 $\big(S,k,D,\{1,\rmi,-1\},A\big)$, resp.,
 such that there exists $\varphi\in\GG$ with $\varphi^*k=j$
 the transition maps are given by $y\mapsto z(y)$ and $z\mapsto y(z)$, resp.
 With the description before Proposition \ref{prop:rtauisanorbifold} we get
 $[\mathfrak{k}]\circ\big(y\mapsto z(y)\big)=\hat{\varphi}(y)^*[\mathfrak{j}]$ and 
 $[\mathfrak{j}]\circ\big(z\mapsto y(z)\big)=\varphi(z)^*[\mathfrak{k}]$ for smooth maps
 $V_k\ra\GG$, $z\mapsto\varphi(z)$, and $V_j\ra\GG$, $y\mapsto\hat{\varphi}(y)$, with
 $\varphi(0)=\varphi$ and $\hat{\varphi}(0)=\varphi^{-1}$.
 By symmetry it will be sufficient to verify holomorphicity in the first situation:
 Taking the derivative w.r.t.\ $y\in V_j$ in $k\big(z(y)\big)=\hat{\varphi}(y)^*j(y)$
 we obtain
 \[
 T_{z(y)}\mathfrak{k}\circ T_yz(\dot y)=
 2 k\big(z(y)\big)\cdot
 \CR {k(z(y))}\Big(\hat{\varphi}(y)^*\big(T_y \hat{\varphi}(\dot y)\big)\Big)+
 \hat{\varphi}(y)^*\big(T_y\mathfrak{j}(\dot y)\big)
 \]
 as in Remark \ref{rem:naturalityofthekd}.
 Replacing $\dot y$ by $j\dot y$ and composing with $-k\big(z(y)\big)$ from the left
 yields
 \[
 -T_{z(y)}\mathfrak{k}\circ k\circ T_yz(j\dot y)=
 2 \cdot
 \CR {k(z(y))}\Big(\hat{\varphi}(y)^*\big(T_y \hat{\varphi}(j\dot y)\big)\Big)+
 \hat{\varphi}(y)^*\big(T_y\mathfrak{j}(\dot y)\big)
 \,.
 \]
 The second summand on the right stays the same
 because $k\big(z(y)\big)=\hat{\varphi}(y)^*j(y)$
 and $\mathfrak{j}(y)\circ T_y\mathfrak{j}=T_y\mathfrak{j}\circ j$
 by complexitiy, see Remark \ref{rem:complexstructure}.
 Similarly,
 to deal with the left hand side use
 $\mathfrak{k}(z)\circ T_z\mathfrak{k}=T_z\mathfrak{k}\circ k$.
 On the right hand side, the first summend is an element in
 \[
 \im
 \Big(
 k\big(z(y)\big)\cdot
 \CR {k(z(y))}
 \Big)
 \]
 because $j\dot y$ vanishes along the boundary $\partial S$,
 so that the vector field $\hat{\varphi}(y)^*\big(T_y \hat{\varphi}(j\dot y)\big)$
 vanishes along $\partial S$ as well,
 and because on boundary vanishing vector fields on $(S,\partial S)$
 the Cauchy--Riemann operator is complex linear,
 cf.\ Section \ref{subsec:infinitesimalaction}.
 Effectivity of $\mathfrak{k}$ yields the algebraic splitting
 \[
 \Omega^{0,1}_{k(z(y))}=
 \,T_{z(y)}\mathfrak{k}\,(E_k)
 \,\oplus\,
 \im
 \Big(
 k\big(z(y)\big)\cdot
 \CR {k(z(y))}
 \Big)
 \,.
 \]
 Modding out the contributions to the second summand the above two equations compare to
 \[
 T_{z(y)}\mathfrak{k}\circ T_yz(\dot y)=
 -T_{z(y)}\mathfrak{k}\circ k\circ T_yz(j\dot y)
 \,.
 \]
 As $T_{z(y)}\mathfrak{k}$ is injective this yields
 \[
 T_yz(\dot y)=-k\circ T_yz(j\dot y)
 \,,
 \]
 i.e.\ $k\circ T_yz= T_yz\circ j$ meaning that $y\mapsto z(y)$ is holomorphic.
\end{rem}

Consequently, we obtain a complex version of Proposition \ref{prop:rtauisanorbifold}, so
that, in particular, $\RR_{\tau}$ is orientable.

\begin{prop}
 \label{prop:rtauisacomplexorbifold}
 The nodal Riemann moduli space $\RR_{\tau}$
 admits the structure of a complex orbifold
 of complex dimension $\#A-\#D$.
\end{prop}

\begin{rem}
 \label{rem:fixinfcombdataandunifgivesthatalso}
 In order to derive Proposition \ref{prop:rtauisacomplexorbifold} we fixed in Section
 \ref{subsec:groupoidasorbitspace} the combinatorial data $\big(S,D,\{m_0,m_1,m_2\},A\big)$
 to represent stable nodal marked discs $\big[S,j,D,\{m_0,m_1,m_2\},A\big]$ and
 used deformations of $j$.
 A direct way to obtain a complex orbifold structure would be to change the roles.
 Apply uniformisation as in Remark \ref{rem:isotropyisfinite} in order to represent the classes
 $\big[S,j,D,\{m_0,m_1,m_2\},A\big]$ by nodal discs whose disc component equals
 $\big(\D,\rmi,\{1,\rmi,-1\}\big)$ and whose sphere components are given by $(\C P^1,\rmi)$.
 The complex orbifold structure can be read off from variations of the configurations of the
 nodal points $D$ and the marked points $A$ as such, resp.
\end{rem}


\subsection{Varying the stable nodal type via desingularisation\label{subsec:varstabnodtypeviadesing}}

For a complex number $a$ of modulus $|a|\leq1$ we consider the intersection of the planar
algebraic curve $\{zw=a\}\subset\C\times\C$ with the polydisc $\D\times\D$.
For $a=0$ this curve is the union of the discs $\{z=0\}=\{0\}\times\D$ and
$\{w=0\}=\D\times\{0\}$ that intersect in the singularity of the curve.
For $a\neq0$ the equation $zw=a$ can be solved by $w=a/z$ so that we obtain a cylinder
that has no singularities:
The restriction of the projection $(z,w)\mapsto z$ to the curve $\{zw=a\}$ yields a
biholomorphism onto the annulus $\{|a|\leq|z|\leq1\}$ in the first coordinate plane.
Interchanging $z$ and $w$ yields a biholomorphism onto $\{|a|\leq|w|\leq1\}$.
Both biholomorphisms constitute holomorphic charts of $\{zw=a\}$.
The transition map from the first annulus to the second is
\[z\longmapsto\frac{a}{z}\,.\]
Taking {\bf positive} and {\bf negative holomorphic polar coordinates} $(z,w)\mapsto -\ln z$
and $(z,w)\mapsto \ln w$, resp., i.e.\ writing
\[
z=\rme^{-(s+\rmi t)}
\qquad\text{and}\qquad
w=\rme^{u+\rmi v}
\,,
\]
the transition map gets
\[
\big[0,-\ln |a|\big]\times S^1\lra\big[\ln |a|,0\big]\times S^1
\,,\qquad
(s,t)\longmapsto\big(s+\ln |a|,t+\arg a\big)
\,,
\]
where $S^1=\partial\D$.
For the complex logarithm we use the main branch.

Observe that rotations $z\mapsto\rme^{-\rmi\theta_+}z$ and $w\mapsto\rme^{\rmi\theta_-}w$
for $\theta_+,\theta_-\in S^1$ of the coordinate planes, which correspond to
\[
(s,t)\longmapsto\big(s,t+\theta_+\big)
\qquad\text{and}\qquad
(s,t)\longmapsto\big(s,t+\theta_-\big)
\]
w.r.t.\ positive and negative holomorphic polar coordinates, resp., result into a change of the
defining equation to
\[
zw=\rme^{-\rmi(\theta_+-\theta_-)}a
\]
as the pull back along $(z,w)\mapsto\big(\rme^{-\rmi\theta_+}z,\rme^{\rmi\theta_-}w\big)$
yields $zw=a$.
The coresponding transition map is
\[
(s,t)\longmapsto\Big(s+\ln |a|,t+\arg a-\big(\theta_+-\theta_-\big)\Big)
\,.
\]
A switch of the coordinates $(z,w)\mapsto (w,z)$ does not effect the
proceeding consideration.

Given $\big[S,j,D,\{1,\rmi,-1\},A\big]\in\RR_{\tau}$ we describe a similar desingularisation about
a nodal pair $\{z_0,w_0\}\in D$ in terms of {\bf parametrised connected sum}.
Choose a small disc structure $\mathbf{D}_j$ on $\big(S,j,D,\{1,\rmi,-1\},A\big)$.
Denote the corresponding discs about the nodal points $z_0,w_0\in|D|$ by $D_{z_0}$ and
$D_{w_0}$, resp., and choose boundary points $z_{\partial}\in\partial D_{z_0}$ and
$w_{\partial}\in\partial D_{w_0}$.
We call the pair $\{z_{\partial},w_{\partial}\}$ a {\bf decoration} of the nodal pair $\{z_0,w_0\}$.
By \cite[Theorem C.5.1]{mcsa04} there exists unique biholomorphic identifications of
$\big((D_{z_0},z_0,z_{\partial}),j\big)$ and $\big((D_{w_0},w_0,w_{\partial}),j\big)$, resp., with
$\big((\D,0,1),\rmi\big)$.

For given {\bf gluing parameter} $a\in\D$, $a\neq0$, replace $-\ln|a|$ by the {\bf modulus}
\[
R=\rme^{1/|a|}-\rme
\]
in the discussion about the planar algebraic curve $\{zw=a\}$.
Identify the first annulus $\{\rme^{-R}\leq|z|\leq1\}$ with the second $\{\rme^{-R}\leq|w|\leq1\}$
via the transition map
\[
z\longmapsto\frac{\rme^{-R+\rmi\arg(a)}}{z}
\,,
\]
which w.r.t.\ positive and negative holomorphic polar coordinates reads as
\[
[0,R]\times S^1\lra[-R,0]\times S^1
\,,\qquad
(s,t)\longmapsto\big(s-R,t+\arg a\big)
\,.
\]
We obtain a surface $S_a$ from $S\setminus\big(\Int(D_{z_0})\cup\Int(D_{w_0})\big)$ by gluing
the finite cylinder $Z_a:=[0,R]\times S^1$, which is identified with $[-R,0]\times S^1$ via the
above transition map, along the respective boundary circles via the restrictions of the
biholomorphic identifications of $D_{z_0}$ and $D_{w_0}$, resp, with $\D$.

The construction of the surface $S_a$ defines a complex structure $j_a$ that coincides with $j$
on $S\setminus\big(\Int(D_{z_0})\cup\Int(D_{w_0})\big)$ and with $\rmi$ on the cylinder $Z_a$
of modulus $R$.
This results into an element $\big[S_a,j_a,D_a,\{1,\rmi,-1\},A\big]$ of $\RR_{\tau'}$ with stable
nodal type $\tau'$, which necessarily differs from $\tau$.
The respective special points are given by
\[
D_a:=D\setminus\big\{\{z_0,w_0\}\big\}
\,,
\]
$\{1,\rmi,-1\}$, and $A$ under the inclusion of
$S\setminus\big(\Int(D_{z_0})\cup\Int(D_{w_0})\big)$ into $S_a$.

A change of biholomorphic identifications of $D_{z_0}$ and $D_{w_0}$ with $\D$ is given by a
rotation of the boundary points $z_{\partial}$ and $w_{\partial}$, resp., which in coordinates
reads as $z\mapsto\rme^{-\rmi\theta_+}z$ and $w\mapsto\rme^{\rmi\theta_-}w$, say.
Gluing with the rotated identifications yields a biholomorphic copy
$\big(S_b,j_b,D_b,\{1,\rmi,-1\},A\big)$ of $\big(S_a,j_a,D_a,\{1,\rmi,-1\},A\big)$, where
\[
b=\rme^{-\rmi(\theta_+-\theta_-)}a
\,.
\]
To obtain a biholomorphic map take the identity map on
$S\setminus\big(\Int(D_{z_0})\cup\Int(D_{w_0})\big)$ and the rotated transition map
$Z_a\ra Z_b$ given by
\[
(s,t)\longmapsto\Big(s-R,t+\arg a-\big(\theta_+-\theta_-\big)\Big)
\]
on $Z_a$.

Such rotations naturally appear when a automorphism $\psi\in\GG_j$ for which the nodal points
$z_0$ and $w_0$ are fixed-points, i.e.\ $\psi(z_0)=z_0$ and $\psi(w_0)=w_0$, acts on $S$.
Indeed, $\psi$ preserves the complement of $D_{z_0}\cup D_{w_0}$ in $S$ and induces
rotations on $D_{z_0}\cup D_{w_0}$.
The rotations are measured by the change of decorations from $\{z_{\partial},w_{\partial}\}$
to $\psi\big(\{z_{\partial},w_{\partial}\}\big)$ in terms of angles $-\theta_+$ and $\theta_-$, say.
Therefore, we get a holomorphic diffeomorphism
\[
\big(S_a,j_a,D_a,\{1,\rmi,-1\},A\big)
\lra
\big(S_b,j_b,D_b,\{1,\rmi,-1\},A\big)
\]
as above which this time coincides with $\psi$ on
$S\setminus\big(\Int(D_{z_0})\cup\Int(D_{w_0})\big)$.

We denote by
\[
\D^D
\]
the set of all maps from the set of nodal points $D$ to the set $\D$ of complex numbers of
modulus less than or equal to $1$.
Choose a small disc structure $\mathbf{D}_j$ on $\big(S,j,D,\{1,\rmi,-1\},A\big)$
together with a decoration for each disc in $\mathbf{D}_j$.
The choice of decorations determine holomorphic diffeomorphisms of all discs of the disc
structure with $\D$ such that the nodal point is mapped to $0\in\D$ and the decoration to
$1\in\D$.
Given $\bfa\in\D^D$ we perform the described parametrised connected sum about each nodal pair
$\{z,w\}\in D$ with gluing parameter
\[
a_{\{z,w\}}:=\bfa\big(\{z,w\}\big)
\,.
\]
This is done by replacing the node $\{z,w\}$ with the cylinder $Z_{a_{\{z,w\}}}^{\{z,w\}}$.
In the case of a vanishing gluing parameter $a_{\{z,w\}}$ formally
\[
Z_0^{\{z,w\}}:=D_z\sqcup D_w
\]
is given by the disjoint union of half-infinite cylinders $[0,\infty)\times S^1$ and
$(-\infty,0]\times S^1$ of {\bf infinite modulus} after removing the nodal points $z$ and $w$.
In other words, if $a_{\{z,w\}}=0$ we do nothing and keep the nodal pair $\{z,w\}\in D_{\bfa}$,
so that $D_{\bfa}$ arises from $D$ by removing all nodal pairs $\{z,w\}$ with $a_{\{z,w\}}\neq0$.
The resulting surface is denoted by
\[
\big(S_{\bfa},j_{\bfa},D_{\bfa},\{1,\rmi,-1\},A\big)
\,.
\]
Starting off with a skyscraper deformation $V_j\ni y\mapsto j(y)$ and
a small disc structure of sufficiently small discs
we will get
\[
\big(S_{\bfa},j(y)_{\bfa},D_{\bfa},\{1,\rmi,-1\},A\big)
\]
by the same construction.

In order to describe the effect of the $\GG_j$-action of $\big(S,j,D,\{1,\rmi,-1\},A\big)$ on the
desingularisation we denote by $\kappa_{z,w}$, $\{z,w\}\in D$, the complex anti-linear map
$T_wS\ra T_zS$ that conjugated with the linearisations of the biholomorphic identifications of
the discs $D_w$ and $D_z$ with $\D$ is equal to the complex conjugation map
$x+\rmi y\mapsto x-\rmi y$ on $\C$.
Interchanging the role of $z$ and $w$ replaces $\kappa_{z,w}$ by its inverse
$\kappa_{w,z}=(\kappa_{z,w})^{-1}$.
We call $\kappa_{z,w}$ a {\bf compatible nodal identifier}.
Given $\psi\in\GG_j$ and $\{z,w\}\in D$ we define the {\bf phase function}
\[
\Theta_{\{z,w\}}(\psi)\co T_zS\lra T_zS
\]
by
\[
\Theta_{\{z,w\}}(\psi):=
\kappa_{z,w}\circ
T_{\psi(w)}(\psi)^{-1}\circ
\kappa_{\psi(w),\psi(z)}\circ
T_z\psi
\,.
\]
Taking positive and negative holomorphic polar coordinates about $z$ and $w$, resp., so that
$\psi$ acts in coordinates by multiplication with $\rme^{-\rmi\theta_+}$ and
$\rme^{\rmi\theta_-}$, resp., we get
\[
\big(\Theta_{\{z,w\}}(\psi)\big)(v)=
\rme^{-\rmi(\theta_+-\theta_-)}v
\]
for all $v\in T_zS$, which we simply declare to a multiplication operator
\[
\Theta_{\{z,w\}}(\psi)\equiv
\rme^{-\rmi(\theta_+-\theta_-)}
\,.
\]
This shows independence of the phase function
\[
\Theta\co D\times\GG_j\lra S^1
\,,\qquad
\big(\{z,w\},\psi\big)\longmapsto\Theta_{\{z,w\}}(\psi)
\,,
\]
of the chosen ordering of $\{z,w\}$ in the
definition of $\Theta_{\{z,w\}}(\psi)$ and of the chosen parity of the holomorphic polar
coordinates about $\psi(z)$ and $\psi(w)$.
This results in a $\GG_j$-action on $\D^D$ defined by $\psi_*\bfa=\bfb$ via
\[
b_{\{\psi(z),\psi(w)\}}:=
\Theta_{\{z,w\}}(\psi)\cdot a_{\{z,w\}}
\]
for all $\{z,w\}\in D$.
Consequently, for any skyscraper deformation $V_j\ni y\mapsto j(y)$,
a small disc structure of sufficiently small discs,
and $\psi\in\GG_j$ we get an isomorphism
\[
\psi_{\bfa}\co
\Big(S_{\bfa},j(\psi^*y)_{\bfa},D_{\bfa},\{1,\rmi,-1\},A\Big)
\lra
\Big(S_{\psi_*\bfa},j(y)_{\psi_*\bfa},D_{\psi_*\bfa},\{1,\rmi,-1\},A\Big)
\]
by the gluing construction and symmetry of $y\mapsto j(y)$.


\subsection{Topology and orbifold structure -- variable stable nodal type\label{subsec:topologyorbistrofrrvaryingnodaltype}}

 A neighbourhood base of a second countable paracompact Hausdorff topology on $\RR_N$,
 $N\geq0$, is given by the family of subsets of $\RR_N$, whose elements are of the form
 \[
 \big[S_{\bfa},k_{\bfa},D_{\bfa},\{1,\rmi,-1\},A\big]
 \]
 with $N=\#A$, which are obtained from a nodal disc $\big(S,k,D,\{1,\rmi,-1\},A\big)$ by the
 parametrised connected sum construction with given decorated small disc structure
 $\mathbf{D}_j$, with gluing parameter $\bfa\in\D^D$ with $|\bfa|<\varepsilon$ for some
 $\varepsilon\in(0,1)$, with complex structures $k$ that belong to an open neighbourhood
 of $j$ in $\JJ$ such that $k=j$ restricted to $\mathbf{D}_j$.
 This follows as in \cite[Proposition 2.4]{hwz-gw17} and
 \cite[Theorem 2.15 and Theorem 5.13]{hwz-dm12} because no extra argument for boundary
 un-noded nodal discs is needed caused by absence of boundary nodes.
 The Hausdorff property follows with Gromov compactness for stable holomorphic discs,
 see \cite{fz15}.
 
 The induced topology on $\RR_{\tau}$ in $\RR_N$ agrees with the one on $\RR_{\tau}$
 previously defined in Remark \ref{rem:topologyfixedstabletype}. 
 The induced notion of convergence of sequences in $\RR_N$ coincides with Gromov
 convergence as described in \cite[Chapter 1]{abb14}, \cite[Appendix B]{wen05} or in
 \cite[Section 4]{behwz03}, \cite[Chapter IV]{hum97} after Schwartz reflection along the
 boundary of the nodal discs for example.
 
In order to obtain an orbifold structure on $\RR_N$ we consider desingularisations
\[
\big(S_{\bfa},j(y)_{\bfa},D_{\bfa},\{1,\rmi,-1\},A\big)
\]
of $\big(S,j,D,\{1,\rmi,-1\},A\big)$ as described in Section
\ref{subsec:varstabnodtypeviadesing}.
For $\bfa_0\in\D^D$ consider the set $D\setminus D_{\bfa_0}$ of all nodal pairs
$\{z,w\}\in D$ on which the map $\bfa_0$ is non-zero.
Define a deformation
\[
\mathfrak{j}_{\bfa_0}\co
V_{j_{\bfa_0}}\times\D^{D\setminus D_{\bfa_0}}
\lra\JJ_{S_{\bfa_0}}
\,,\qquad
(y,\bfb)\longmapsto j(y)_{\bfa_0+\bfb}
\,,
\]
of
\[
\big(S_{\bfa_0},j(y)_{\bfa_0},D_{\bfa_0},\{1,\rmi,-1\},A\big)
\]
by setting $\bfa=\bfa_0+\bfb$.
For small deformation parameter $\bfb$ the deformed family of surfaces equals
\[
\Big(S_{\bfa_0+\bfb},j(y)_{\bfa_0+\bfb},D_{\bfa_0+\bfb},\{1,\rmi,-1\},A\Big)
\,.
\]
The nodal discs family is isomorphic to 
\[
\Big(S_{\bfa_0},j'(y)_{\bfb},D_{\bfa_0},\{1,\rmi,-1\},A\Big)
\]
with corresponding deformation
\[
\mathfrak{j}'_{\bfb}\co
V_{j_{\bfa_0}}\times\D^{D\setminus D_{\bfa_0}}
\lra\JJ_{S_{\bfa_0}}
\,,\qquad
(y,\bfb)\longmapsto j'(y)_{\bfb}
\,,
\]
via an isomorphism that is the identification map on the complement of the respective small
disc structure, so that the deformation is given by rotations and stretchings of the cylindrical
neck regions that correspond to the nodes, on which $\bfa_0$ not vanishes.
The {\bf partial Kodaira differential} of $\mathfrak{j}'_{\bfb}$ at $(y,0)$ is
\[
\big[T_{(y,0)}\mathfrak{j}'_{\bfb}\big]\co
E_{j_{\bfa_0}}\times\C^{D\setminus D_{\bfa_0}}
\lra
\Omega^{0,1}_{{j(y)}_{\bfa_0}}\lra H^1_{{j(y)}_{\bfa_0}}
\,.
\]

Similarly to \cite[Theorem 2.13]{hwz-gw17} one constructs uniformisers of an orbifold structure
on $\RR_N$ as the above desingularisations stay away from the boundary of the nodal discs
in $\RR_N$.
For given $\big[S,j,D,\{1,\rmi,-1\},A\big]\in\RR_N$ such a {\bf uniformiser}
is a deformation
\[
\VV\ni(y,\bfa)\longmapsto\big(S_{\bfa},j(y)_{\bfa},D_{\bfa},\{1,\rmi,-1\},A\big)
\]
of $\big(S,j,D,\{1,\rmi,-1\},A\big)$ for an open subset $\VV$ of $V_j\times\D^D$
such that the following holds:

\begin{itemize}
 \item
 The union of all equivalence classes $\big[S_{\bfa},j(y)_{\bfa},D_{\bfa},\{1,\rmi,-1\},A\big]$
 over all $(y,\bfa)\in\VV$ is an open subset of $\RR_N$.
 \item
 The map $\VV\ra\UU$ that assigns to $(y,\bfa)$ the class
 $\big[S_{\bfa},j(y)_{\bfa},D_{\bfa},\{1,\rmi,-1\},A\big]$ descends to a homeomorphism
 $\VV/\GG_j\ra\UU$.
 \item
 An isomorphism between the classes belonging to $(y,\bfa),(z,\bfb)\in\VV$
 is given by $\psi_{\bfa}$ for $\psi\in\GG_j$ and $(z,\bfb)=\big(\psi_*y,\psi_*\bfa\big)$.
 \item
 For all points in $\VV$ the partial Kodaira differential is an isomorphism.
\end{itemize}

Compatibility of uniformisers is expressed via the sets
\[
\mathbf{T}_{j,k}:=
\Big\{
 \big(\varphi,(y,\bfa),(z,\bfb)\big)
\Big\}
\subset\GG\times(V_j \times\D^D)\times(V_k \times\D^D)
\]
corresponding to all isomorphisms
\[
\varphi\co
\Big(S_{\bfa},j(y)_{\bfa},D_{\bfa},\{1,\rmi,-1\},A\Big)
\lra
\Big(S_{\bfb},k(z)_{\bfb},D_{\bfb},\{1,\rmi,-1\},A\Big)
\,,
\]
which are smooth manifolds of dimension $2\#A$, so that $\RR_N$ supports an \'etale proper Lie
groupoid structure as formulated after Proposition \ref{prop:rtauisanorbifold}.
This follows with the (anti-)gluing construction (\cite[Section 2.4] {hwz-gw17}) for the
non-linear Cauchy--Riemann operator
along the nodes (which take place away from the boundary) known from Floer theory, cf.\
\cite[Theorem 2.16]{hwz-gw17} and \cite[Theorem 2.24]{hwz-dm12}.
Similarly, the universal property of the construction stated in \cite[Theorem 2.16]{hwz-gw17}
translates into the present situation.
The involved variation of marked points can be treaded as in \cite[Remark 3.17]{hwz-dm12}.
Finally, using convex interpolation between the exponential gluing profile $\rme^{1/r}-\rme$ we used
in the gluing construction and the logarithmic gluing profile $-\ln r$ that appeared in the
desingularisation of the complex algebraic curve at the beginning of Section
\ref{subsec:varstabnodtypeviadesing} naturally yields an orientation on $\RR_N$ that extends the
complex orientation on $\RR_{\tau}$ given in Proposition \ref{prop:rtauisacomplexorbifold},
see \cite[Section 2.3.2]{hwz-dm12}.

\begin{thm}
 \label{thm:rnisaorientedorbifold}
 The nodal Riemann moduli space $\RR_N$
 of stable nodal boundary un-noded discs
 with $N=\#A$ interior marked points
 admits a naturally oriented orbifold structure of
 dimension $2\#A$.
\end{thm}


\section{Polyfold perturbations}
\label{sec:polyfoldperturbations}

We prove Theorem \ref{thm:maindirectsyplcoborthm} under assumption (ii).
For that we place ourselves into the situation of Section \ref{subsec:thesemiposcase}
and follow the line of reasoning of the proof of Theorem \ref{thm:maindirectsyplcoborthm}
part (i).
As we will not assume semi-positivity this time regularity of relevant moduli spaces
can only be achieved for simple nodal holomorphic discs via perturbing the almost
complex structure, cf.\ Section \ref{subsec:thesemiposcase}.
For non-simple nodal holomorphic discs we will use additional abstract polyfold
perturbations as introduced in \cite{hwz-gw17}.


\subsection{Boundary un-noded stable disc maps}
\label{subsec:boundunnodstabdiscmaps}

We consider the tame almost complex manifold $\big(\hat{W},\hat{\Omega},\hat{J}\big)$
defined in Section \ref{subsec:complthecob}.
For boundary un-noded nodal discs $(S,j,D)$ as introduced in Section
\ref{subsec:boundunnodnoddisc} we consider smooth maps
\[
u\co(S,\partial S)\lra(\hat{W},N^*)
\]
that descend to continuous maps on $S/D$.
If $D$ is empty we call $u$ {\bf un-noded}.
If in addition $Tu\circ j=\hat{J}(u)\circ Tu$ we call $u$ a {\bf nodal holomorphic disc map}.
Observe that we do not need to consider nodal points on the boundary due to the Gromov
compactification described in Remark \ref{rem:gromovcomp}.

More generally, we consider continuous maps $u\co(S,\partial S)\ra(\hat{W},N^*)$ defined on a
marked boundary un-noded nodal disc $\big(S,j,D,\{m_0,m_1,m_2\}\big)$ (see Section
\ref{subsec:boundunnodnoddisc}) such that $u$ descends to a continuous map on the quotient
$S/D$ and such that $u\big(\!\Int S\big)\subset\Int\hat{W} $.
Moreover, we require that $u$ is contained in the Sobolev space of square integrable maps
\[
H^{3,\sigma}(S,j)\equiv
H^{3,\sigma}\big(S,j,D,\{m_0,m_1,m_2\}\big)
\]
following \cite[Definition 1.1]{hwz-gw17}:
We require that $u$ is of class $u\in H_{\loc}^3\big(S\setminus|D|\big)$ and that w.r.t.\ positive
holomorphic polar coordinates $[0,\infty)\times S^1$, $S^1=\R/2\pi\Z$, about the nodal points
$|D|$ (see Section \ref{subsec:varstabnodtypeviadesing}) the map $u$ is of weighted Sobolev
class $H^{3,\sigma}$.
The weights are given by $\rme^{\sigma s}$, $s\in[0,\infty)$, for some $\sigma\in(0,1)$.
In other words, $u$ is contained in $H^{3,\sigma}$ precisely if all weak derivatives
$D^{\alpha}u$, $|\alpha|\leq3$, of $u$ on $[0,\infty)\times S^1$ exist and all
$D^{\alpha}u\cdot\rme^{\sigma s}$, $|\alpha|\leq3$, are square integrable on
$[0,\infty)\times S^1$.
The latter is equivalent to $u\,\rme^{\sigma s}\in H^3$ on $[0,\infty)\times S^1$.
In particular, by Sobolev embedding, $u$ is $C^1$ (up to the boundary $\partial S$)
restricted to $S\setminus|D|$.
But in general $u$ is not differentiable at the nodal points $|D|$ on $S$.
Consider for example the continuous function $u(z)=|z|^{\frac{1+\sigma}{2}}$ in holomorphic
coordinates $z\in\C$, which w.r.t.\ positive holomorphic polar coordinates reads as
$(s,t)\mapsto\rme^{-\frac{1+\sigma}{2}s}$.

The space $H^{3,\sigma}(S,j)$ is well defined, i.e.\ invariant under coordinate changes after
possibly shrinking the chart domains.
Away from the nodes $|D|$ this follows as for $H_{\loc}^3\big(S\setminus|D|\big)$ via
\cite[Theorem 3.41]{af03}.
Near the nodes we observe that the area form $\rme^{2\sigma s}\rmd t\wedge\rmd s$
w.r.t.\ positive holomorphic polar coordinates corresponds to the singular area form
$|z|^{-2(1+\sigma)}\frac{\rmi}{2}\rmd z\wedge\rmd\bar{z}$ in holomorphic coordinates
about the nodal point $0\in\C$.
The area form $\frac{\rmi}{2}\rmd z\wedge\rmd\bar{z}$ transforms under biholomorphic
coordinate changes via a conformal factor, which we can assume to be bounded above and
away from zero by shrinking the chart domains if necessary.
The coordinate change itself is of the form $z\mapsto zh(z)$, where $0$ corresponds to a nodal
point.
Here $h$ is a holomorphic function, whose absolute value can be assumed to be bounded
above and away from zero also.
Consequently, the singular area form $|z|^{-2(1+\sigma)}\frac{\rmi}{2}\rmd z\wedge\rmd\bar{z}$
transforms via a bounded above and away from zero conformal factor also.
Hence, the same holds true for $\rme^{2\sigma s}\rmd t\wedge\rmd s$.
In fact, the above coordinate change becomes
\[
(s,t)\longmapsto(s,t)-\ln\Big(h\big(\rme^{-(s+\rmi t)}\big)\Big)
\,,
\]
whose derivatives are bounded above and whose first derivative is bounded away from zero.
Therefore, invariance under coordinate changes near the nodes follows as in
\cite[Theorem 3.41]{af03}.
By the same arguments we see that locally defined norms on $H_{\loc}^3$ and $H^{3,\sigma}$
transform via the respective coordinate changes to equivalent norms.
This defines a topology on $H^{3,\sigma}(S,j)$; a neighbourhood base is given by the set of
those maps that restricted to one of the above charts belong to an open set in $H_{\loc}^3$
and $H^{3,\sigma}$, resp.

To each $u\in H^{3,\sigma}(S,j)$ we assign the {\bf symplectic energy integral}
\[
\int_Su^*\hat{\Omega}
\]
by approximating the continuous map $u$ by a $C^1$-map $v$ and defining the symplectic energy
integral via $\int_Su^*\hat{\Omega}:=\int_Sv^*\hat{\Omega}$.
This is well defined and, in fact, by Stokes theorem, independent of the choice of representative
of the homology class $[u]$ in $\hat{W}$ relative $u(\partial S)\subset N^*$.
This can be seen as follows:
Taking approximations $v$ of $u$ that are equal to $u$ restricted to the complement of disc
like neighbourhoods $B_r(0)$ in $S$ of the nodal points $0\in|D|$ the symplectic energy integral
is given by $\int_{S\setminus|D|}u^*\hat{\Omega}$.
Indeed, take $r>0$ so small such that the $B_r(0)$ are contained in pair-wise disjoint chart
domains of $S$ about the nodal points $0$ in $|D|$ and such that the $u\big(B_r(0)\big)$ are
contained in pair-wise disjoint ball like chart domains of $\hat{W}$.
By Stokes theorem decomposing
$B_r(0)=\big(B_r(0)\setminus B_{\varepsilon}(0)\big)\cup B_{\varepsilon}(0)$ it suffices to show
that the integrals
\[
\int_{B_{\varepsilon}(0)}u^*\hat{\Omega}
\qquad
\text{and}
\qquad
\int_{\partial B_{\varepsilon}(0)}u^*\lambda
\]
converge to zero as $\varepsilon\in(0,r)$ tends to $0$, where $\lambda$ is a local primitive of
$\hat{\Omega}$ defined on the ball like neighbourhoods of $u(|D|)$ in $\hat{W}$.
By the transformation formula we can compute the integrals w.r.t.\ positive holomorphic polar
coordinates via
\[
\int_{(R,\infty)\times S^1}\hat{\Omega}(u_s,u_t)\,\rmd s\wedge\rmd t
\qquad
\text{and}
\qquad
\int_{\{R\}\times S^1}\lambda(u_t)\,\rmd t
\]
for $R=-\ln\varepsilon$.
By the Sobolev inequality the $C^1$-norm of $u\,\rme^{\sigma s}$ on $[0,\infty)\times S^1$ is
bounded by $\|u\|_{3,\sigma}$, so that up to a positive constant the absolute value of the
integrals is bounded by
\[
\|u\|_{3,\sigma}^2\int_R^{\infty}\rme^{-2\sigma s}\rmd s
\qquad
\text{and}
\qquad
\|u\|_{3,\sigma}\,\rme^{-\sigma R}
\,,
\]
resp.
In both cases the first factor is bounded by assumption; the second tends to zero for $R\ra\infty$
and the claim follows, namely, that $\int_Su^*\hat{\Omega}$ is well defined.

\begin{rem}
\label{symplenerintisc0}
 The above arguments show that $u\mapsto \int_Su^*\hat{\Omega}$ is a continuous function
 on $H^{3,\sigma}(S,j)$.
\end{rem}

We call $\big(S,j,D,\{m_0,m_1,m_2\},u\big)$ a {\bf nodal disc map}
provided that the following conditions are satisfied
(cf.\ Section \ref{subsec:thesemiposcase}):

\begin{enumerate}
  \item
  $u\in H^{3,\sigma}(S,j)$,
  \item
  the {\bf symplectic energy integral} restricted to a connected component $C$ of $S$
	 \[
	 \int_Cu^*\hat{\Omega}\geq0
	 \]
  is non-negative for all spherical components $C$ of $S$; positive on the disc component,
  \item
  the continuous map on $S/D$ induced by $u$ is homologous to a local Bishop discs
  $u_{\varepsilon,b_o}$ relative $N^*$, so that $[u(S)]=[u_{\varepsilon,b_o}(\D)]$ in
  $H_2(\hat{W},N^*)$, and
  \item
  $u(m_0)\in\gamma$ and $\vartheta\circ u(m_k)=\rmi^k$ for $k=1,2$.
\end{enumerate}

\noindent
For given $\big(S,j,D,\{m_0,m_1,m_2\}\big)$ the space
\[
\HH^{3,\sigma}(S,j)
\]
of nodal disc maps is called the {\bf space of admissible maps}.

It follows that the degree of the $C^1$-map $\vartheta\circ u\co\partial S\ra S^1$ equals
$1$ for all nodal disc maps $\big(S,j,D,\{m_0,m_1,m_2\},u\big)$.
With the properties of the symplectic energy integral discussed above we obtain as in item
(2) of Section \ref{subsec:gromovcomp} that
\[
\int_Su^*\hat{\Omega}=
\int_{\partial S} u^*f\cdot(\vartheta\circ u)^*\rmd\theta
\,,
\]
where $f$ is a smooth function on $N$ that is positive on $N^*$ and vanishes on
$B\cup\partial N$.
As $u$ takes values in $N^*$ along the boundary $\partial S$ we get that
\[
\int_Su^*\hat{\Omega}
\in\big(0,2\pi\max f\big]
\]
for all nodal disc maps $\big(S,j,D,\{m_0,m_1,m_2\},u\big)$.
By non-negativity of the symplectic energy integral on each connected component $C$ of $S$ we
get that $\int_Cu^*\hat{\Omega}$ takes values in $\big[0,2\pi\max f\big]$.
Moreover, as $\int_Cu^*\hat{\Omega}$ only depends on the homology class represented by $u(C)$
for the spherical components $C$ of $S$ assumption (2) puts an open condition to the space
defined via $H^{3,\sigma}(S,j)$ and the constraints given by (3) and (4), so that the space of
admissible maps $\HH^{3,\sigma}(S,j)$ is an open subset.

In fact, $\HH^{3,\sigma}(S,j)$ is a Hilbert manifold whose tangent space
$\HH^{3,\sigma}\big(u^*T\hat{W}\big)$ at $u\in\HH^{3,\sigma}(S,j)$ is the space of
$H^{3,\sigma}$-sections into $u^*T\hat{W}$ that descent to continuous sections on $S/D$,
that are tangent to $N^*$ along $\partial S$ as well as tangent to $\gamma$ at $m_0$ and
to the page $\vartheta^{-1}(\rmi^k)$ at $m_k$ for $k=1,2$.
This follows with the exponential map taken w.r.t.\ a metric on $\hat{W}$ for which each of the
submanifolds $N$, $\vartheta^{-1}(\rmi^k)$, $k=1,2$, and $\gamma$ is totally geodesic.
The requirement for the sections to be of class $H^{3,\sigma}$ is understood as in
Section \ref{subsec:boundunnodstabdiscmaps}, so that a norm on
$\HH^{3,\sigma}\big(u^*T\hat{W}\big)$ as on \cite[p.~66]{hwz-gw17} can be defined.
This turns $\HH^{3,\sigma}(S,j)$ into a Riemannian Hilbert manifold.

By removal of singularities (see \cite{mcsa04})
a holomorphic $u\in H^{3,\sigma}(S,j)$,
which is continuous and has finite symplectic energy by the above discussion,
is holomorphic on $S$.
Therefore,
$u$ is smooth up to the boundary including all nodal points $|D|$
so that holomorphicity coincides with the notion of holomorphicity
from the beginning of this section.

Given a nodal disc map $\big(S,j,D,\{m_0,m_1,m_2\},u\big)$
we call a connected component $C$ of $S$
with vanishing symplectic energy integral a {\bf ghost bubble}.
Observe that a holomorphic nodal disc map restricted to a ghost bubble is constant.
If $\big(S,j,D,\{m_0,m_1,m_2\},u\big)$ is any nodal disc map
such that each ghost bubble admits at least $3$ nodal points,
then we call $\big(S,j,D,\{m_0,m_1,m_2\},u\big)$ a {\bf stable nodal disc map}.


\subsection{Boundary un-noded stable discs}
\label{subsec:boundunnodstabdiscs}

We call two stable nodal disc maps
\[
\big(S,j,D,\{m_0,m_1,m_2\},u\big)
\quad
\text{and}
\quad
\big(S',j',D',\{m'_0,m'_1,m'_2\},u'\big)
\]
{\bf equivalent} if there exists a diffeomorphism $\varphi\co S\ra S'$
such that $\varphi^*j'=j$,
the injection $D\ra D'$ defined by
$\{\varphi(x),\varphi(y)\}\in D'$ for all $\{x,y\}\in D$ is surjective,
$\varphi(m_k)=m'_k$ for $k=0,1,2$, and $u'\circ\varphi=u$.
The discussions in Section \ref{subsec:boundunnodstabdiscmaps} about $H^{3,\sigma}(S,j)$
imply that this equivalence relation is well defined.
The equivalence classes
\[
\mathbf u=\big[S,j,D,\{m_0,m_1,m_2\},u\big]
\]
are called {\bf stable nodal discs} in $\big(\hat{W},\hat{\Omega}\big)$ relative $N^*$.
The space of all equivalence classes is denoted by $\ZZ$.

Fixing the diffeomorphism type of $S$ and the combinatorial data
$\big(D,\{1,\rmi,-1\}\big)$ of $\big(S,j,D,\{1,\rmi,-1\},u\big)$ as at the beginning of
Section \ref{subsec:groupoidasorbitspace} we can write
\[
\ZZ=\Big\{
\bfu=[j,u]\,\,\text{stable}
\,\Big|\,
[u(S)]=[u_{\varepsilon,b_o}(\D)]
\,,\,\,
u(1)\in\gamma
\,,\,\,
\vartheta\circ u(\rmi^k)=\rmi^k
\,,
k=1,2
\Big\}
\]
for the {\bf space of all stable nodal discs}
\[
\bfu=\big[S,j,D,\{1,\rmi,-1\},u\big]\equiv[j,u]
\,,\quad
u\in \HH^{3,\sigma}(S,j)\,,
\]
in $\big(\hat{W},\hat{\Omega}\big)$ relative $N^*$.

We define the nodal type $\tau$ of $\big(S,D,\{1,\rmi,-1\}\big)$
as in Section \ref{subsec:domainstabilisation}.
Namely, the nodal type is the isomorphism class of the rooted tree
given as follows:
The vertices correspond to the components of $S$.
The root is given by the disc component.
The edge relation is induced by the nodes in $D$.
As this time there are no auxiliary marked points
all vertices different from the root are not weighted;
the root has weight $3$.
The induced nodal type $\tau$ is necessarily unstable
provided that there is at least one sphere component.
Indeed,
in this case, any end of a branch admits only one special point.

We denote by $\ZZ_{\tau}$
the {\bf space of all stable nodal discs of nodal type} $\tau$,
so that $\ZZ$ is the disjoint union of the $\ZZ_{\tau}$
where $\tau$ ranges over all nodal types just described.
Each of the subspaces $\ZZ_{\tau}$ of $\ZZ$ is the quotient of the total space of the fibration
over $\JJ\equiv\JJ(S)$ with fibre $\HH^{3,\sigma}(S,j)$ over $j\in\JJ$ by the action
$\varphi\mapsto(\varphi^*j,u\circ\varphi)$
of the group of orientation preserving diffeomorphisms $\varphi$ of $S$ preserving
$\big(D,\{1,\rmi,-1\}\big)$, cf.\ Section \ref{subsec:groupoidasorbitspace}.
This puts a topology to $\ZZ_{\tau}$ similarly to Remark \ref{rem:topologyfixedstabletype}.

Notice that the stabiliser of the action is finite by the stability condition formulated at the end of
Section \ref{subsec:boundunnodstabdiscmaps}:
Each automorphism of $\bfu$ acts via the identity map on the disc component due to the
ordered boundary marked points $\{1,\rmi,-1\}$.
If $\int_Cu^*\hat{\Omega}=0$ for a connected component $C$ of $S$, then the number of
nodal points $C\cap|D|$ on $C$ is at least $3$.
Furthermore, the automorphisms of $\bfu$ preserve those ghost components
due to the transformation formula.
If $\int_Cu^*\hat{\Omega}>0$, one finds $z\in C\setminus|D|$ such that $u$ is immersive on
$C\cap u^{-1}\big(u(z)\big)$ defining finitely many local branches via
$C\cap u^{-1}\big(B_r\big)\subset C\setminus|D|$ for a sufficiently small ball
$B_r\subset\Int\hat{W}$ around $u(z)$.
In fact, due to the positivity of the symplectic energy integral we can find $z\in C\setminus|D|$
and $r>0$ sufficiently small such that $u^*\hat{\Omega}$ is a positive area form on the branch
through $z$, which is oriented via $j$.
Observe that $u^*\hat{\Omega}$ is a positive area form on all branches through the points of
the orbit (of the automorphism group of $\bfu$) defined by $z$.
Identifying the sphere components with $(\C P^1,\rmi)$ as in Remark \ref{rem:isotropyisfinite}
the identity theorem yields that an automorphism of $\bfu$ acts by a permutation on the local 
branches.
This proves finiteness of the stabiliser.

In the following we describe a polyfold structure on $\ZZ$ that glues the components
$\ZZ_{\tau}$ together.
For any $\bfu=\big[S,j,D,\{1,\rmi,-1\},u\big]$ in $\ZZ$
one can choose a so-called {\bf stabilisation},
which is a finite set of auxiliary marked points $A\subset S$
disjoint from the special points $D\cup\{1,\rmi,-1\}$ 
such that the nodal disc $\big(S,j,D,\{1,\rmi,-1\},A\big)$ is stable
in the sense of Section \ref{subsec:domainstabilisation}.
Due to the stability condition
there is no need to provide the ghost bubbles with an auxiliary markt point.
In addition one can assume,
that the automorphisms of $\bfu$ preserve $A$,
$u(A)$ is disjoint from the $u$-image of $D\cup\{1,\rmi,-1\}$,
and the following two conditions hold:

\begin{enumerate}
  \item
  Whenever $z,w\in A$ are mapped to the same point $u(z)=u(w)$ in $\hat{W}$,
  then there exists an automorphism of $\bfu$ sending $z$ to $w$.
  \item
  For all $z\in A$ the $2$-form $(u^*\hat{\Omega})_z$ is positive on $(T_zS,j_z)$.
\end{enumerate}

\noindent
This follows with \cite[Lemma 3.2]{hwz-gw17} ignoring the disc component,
which already is stable:
Namely,
successively select finite orbits of the action of the automorphism group of $\bfu$
on local branches similarly to the above finiteness argument
until all components are stable.
Consequently,
the underlying stable nodal disc $\big[S,j,D,\{1,\rmi,-1\},A\big]$
possesses a uniformiser
as described in Section \ref{subsec:topologyorbistrofrrvaryingnodaltype}.

As in Section \ref{subsec:thesemiposcase}
we wish to achieve an index-$1$ Fredholm problem.
In view of Theorem \ref{thm:rnisaorientedorbifold}
we compensate the stabilising auxiliary marked points $A$ index-wise as follows:
We choose a finite collection of pairwise disjoint
codimension-$2$ symplectic discs in $(\Int\hat{W},\hat{\Omega})$
that intersect $u(S)$ along $u(A)$ transversally.
This is possible by condition (2) above.
Namely,
the image of $Tu$ at each auxiliary marked point in $A$
is a symplectic plane in $T\hat{W}$.
Integrating the respective symplectic normal subspaces
one finds symplectic embeddings of small discs of codimension $2$
that are normal to $u(S)$ at the images of the auxiliary marked points $u(A)$.
We call the union of the discs $H_{u,A}$ {\bf local transversal constraints}
if the intersection of $u(S)$ and $H_{u,A}$ equals $u(A)$
and if each component of $H_{u,A}$ intersects $u(S)$ in a single point.

We denote by
\[
E_{u,A}\subset\HH^{3,\sigma}\big(u^*T\hat{W}\big)
\]
the subspace of sections that are tangent to $H_{u,A}$ at the stabilising auxiliary points in $A$,
which is scale-linear w.r.t.\ to $(3+\nu,\sigma_{\nu})$, $\nu\in\N_0$, for a strictly increasing
sequence $\sigma_{\nu}$ in $(0,1)$ with $\sigma_0=\sigma$, see \cite[Section 2.6]{hwz-gw17}.
Uniformiser about any $\bfu=\big[S,j,D,\{1,\rmi,-1\},u\big]$ in $\ZZ$ of the desired polyfold
structure are obtained as in \cite[Section 3.1/3.2]{hwz-gw17}.
To adapt to our situation start off with uniformisers for the stabilised domain
$\big[S,j,D,\{1,\rmi,-1\},A\big]\in\RR$ from
Section \ref{subsec:topologyorbistrofrrvaryingnodaltype} and consider the deformation
\[
(y,\bfa,\eta)
\longmapsto
\big(S_{\bfa},j(y)_{\bfa},D_{\bfa},\{1,\rmi,-1\},\oplus_{\bfa}\exp_u(\eta)\big)
\,,
\]
where $(y,\bfa)\in\VV$ for an open subset $\VV$ of $V_j\times\D^D$, $\eta\in E_{u,A}$ is a
sufficiently small section that is a fixed point of the splicing projection $\pi_{\bfa}$ and
$\oplus_{\bfa}\exp_u(\eta)$ denotes the gluing operation both introduced in
\cite[Section 2.4/2.5]{hwz-gw17}.
Choosing $u$ to be a smooth approximation of an element in $H^{3,\sigma}(S,j)$ we obtain
scale-smooth gluing maps w.r.t.\ to the scale $(3+\nu,\sigma_{\nu})$,
see \cite[Section 2.2/2.6]{hwz-gw17}.
Using Remark \ref{rem:topologyfixedstabletype},
Section \ref{subsec:topologyorbistrofrrvaryingnodaltype} and \cite[Section 3.3/3.4]{hwz-gw17}
one obtains a natural second countable paracompact Hausdorff topology on $\ZZ$ similarly to
\cite[Theorem 1.6]{hwz-gw17}.
In the same way using this time modifications in \cite[Section 3.5]{hwz-gw17} the space
$\ZZ$ carries the structure of a polyfold as formulated in \cite[Theorem 1.7]{hwz-gw17} with
a scale-smooth evaluation map $\ZZ\ra\gamma$ sending $\bfu$ to $u(1)$,
cf.\ \cite[Theorem 1.8]{hwz-gw17}.


\subsection{A nodal moduli space\label{subsec:modulisp}}

We call $\bfu$ a {\bf stable nodal holomorphic disc} if $\bfu$ can be represented
by a stable nodal holomorphic disc map $u$. 
Notice, that all stable nodal disc maps $u$ that represent a stable nodal holomorphic disc
$\bfu$ are holomorphic.
Denote by
\[
\NN:=\big\{
\bfu\in\ZZ
\,\big|\,
\bfu
\,\,\,\text{is holomorphic}
\big\}
\]
the {\bf nodal moduli space} of all stable nodal holomorphic discs.

Using uniformisation it is convenient to represent the classes $\mathbf u\in\NN$ by
holomorphic maps $u\in\HH^{3,\sigma}(S,j)$ whose disc component has domain
$\big(\D,\rmi,\{1,\rmi,-1\}\big)$ and for which the sphere components are given by
$(\C P^1,\rmi)$, cf.\ Remark \ref{rem:isotropyisfinite}.
If $\bfu$ is un-noded, then we obtain $\bfu=\big[\D,\rmi,\emptyset,\{1,\rmi,-1\},u\big]$.
We abbriviate the elements $\mathbf u=[\rmi,u]\in\NN$ (noded or un-noded)
simply by $[u]$ for the following discussion:

The boundary conditions for $\HH^{3,\sigma}(S,j)$ formulated in Section
\ref{subsec:boundunnodstabdiscmaps} are the boundary conditions used in
Sections \ref{subsec:gromovcomp} and \ref{subsec:thesemiposcase}.
In particular, all properties formulated in the un-noded case for holomorphic discs
in Section \ref{subsec:gromovcomp} continue to hold in the noded case,
hence, for all $\bfu=[u]\in\NN$ in the following sense:

\begin{enumerate}
\item
 The {\bf winding number} of $\bfu=[u]\in\NN$, which by definition is the degree of
 the map $\vartheta\circ u\co\partial S\ra S^1$, is equal to $1$.
 In particular, $u(\partial S)$ is an embedded curve in $N^*$ positively transverse
 to $\xi$ and the restriction of $u$ to the disc component of $S$ is a simple holomorphic
 map.
\item
 The {\bf symplectic energy} $\int_{S}u^*\hat{\Omega}$ of $\bfu=[u]\in\NN$, which is well
 defined and positive by Section \ref{subsec:boundunnodstabdiscmaps}, is uniformly bounded.
\item
 The boundary circle $u(\partial S)$ of $\bfu=[u]\in\NN$ is disjoint from $U_{\partial N}$
 because the restriction of $u$ to the disc component of $S$ must be disjoint from
 $U_{\partial N}$ by Lemma \ref{lem:blockinglemma}.
 If $u(S)$ intersects $U_B$ then $\bfu$ is un-noded and equivalent to a local
 Bishop disc $u_{\varepsilon,b_o}$ by Lemma \ref{lem:semiglobunique} combined with the final
 paragraph of Remark \ref{rem:gromovcomp}.
\end{enumerate}

\noindent
The local Bishop discs $u_{\varepsilon,b_o}$, $\varepsilon\in(0,\delta)$, represent
elements
\[
\bfu_{\varepsilon,b_o}=
\big[\D,\rmi,\emptyset,\{1,\rmi,-1\},u_{\varepsilon,b_o}\big]
\]
in $\NN$.
The corresponding local Bishop filling can be identified with $(0,\delta)$.
We truncate the nodal moduli space $\NN$ via
\[
\NN_{\cut}=\NN\setminus (0,\delta/2)\,.
\]

\begin{rem}
\label{rem:compactimageunderev}
 If there exists a compact subset $K$ of $\hat{W}$ such that $\bfu(S)$ is
 contained in $K$ for all $\bfu\in\NN$, then the Gromov compactification
 of $\MM_{\gamma}$ can be identified with a subset of $\NN$ by taking equivalence
 classes, see \cite{fz15}.
\end{rem}


\subsection{Cauchy--Riemann section\label{subsec:cauriemset}}

The moduli space $\NN$ is the zero set
\[
\NN=\big\{
\bfu\in\ZZ
\,\big|\,
\CR{\!\hat{J}}\mathbf u=\mathbf 0
\big\}
\]
of the Cauchy--Riemann operator $\CR{\!\hat{J}}$, which appears as a section into the bundle
\[
p\co\WW\lra\ZZ
\]
over $\ZZ$.
The fibre of $p$ over $\mathbf u=\big[S,j,D,\{1,\rmi,-1\},u\big]\in\ZZ$ consists of
equivalence classes $\boldxi=\big[S,j,D,\{1,\rmi,-1\},u,\xi\big]$ of continuous sections $\xi$
of $\Hom\big(TS,u^*T\hat{W}\big)$ so that for each $z\in S$ the map
$\xi(z)\co T_zS\ra T_{u(z)}\hat{W}$ is complex anti-linear with respect to $j(z)$ and
$\hat{J}\big(u(z)\big)$.
Moreover, $\xi$ is of Sobolev class $H^2_{\loc}$ on $S\setminus|D|$ and of weighted
Sobolev class $H^{2,\sigma}$ near $|D|$ similarly to the description at the beginning of
Section \ref{subsec:boundunnodstabdiscmaps}, see \cite[Section 1.2]{hwz-gw17}.
Two such sections $\big(S,j,D,\{m_0,m_1,m_2\},u,\xi\big)$ and
$\big(S',j',D',\{m'_0,m'_1,m'_2\},u',\xi'\big)$ are {\bf equivalent}, if there exists an equivalence
$\varphi$ of stable nodal disc maps $\big(S,j,D,\{m_0,m_1,m_2\},u\big)$ and
$\big(S',j',D',\{m'_0,m'_1,m'_2\},u'\big)$ as described at the beginning of Section
\ref{subsec:boundunnodstabdiscs} such that $\xi'\circ T\varphi=\xi$.
By adapting \cite[Theorem 1.9]{hwz-gw17} to the situation of the current Sections
\ref{sec:adelignemumfordtypespace} and \ref{sec:polyfoldperturbations} we obtain a natural
second countable paracompact Hausdorff topology on the total space $\WW$ and the
bundle projection $p\co\WW\ra\ZZ$ that maps $\big[S,j,D,\{1,\rmi,-1\},u,\xi\big]$ to
$\big[S,j,D,\{1,\rmi,-1\},u\big]$ is continuous.
Furthermore $p\co\WW\ra\ZZ$ constitutes a strong polyfold bundle in view of
\cite[Theorem 1.10]{hwz-gw17}.

The {\bf Cauchy--Riemann operator} $\CR{\!\hat{J}}$ is the section of $p$ given by
\[
\CR{\!\hat{J}}\mathbf u
:=
\Big[
  S,j,D,\{1,\rmi,-1\},u,\tfrac12\big(Tu+\hat{J}(u)\circ Tu\circ j\big)
\Big]
\]
for all $\mathbf u=\big[S,j,D,\{1,\rmi,-1\},u\big]\in\ZZ$.
For a representative we write $\CR{\!\hat{J}}u$ also.
As in \cite[Theorem 1.11]{hwz-gw17} the Cauchy--Riemann operator
$\CR{\!\hat{J}}\co\ZZ\ra\WW$ is a scale-smooth component-proper Fredholm section that
admits a natural orientation which we describe in Remark \ref{rem:crisorienbyrelspinstr} below.
The Fredholm index of $\CR{\!\hat{J}}\co\ZZ\ra\WW$ is $1$ by the index computation in
Section \ref{subsec:thesemiposcase} taking local transversal constraints from Section
\ref{subsec:boundunnodstabdiscs} in view of Theorem \ref{thm:rnisaorientedorbifold} into
account.
As in \cite[Section 5.3]{sz17} the vertical differential of a local representation of $\CR{\!\hat{J}}$ near
the local Bishop discs $\bfu_{\varepsilon,b_o}$, $\varepsilon\in(0,\delta)$ has a right-inverse.
The same holds true for all simple stable nodal holomorphic discs in $\ZZ$ due to the generic
choice of $\hat{J}$, see Section \ref{subsec:thesemiposcase}.

\begin{rem}
 \label{rem:homotopunitrivw2=0}
 Preparing the orientation considerations in Remark \ref{rem:crisorienbyrelspinstr} we will
 establish {\bf homotopically unique trivialisations} under the assumption that the second
 Stiefel--Whitney class of $N^*$ vanishes.
 This approach requires to build up the spaces $\HH^{3,\sigma}(S,j)$ with continuous maps on
 $S/D$ {\it homotopic} in $(\hat{W},N^*)$ to a local Bisphop disc, see item (3) in Section
 \ref{subsec:boundunnodstabdiscmaps}.
 As the relative homotopy class is preserved under Gromov convergence (see \cite{fz15})
 this is not a restriction.
 
 Consider the space of continuous maps $(\D,\partial\D)\ra\big(\hat{W},N^*\big)$ sending 
 the marked points $\{1\}$ and $\{\rmi^k\}$ into $\gamma$ and $\vartheta^{-1}(\rmi^k)$, $k=1,2$,
 resp.
 Denote by $\CC$ the connected component of the Bishop disc $u_0=u_{\delta/2,b_o}$.
 We claim that for all $u\in\CC$ the pull back bundle $u^*TN^*$ has a canonical trivialisation.
 
 In order to specify what is meant by this we describe the situation for $u_0$.
 By Section \ref{subsec:agermofbishopdisfill} the base point $u_0$ of $\CC$ is the map
 \[
 (\D,\partial\D)\lra
 \Big(
  (-\infty,0]\times\R\times\C\times T^*B,
  \{0\}\times\{0\}\times\C^*\times B
 \Big)
 \]
 given by
 \[
 u_0(z)=
 \Big(
 \tfrac{\delta^2}{4}\big(|z|^2-1\big),0,\delta\cdot z,b_o
 \Big)
 \;.
 \]
 in the local model $U_B$.
 The embedded path $\gamma$ corresponds to $\{0\}\times\{0\}\times\R^+\times\{b_o\}$
 as oriented curve and the pages $\vartheta^{-1}(\rmi^k)$, $k=1,2$, correspond to
 $\{0\}\times\{0\}\times\R^+\rmi^k\times B$, resp.
 The co-orientation of the pages $\vartheta^{-1}(\rmi^k)$, $k=1,2$, given in Section
 \ref {subsec:legopenbooks} is represented by the normal vectors $(0,0,\rmi^{k+1},0)$, resp.
 The local model defines a local frame
 $\partial_s,\partial_t,\partial_x,\partial_y,\partial_{\bfp},\partial_{\bfq}$ of the tangent bundle
 $T\hat{W}$ near $(-\infty,0]\times\R\times\C\times\{b_o\}$ inducing a trivialisation
 $\Phi\co u_0^*T\hat{W}\ra\D\times\R^{2n}$ of the pull back bundle $u_0^*T\hat{W}$.
 The trivialisation $\Phi$ restricts to a trivialisation $\Phi\co u_0^*TN^*\ra\partial\D\times\R^n$
 of $u_0^*TN^*$, which corresponds to the sub-frame $\partial_x,\partial_y,\partial_{\bfp}$.
 Further, $\Phi$ restricts to isomorphisms $\Phi\co T_{u_0(1)}\gamma\ra\{1\}\times\R$ via
 the vector field $\partial_x$ and
 $\Phi\co T_{u_0(\rmi^k)}\big(\vartheta^{-1}(\rmi^k)\big)\ra\{\rmi^k\}\times\R^{n-1}$, $k=1,2$, via
 the sub-frames $\partial_y/-\partial_x,\partial_{\bfp}$.
 The co-orientations of the pages $\vartheta^{-1}(\rmi^k)$ correspond to
 $-\partial_x/-\partial_y$, resp.
 We remark that $\Phi$ is not a complex trivialisation of the complex bundle pair
 $\big(u_0^*T\hat{W},u_0^*TN^*\big)$ as used to compute the Maslov index to be $2$, see
 \cite[Proposition 8]{nie06}.
 
 Given any $u\in\CC$ we claim that the pull back bundle $u^*TN^*$ admits a homotopically
 unique trivialisation $\Phi_u$ with the properties listed for $\Phi_{u_0}:=\Phi|_{u_0^*TN^*}$.
 To see this let $u_{\tau}$, $\tau\in[0,1]$, be a path in $\CC$ connecting $u_0$ with $u_1=u$ and
 define $U\co[0,1]\times\partial\D\ra\hat{W}$ by $U(\tau,z):=u_{\tau}(z)$.
 By \cite[Corollary 3.4.5]{hus94} there exists a trivialisation
 $\Phi_U\co U^*TN^*\ra\big([0,1]\times\partial\D\big)\times\R^n$ that extends $\Phi_{u_0}$.
 As above we denote the coordinates of $\R^n$ by $(x,y,\bfp)$.
 We can assume that $\Phi_U$ restricts to isomorphisms
 $\Phi_U\co\big(U(\,.\,,1)\big)^*T\gamma\ra\big([0,1]\times\{1\}\big)\times\R$ with $\R$ provided
 with the coordinate $x$ as well as
 $\Phi_U\co\big(U(\,.\,,\rmi^k)\big)^*T\big(\vartheta^{-1}(\rmi^k)\big)\ra
 \big([0,1]\times\{\rmi^k\}\big)\times\R^{n-1}$, $k=1,2$, with $\R^{n-1}$ provided
 with coordinates $(y/x,\bfp)$ and co-oriantations $-\partial_x/-\partial_y$, resp.
 The claimed trivialisation $\Phi_u$ is $\Phi_{u_1}=\Phi_U|_{\{1\}\times\partial\D}$.
 
 It remains to show homotopic uniqueness of $\Phi_u$, i.e.\ that $\Phi_u$ is independent of the
 chosen path $u_{\tau}$ up to homotopy:
 We consider a loop $u_{\tau}$ in $\CC$ for $\tau\in T^1=\R/2\Z$ extending a path $u_{\tau}$,
 $\tau\in[0,1]$, in $\CC$ as above and define $\hat{U}\co T^1\times\partial\D\ra\hat{W}$ by
 $\hat{U}(\tau,z):=u_{\tau}(z)$.
 The claim will follow by constructing a trivialisation $\Phi_{\hat{U}}$ that shares the triviality
 properties established for $\Phi_U$.
 
 Restricted to $[0,1]\times\partial\D$ we define $\Phi_{\hat{U}}$ to be equal to $\Phi_U$.
 As $\hat{U}(T^1\times\{1\})$ is a subset of the embedded interval
 $\gamma=[0,1]$ and the tangent bundle $TN^*$ is trivialised by
 $\partial_x,\partial_y,\partial_{\bfp}$ along $\gamma\cap U_B$ there exists by
 \cite[Corollary 3.4.8]{hus94} a trivialisation
 $\Phi_{\gamma}\co  T_{\gamma}N^*\ra[0,1]\times\R^n$ that extends the canonical trivialisation
 over $\gamma\cap U_B$ such that $(x,y,\bfp)$ are coordinates on $\R^n$ and such that
 $\Phi_{\gamma}\co T\gamma\ra[0,1]\times\R$ is provided with the fibre coordinate $x$.
 
 Gluing the trivialisations $\Phi_{\hat{U}}$ and $\Phi_{\gamma}$ via the identity along the overlap
 we obtain a trivialisation (still denoted by) $\Phi_{\hat{U}}$ over
 $\big([0,1]\times\partial\D\big)\cup\big(T^1\times\{1\}\big)$.
 In other words, $\Phi_{\hat{U}}$ trivialises $\hat{U}^*TN^*$ over the boundary of the $2$-disc
 \[
 \big(T^1\times\partial\D\big)\setminus
 \Big(\big([0,1]\times\partial\D\big)\cup\big(T^1\times\{1\}\big)\Big)
 \,.
 \]
 By the assumption that the second Stiefel--Whitney class of $N^*$ vanishes this trivialisation
 extends to a trivialisation of $\hat{U}^*TN^*$, see \cite[p.~75 and Section 3.3]{hat17}.
 Hence, $\Phi_{\hat{U}}\co \hat{U}^*TN^*\ra\big(T^1\times\partial\D\big)\times\R^n$ is a
 trivialisation with fibre coordinates $(x,y,\bfp)$.
 By construction we have a trivialisation
 $\Phi_{\hat{U}}\co\big(\hat{U}(\,.\,,1)\big)^*T\gamma\ra\big(T^1\times\{1\}\big)\times\R$
 with fibre coordinate $x$.
 Further, because a co-oriented linear subspace of $\R^n$ of codimension $1$ is determined
 by the normal vector and $S^{n-1}$, $n\geq3$, is simply connected we can assume that we
 have trivialisations $\Phi_{\hat{U}}\co\big(\hat{U}(\,.\,,\rmi^k)\big)^*
 T\big(\vartheta^{-1}(\rmi^k)\big)\ra\big(T^1\times\{\rmi^k\}\big)\times\R^{n-1}$, $k=1,2$, with
 fibre coordinates $(y/x,\bfp)$ and co-oriantations $-\partial_x/-\partial_y$, resp.
 
 Consequently, $u_1^*TN^*$ shares the same triviality properties as $u_0$ independently of the
 chosen path $u_t$ such that $\Phi_{u_1}$ is homotopically unique as claimed.
\end{rem}

\begin{rem}
 \label{rem:homotopunitrivrelspin}
 If the second Stiefel--Whitney class $w_2(TN^*)$ of $N^*$ is not trivial a variant of Remark
 \ref{rem:homotopunitrivw2=0} gives {\bf homotopically unique stable trivialisations}
 assuming $N^*$ to be orientable and that $w_2(TN^*)$ lifts to a class in $H^2(\hat{W};\Z_2)$.
 
 Following \cite[Chapter 8.1]{fooo09b} we choose a triangulation of $\hat{W}$ such that $N$ will
 be a subcomplex and $B\cup\partial N$ a subcomplex of $N$.
 The assumptions made allow the choice of a {\bf relative spin structure} on $(\hat{W},N^*)$
 which is a choice of orientation on $N^*$, an oriented vector bundle $V$ over the $3$-skeleton
 $\hat{W}_{[3]}$ of $\hat{W}$ such that $w_2(V)$ restricts to $w_2(TN^*)$, and a spin structure
 on the vector bundle $TN^*\oplus V$ over the $2$-skeleton $N^*_{[2]}$ of $N^*$.
 Such a choice of a spin structure is possible because $w_2$ of $TN^*\oplus V$ over
 $N^*_{[2]}$ vanishes, see \cite{bh58}.
 
 As in Remark \ref{rem:homotopunitrivw2=0} we consider the space of continuous maps
 $(\D,\partial\D)\ra\big(\hat{W},N^*\big)$ that map $\{1\}$ and $\{\rmi^k\}$ into $\gamma$ and
 $\vartheta^{-1}(\rmi^k)$, $k=1,2$, resp.
 By simplicial approximation (see \cite[Theorem IV.22.10]{bred93})
 we can replace all maps and homotopies of maps by simplicial
 representatives $u$ and $u_t$ up to homotopy.
 Therefore, the proof of \cite[Theorem 8.1.1]{fooo09b} yields homotopically unique trivialisations
 of $u^*(TN^*\oplus V)$ and $u^*V$.
 Similarly to Remark \ref{rem:homotopunitrivw2=0} we can achieve that $T\gamma$ and
 $T\big(\vartheta^{-1}(\rmi^k)\big)$, $k=1,2$, correspond to $\{1\}\times\R$ and
 $\{\rmi^k\}\times\R^{n-1}$, resp., in the trivialisation $\D\times\R^{n+v}$ of $u^*(TN^*\oplus V)$,
 where $v$ denotes the rank of the vector bundle $V$.
 Moreover, by possibly changing the spin structure on $TN^*\oplus V$ over $N^*_{[2]}$ we can
 assume that the obtained trivialisation of $u_0^*(TN^*\oplus V)$ for the Bishop disc $u_0$ is
 homotopic to the canonical one induced by $\Phi_{u_0}$, see Remark
 \ref{rem:homotopunitrivw2=0}.
\end{rem}

\begin{rem}
 \label{rem:crisorienbyrelspinstr}
 The canonical trivialisations of the involved pull back bundles in Remark
 \ref{rem:homotopunitrivw2=0} and \ref{rem:homotopunitrivrelspin} orient the
 Cauchy--Riemann section $\CR{\!\hat{J}}\co\ZZ\ra\WW$ in a natural way.
 This is based on \cite[Lemma 8.1.4]{fooo09b}.
 
 Namely, given a complex bundle pair $(E,F)$ over $(\D,\partial\D)$ such that the real
 sub-bundle $F$ is trivial over $\partial\D$ each trivialisation orients the associated linear
 Cauchy--Riemann operator.
 The complexification of the trivialisation extends to a complex trivialisation of $E$ over an
 annulus neighbourhood of $\partial\D$.
 Collapsing the inner boundary component of a slightly smaller annulus neighbourhood of
 $\partial\D$ yields a complex bundle pair over a one-noded disc.
 Over the sphere component the Cauchy--Riemann operator admits the complex orientation,
 which is canonical.
 Over the disc component the Cauchy--Riemann operator is onto with kernel consisting of
 constant sections.
 Hence, the kernel is isomorphic to an Euclidean space canonically,
 so that the Cauchy--Riemann operator is canonically oriented over the disc component.
 Incorporating the matching condition of the bundles over the two components the functoriality
 properties of the determinant line bundle canonically determine an orientation
 of the Cauchy--Riemann operator on $(E,F)$, see \cite[Lemma 8.1.4]{fooo09b} and cf.\ 
 \cite[Section 5.10]{hwz-gw17}.
 
 Observe that this construction is compatible with point-wise boundary conditions and also
 allows to begin with a complex bundle pair $(E,F)$ with matching conditions over a noded disc.
 
 In order to orient the linearised Cauchy--Riemann operator at an un-noded element $u$ of
 $\HH^{3,\sigma}(S,j)$ apply the above construction to the complex bundle pair
 $\big(u^*T\hat{W},u^*TN^*\big)$ in the context of Remark \ref{rem:homotopunitrivw2=0}
 (restricting to the connected component of discs homotopic to a local Bishop disc), and to
 the complex bundle pairs $\big(u^*(T\hat{W}\oplus V_{\C}),u^*(TN^*\oplus V)\big)$ and
 $(u^*V_{\C},u^*V)$, where $V_{\C}:=V\otimes\C$, in the context of Remark
 \ref{rem:homotopunitrivrelspin}, resp.
 For the latter use the arguments in the proof of \cite[Theorem 8.1.1]{fooo09b} and the
 observation that the noded discs in $\HH^{3,\sigma}(S,j)$ are at least of codimension $2$.
 In fact, we obtain canonical orientations of the linearised Cauchy--Riemann operator at
 noded elements of $\HH^{3,\sigma}(S,j)$ also with the above construction.
 
 With the proceeding remarks a canonical orientation of the Cauchy--Riemann section
 $\CR{\!\hat{J}}\co\ZZ\ra\WW$ is obtained as in \cite[Section 5.11]{hwz-gw17}.
 Simply, replace the {\it complex orientation} of the sphere case by the canonical orientation
 induced by boundary trivialisations of pull back bundles in the arguments of
 \cite[Section 5.11]{hwz-gw17}.
 Furthermore observe that preservation of orientations of the partial Kodaira differentials on
 the Riemann moduli spaces is ensured by Theorem \ref{thm:rnisaorientedorbifold},
 $\ZZ_{\tau}$ is at least of codimension $2$ for all non-trivial nodal types $\tau$ by
 Proposition \ref{prop:rtauisacomplexorbifold}, automorphisms of nodal discs in $\ZZ$
 restrict to the identity on the disc component as well as that we can collapse the interior
 boundary component of a small collar annulus in \cite[Lemma 8.1.4]{fooo09b} such that
 auxiliary marked points are contained on the resulting sphere components exclusively.
\end{rem}

\begin{proof}[{\bf Proof of Theorem \ref{thm:maindirectsyplcoborthm} part (ii)}]
  We place ourselves in to the situation of Section \ref{subsec:complthecob} and
  \ref{subsec:thesemiposcase}; but this time we do not assume semi-positivity as
  in Theorem \ref{thm:maindirectsyplcoborthm} part (i).
  Instead, we assume the vanishing of $w_2(TN^*)$ or the relative spin condition as formulated
  in Theorem \ref{thm:maindirectsyplcoborthm} part (ii) so that Remark
  \ref{rem:crisorienbyrelspinstr} applies.
  The aim is to derive a contradiction to the existence of a compact subset $K$ of $\hat{W}$
  such that $u(S)\subset K$ for all $\bfu=\big[S,j,D,\{1,\rmi,-1\},u\big]\in\NN$.
  Theorem \ref{thm:maindirectsyplcoborthm} part (ii) will then follow as in the proof of part (i).
  
  We argue by contradiction assuming that such a compact subset $K$ as above exists.
  The arguments form Remark \ref{rem:gromovcomp} under the assumed $C^0$-bounds on
  $\NN$ combined with Section \ref {subsec:modulisp} show compactness of $\NN_{\cut}$,
  see Remark \ref{rem:compactimageunderev}.
  Let $W_{\!K}$ be a relative compact open neighbourhood of $K$ in $\hat{W}$.
  Using Sobolev embedding we choose a neighborhood $\UU\subset\ZZ$ of
  $\NN_{\cut}$ such that $u(S)\subset W_{\!K}\setminus\big(U_B'\cup U_{\partial N}\big)$ for all
  $\bfu=\big[S,j,D,\{1,\rmi,-1\},u\big]$ in $\UU$, where $U_B'\subset U_B$ is defined as $U_B$
  but with $\delta$ replaced by $\delta/2$ in the proof of Lemma \ref{lem:semiglobunique}.
  
  Let $\lambda\co\WW\ra\Q\cap [0,\infty)$ be a {\bf scale$^+$-multisection} of $p\co\WW\ra\ZZ$,
  i.e.\ $\lambda$ is a groupoidal functor which in a local presentation is given by finitely many
  weighted local scale$^+$-sections $(s_i,w_i)$, $w_i\in\Q\cap(0,\infty)$, of total weight
  $\sum w_i=1$ such that $\lambda(\boldxi)$ is the sum of those weights $w_i$ for which the
  corresponding sections $s_i$ satisfy $s_i\big(p(\boldxi)\big)=\boldxi$; we set
  $\lambda(\boldxi)=0$ if there is no such section among the $s_i$, cf.\
  \cite[Definition 3.34]{hwz-III09}.
  The {\bf support} of $\lambda$ is the smallest closed set in $\ZZ$ outside which $\lambda$
  is trivial in the sense that $\lambda(0_{\bfu})=1$ for these $\bfu\in\ZZ$, see
  \cite[Definition 3.35]{hwz-III09}.
  The {\bf solution set}
  \[
  \SS=\Big\{\bfu\in\ZZ\;\Big|\;\lambda\big(\CR{\!\hat{J}}\bfu\big)>0\Big\}
  \]
  of the pair $\big(\CR{\!\hat{J}},\lambda\big)$ is the set of all
  $\bfu=\big[S,j,D,\{1,\rmi,-1\},u\big]\in\ZZ$ for which in a local presentation of $\lambda$ there
  exist at least one $s_i$ such that $\CR{\!\hat{J}}u=s_i(u)$ and
  $\lambda\big(\CR{\!\hat{J}}\bfu\big)$ is the sum of all the weights $w_i$ for which the
  corresponding $s_i$ satisfy such an equation.
  The solution set $\SS$ is equipped with the {\bf weight function}
  \[
  \lambda_{\CR{\!\hat{J}}}\co\ZZ\lra\Q\cap(0,\infty)
  \,,
  \quad
  \bfu\longmapsto\lambda\big(\CR{\!\hat{J}}\bfu\big)
  \,,
  \]
  see \cite[Section 4.3]{hwz-III09}.
  
  With \cite[Theorem 4.17]{hwz-III09} we choose $\lambda$ such that the support of $\lambda$
  is contained in $\UU$ and that $\big(\CR{\!\hat{J}},\lambda\big)$ is {\bf transverse}.
  The latter means that the vertical differentials
  \[\big(\CR{\!\hat{J}}\big)'(u)-s'_i(u)\]
  of local presentations $\CR{\!\hat{J}}u$ of $\CR{\!\hat{J}}\bfu$ and $s_i$ of $\lambda$ are
  surjective for all $\bfu\in\SS$ and for all (the finitely many) $i$, see
  \cite[Definition 4.7(1)]{hwz-III09}.
  If $\big(\CR{\!\hat{J}}\big)'(u)$ is onto for an un-noded $\bfu\in\NN$, which is representable by
  a necessarily simple holomorphic disc map, we choose $\lambda$ to be a single local section
  $s_1$ that is identically $0$ in a neighbourhood of $\bfu$ in $\ZZ$.
  This is possible in view of the proof of \cite[Theorem 4.17]{hwz-III09}.
  In particular, $\lambda$ is trivial over those $\bfu$.
  As observed right before Remark \ref{rem:homotopunitrivw2=0} this applies to all local Bishop
  discs $\bfu_{\varepsilon,b_o}$, $\varepsilon\in[\delta/2,\delta)$, so that $\lambda$ is trivial
  over all local Bishop discs $\bfu_{\varepsilon,b_o}$.
  Consequently, the truncated solution set
  \[
  \SS_{\cut}=\SS\setminus(0,\delta/2)
  \]
  of $\big(\CR{\!\hat{J}},\lambda\big)$ is a $1$-dimensional oriented compact branched
  {\it suborbifold} with boundary $\partial\SS_{\cut}$ given by the single Bishop disc
  $\bfu_{\delta/2,b_o}$, see \cite[Theorem 4.17]{hwz-III09} or \cite[Section 1.4]{hwz-gw17}.
  A collar neighbourhood of $\partial\SS_{\cut}$ in $\SS_{\cut}$ is equal to a collar
  neighbourhood of $\partial\NN_{\cut}$ in $\NN_{\cut}$ given by the local Bishop discs
  $\bfu_{\varepsilon,b_o}$, $\varepsilon\in[\delta/2,\delta)$.
  
  Furthermore observe that by compactness of $\SS_{\cut}$
  the intersection $\SS_{\cut}\cap\ZZ_{\tau}$
  is not empty only for finitely many nodal types $\tau$.
  Therefore,
  we choose $\big(\CR{\!\hat{J}},\lambda\big)$ to be transverse along the subpolyfolds
  $\ZZ_{\tau}$ for these nodal types $\tau$ turning the subsets $\SS_{\cut}\cap\ZZ_{\tau}$
  into suborbifolds of $\SS$.
  As the codimensions will be at least $2$ whenever the nodal type $\tau$ is non-trivial,
  the resulting suborbifolds $\SS_{\cut}\cap\ZZ_{\tau}$ have negative dimension,
  hence, are empty.
  Therefore, all elements in $\SS_{\cut}$ are un-noded
  and have trivial isotropy as they can be represented
  by un-noded stable nodal disc maps with trivial automorphism group.
  In other words,
  $\SS_{\cut}$ is a $1$-dimensional oriented compact branched {\it manifold} with precisely
  one boundary point, which has weight $1$.
  This contradicts the fact that by \cite[Lemma 5.11]{sal99}
  the oriented sum of the weights taken over
  all boundary points vanishes.
\end{proof}

\begin{rem}
  \label{rem:alternativeargument}
  We give an alternative argument to obtain a contradiction which does not use the classification of
  $1$-dimensional oriented compact branched manifolds with boundary given in \cite[Section 5.4]{sal99}:
  We identify $\gamma$ with the interval $[0,3\delta]$ such that $(0,\delta)$ corresponds to the local
  Bishop family and $[2\delta,3\delta]$ is not contained in the image of the evaluation map
  $\ev\co\SS\ra\gamma$ that evaluates $\bfu$ at the first boundary marked point $1$.
  Let $f$ be a smooth function on $[0,3\delta]$ with support in $(\delta/2,\delta)$ such that
  $\int_0^{3\delta}f(x)\rmd x=1$.
  Because $\ev$ restricts to a degree $1$ map on the local Bishop discs,
  \[
  \int_{(\SS_{\cut},\lambda _{\DB})}\ev^*(f\rmd x)=1
  \]
  writing $\lambda_{\DB}$ for the weight function $\lambda_{\CR{\!\hat{J}}}$.
  Denote by $f_1$ the function obtained from $f$ by shifting $f$ by $2\delta$ and observe that the
  closed $1$-form $(f-f_1)\rmd x$ has a primitive $g(x)=\int_0^x\big(f(t)-f_1(t)\big)\rmd t$
  with support in $(\delta/2,3\delta)$.
  Hence, $\ev^*\big((f-f_1)\rmd x\big)$ has primitive $\ev^*g$ and
  \[
  \int_{(\SS_{\cut},\lambda _{\DB})}\ev^*(f_1\rmd x)=0
  \]
  as the support of $f_1$ is contained in $(5\delta/2,3\delta)$.
  With Stokes theorem \cite[Theorem 1.27]{hwz-int10} for weighted integrals
  \[
  1=\int_{(\SS_{\cut},\lambda _{\DB})}\ev^*\big((f-f_1)\rmd x\big)=
  \int_{(\partial\SS_{\cut},\lambda _{\DB})}\ev^*g=
  g(\delta/2)\,.
  \]
  As $g(\delta/2)=0$ we reach the desired contradiction.
\end{rem}


\begin{ack}
  Part of the research of this article was initiated while S.S.\ and K.Z.'s stay at the
  {\it Institut for Advanced Studies} Princeton.
  We would like to thank Helmut Hofer for the invitation and for introducing us to polyfolds.
  K.Z. would like to thank the {\it Mathematisches Forschungsinstitut Oberwolfach} for the support.
  We would like to thank Katrin Wehrheim for her constant support and motivation.
  We thank Peter Albers, Matthew Strom Borman, Joel Fish, Hansj\"org Geiges,
  Stefan Nemirovski, and Sven Pr\"ufer
  for enlightening discussions.
\end{ack}

\begin{coi}
  Not Applicable
\end{coi}




\begin{thebibliography}{10}
%
\bibitem{abb14}
  {\sc C. Abbas},
  \textit{An introduction to compactness results in symplectic field
              theory},
 Springer, Heidelberg, (2014),
 viii+252.
%
\bibitem{ach05}
  {\sc C. Abbas, K. Cieliebak, H. Hofer},
  The Weinstein conjecture for planar contact structures in dimension three,
  \textit{Comment. Math. Helv.} {\bf 80} (2005),
  771--793.
%
\bibitem{ah19}
  {\sc C. Abbas, H. Hofer},
  \textit{Holomorphic curves and global questions in contact geometry},
  Birkh\"{a}user, (2019),
  xii+322.
%
\bibitem{af03}
  {\sc R. A. Adams, J. J. F. Fournier},
  \textit{Sobolev spaces},
  Pure and Applied Mathematics (Amsterdam),
  {\bf 140}, Second, Elsevier/Academic Press, Amsterdam, (2003), xiv+305.
%
\bibitem{abklr94}
  {\sc B. Aebischer, M. Borer, M. K\"{a}lin, Ch. Leuenberger, H. M. Reimann},
  \textit{Symplectic geometry}, Progress in Mathematics, {\bf 124}, An introduction based
  on the seminar in Bern, 1992, Birkh\"{a}user Verlag, Basel, (1994),
  xii+239.
%
\bibitem{ah09}
  {\sc P. Albers, H. Hofer},
  On the {W}einstein conjecture in higher dimensions,
  \textit{Comment. Math. Helv.} {\bf 84} (2009),
  429--436.
%
\bibitem{aud94}
  {\sc M. Audin},
  Symplectic and almost complex manifolds, with an appendix by P. Gauduchon,
  Chapter~II of~\cite{aula94}, pp.~41--74.
%
\bibitem{aula94}
  {\sc M. Audin and J. Lafontaine} (eds.),
  {\it Holomorphic Curves in Symplectic Geometry},
  Progr. Math. {\bf 117},
  Birkh\"auser Verlag, Basel (1994).
%
\bibitem{alp94}
  {\sc M. Audin, J. Lafontaine and L. Polterovich},
  Symplectic rigidity: Lagrangian submanifolds,
  Chapter~X of~\cite{aula94}, pp.~271--321.
%
\bibitem{bschz19}
  {\sc K. Barth, J. Schneider, K. Zehmisch},
  Symplectic dynamics of contact isotropic torus complements,
  \textit{M\"{u}nster J. Math.} \textbf{12} (2019),
  31--48.
%
\bibitem{bem15}
  {\sc M. S. Borman, Y. Eliashberg, E. Murphy},
  Existence and classification of overtwisted contact structures
              in all dimensions,
  \textit{Acta Math.} \textbf{215} (2015),
  281--361.
%
\bibitem{bh58}
  {\sc A. Borel, F. Hirzebruch},
  Characteristic Classes and Homogeneous Spaces, I,
  \textit{Amer. J. Math.} \textbf{80} (1958),
  458--538.
%
\bibitem{bou09}
  {\sc F. Bourgeois},
  A survey of contact homology,
  in: \textit{New perspectives and challenges in symplectic field theory},
  CRM Proc. Lecture Notes, 49, Amer. Math. Soc., Providence, RI, 2009,
  45--71.
%
\bibitem{bvk10}
  {\sc F. Bourgeois, O. van Koert},
  Contact homology of left-handed stabilizations and plumbing of
              open books,
  \textit{Commun. Contemp. Math.} \textbf{12} (2010),
  223--263.     
%
\bibitem{behwz03}
  {\sc F. Bourgeois, Y. Eliashberg, H. Hofer, K. Wysocki,
  E. Zehnder},
  Compactness results in symplectic field theory,
  \textit{Geom. Topol.} {\bf 7} (2003),
  799--888.
%
\bibitem{bred93}
  {\sc G. E. Bredon},
  \textit{Topology and Geometry},
  Grad. Texts in Math. \textbf{139},
  Springer-Verlag, New York (1993).
%
\bibitem{cmp19}
  {\sc R. Casals, E. Murphy, F. Presas},
  Geometric criteria for overtwistedness,
  \textit{J. Amer. Math. Soc.} {\bf 32} (2019),
  563--604.
%
\bibitem{cie18}
  {\sc K. Cieliebak},
  Nonlinear Functional Analysis, lecture note, (2018).
%
\bibitem{ce12}
  {\sc K. Cieliebak, Y. Eliashberg},
  \textit{From {S}tein to {W}einstein and back},
  American Mathematical Society Colloquium Publications, 59,
  Symplectic geometry of affine complex manifolds,
  American Mathematical Society, Providence, RI, 2012,
  xii+364.
%
\bibitem{dg12}
  {\sc F. Ding, H. Geiges},
  Contact structures on principal circle bundles,
  \textit{Bull. Lond. Math. Soc.} {\bf 44} (2012),
  1189--1202.
%
\bibitem{dgz14}
  {\sc M. D\"orner, H. Geiges, K. Zehmisch},
  Open books and the Weinstein conjecture,
  \textit{Q. J. Math.} {\bf 65} (2014),
  869--885.
%
\bibitem{eli89}
  {\sc Y. Eliashberg},
  Classification of overtwisted contact structures on {$3$}-manifolds,
  \textit{Invent. Math.} {\bf 98} (1989),
  623--637.
%
\bibitem{eli90}
  {\sc Y. Eliashberg},
  Filling by holomorphic discs and its applications,
  in:\ \textit{Geometry of Low-Dimensional Manifolds, Vol.~2} (Durham, 1989),
  London Math. Soc. Lecture Note Ser. {\bf 151},
  Cambridge University Press (1990),
  45--67.
%
\bibitem{egh00}
  {\sc Y. Eliashberg, A. Givental, H. Hofer},
  Introduction to symplectic field theory,
  \textit{Geom. Funct. Anal.} {\bf 2000},
  Special Volume, Part II,
  560--673.
%
\bibitem{fz15}
  {\sc U. Frauenfelder, K. Zehmisch},
  Gromov compactness for holomorphic discs with totally real
              boundary conditions,
  \textit{J. Fixed Point Theory Appl.} {\bf 17} (2015),
  521--540.
%
\bibitem{fooo09b}
  {\sc K. Fukaya, Y.-G. Oh, H. Ohta, K. Ono},
  \textit{Lagrangian intersection {F}loer theory: anomaly and obstruction. {P}art {II}},
  AMS/IP Studies in Advanced Mathematics, {\bf 46}, American Mathematical Society,
  Providence, RI; International Press, Somerville, MA, (2009),
  i--xii and 397--805.
%
\bibitem{gz10}
  {\sc H. Geiges, K. Zehmisch},
  Eliashberg's proof of {C}erf's theorem,
  \textit{J. Topol. Anal.} \textbf{2} (2010),
  543--579.
%
\bibitem{gz12}
  {\sc H. Geiges, K. Zehmisch},
  Symplectic cobordisms and the strong Weinstein conjecture,
  \textit{Math. Proc. Cambridge Philos. Soc.} \textbf{153} (2012),
  261--279.
%
\bibitem{gz13}
  {\sc H. Geiges, K. Zehmisch},
  How to recognise a 4-ball when you see one,
  \textit{M\"unster J. Math.} {\bf 6} (2013),
  525--554.
%
\bibitem{gz16b}
  {\sc H. Geiges, K. Zehmisch},
  Reeb dynamics detects odd balls,
  \textit{Ann. Sc. Norm. Super. Pisa Cl. Sci. (5)} {\bf 15} (2016),
  663--681.
%
\bibitem{gz23}
  {\sc H. Geiges, K. Zehmisch},
  \textit{A Course on Holomorphic Discs},
  Birkh\"{a}user Advanced Texts Basler Lehrb\"{u}cher,
  Birkh\"{a}user/Springer, Cham, 2023, xviii+189 pp.  
%
\bibitem{gro85}
  {\sc M. Gromov},
  Pseudoholomorphic curves in symplectic manifolds,
  \textit{Invent. Math.} {\bf 82} (1985),
  307--347.
%
\bibitem{gro86}
  {\sc M. Gromov},
  \textit{Partial differential relations},
  Ergebnisse der Mathematik und ihrer Grenzgebiete (3),
  {\bf 9}, Springer-Verlag, Berlin, (1986),
  x+363.
 %
\bibitem{hat17}
  {\sc A. Hatcher},
  \textit{Vector Bundles \& K-Theory}, (2017),\\
  {\tt http://pi.math.cornell.edu/{\raisebox{-0.7ex}{\~ { }}}hatcher/VBKT/VBpage.html}
%
\bibitem{hofe93}
  {\sc H. Hofer},
  Pseudoholomorphic curves in symplectizations with applications to the
  Weinstein conjecture in dimension three,
  \textit{Invent. Math.} {\bf 114} (1993),
  515--563.
%
\bibitem{hwz98}
  {\sc H. Hofer, K. Wysocki, E. Zehnder},
  The dynamics on three-dimensional strictly convex energy surfaces,
  \textit{Ann. of Math. (2)} {\bf 148} (1998),
  197--289.
%
\bibitem{hwz03}
  {\sc H. Hofer, K. Wysocki, E. Zehnder},
  Finite energy foliations of tight three-spheres and Hamiltonian dynamics,
  \textit{Ann. of Math. (2)} {\bf 157} (2003),
  125--255.
%
\bibitem{hwz-I07}
  {\sc H. Hofer, K. Wysocki, E. Zehnder},
  A general {F}redholm theory. {I}. {A} splicing-based differential geometry,
  \textit{J. Eur. Math. Soc. (JEMS)} {\bf 9} (2007),
  841--876.
%
\bibitem{hwz-II09}
  {\sc H. Hofer, K. Wysocki, E. Zehnder},
   A general {F}redholm theory. {II}. {I}mplicit function theorems,
   \textit{Geom. Funct. Anal.} {\bf 19} (2009),
   206--293.
%
\bibitem{hwz-III09}
  {\sc H. Hofer, K. Wysocki, E. Zehnder},
  A general {F}redholm theory. {III}. {F}redholm functors and polyfolds,
  \textit{Geom. Topol.} {\bf 13} (2009),
  2279--2387.
%
\bibitem{hwz-int10}
  {\sc H. Hofer, K. Wysocki, E. Zehnder},
  Integration theory on the zero sets of polyfold {F}redholm sections,
  \textit{Math. Ann.} {\bf 346} (2010),
  139--198.
%
\bibitem{hwz-dm12}
  {\sc H. Hofer, K. Wysocki, E. Zehnder},
  Deligne--Mumford-type spaces with a view towards symplectic field theory,
  preprint, (2012)
%
\bibitem{hwz-gw17}
  {\sc H. Hofer, K. Wysocki, E. Zehnder},
  \textit{Applications of polyfold theory {I}: {T}he polyfolds of {G}romov-{W}itten theory},
  Mem. Amer. Math. Soc. {\bf 248} (2017),
  v+218 pp.
%
\bibitem{hua15}
  {\sc Y. Huang},
  On {L}egendrian foliations in contact manifolds {I}:
              {S}ingularities and neighborhood theorems,
  \textit{Math. Res. Lett.} \textbf{22} (2015),
  1373--1400.
%
\bibitem{hua17}
  {\sc Y. Huang},
  On plastikstufe, bordered {L}egendrian open book and
              overtwisted contact structures,
  \textit{J. Topol.} \textbf{10} (2017),
  720--743.
%
\bibitem{hum97}
  {\sc C. Hummel},
  \textit{Gromov's compactness theorem for pseudo-holomorphic curves},
  Progress in Mathematics, {\bf 151} Birkh\"{a}user Verlag, Basel, (1997),
  viii+131. 
%
\bibitem{hus94}
  {\sc D. Husemoller},
  \textit{Fibre bundles},
  Graduate Texts in Mathematics {\bf 20}, Third,
  Springer-Verlag, New York, (1994),
  xx+353.
%
\bibitem{kwz22}
  {\sc M. Kwon, K. Wiegand, K. Zehmisch},
  Diffeomorphism type via aperiodicity in Reeb dynamics,
  \textit{J. Fixed Point Theory Appl.} {\bf 24} (2022),
  Paper No. 21, 26 pp.
%
\bibitem{lw11}
  {\sc J. Latschev, C. Wendl},
  Algebraic torsion in contact manifolds
  (With an appendix by Michael Hutchings.),
  \textit{Geom. Funct. Anal.}  {\bf 21} (2011),
  1144--1195.
%
\bibitem{mnw13}
  {\sc P. Massot, K. Niederkr{\"u}ger, C. Wendl},
  Weak and strong fillability of higher dimensional contact manifolds,
  \textit{Invent. Math.} {\bf 192} (2013),
  287--373.
%
\bibitem{mcd91}
  {\sc D. McDuff},
  Symplectic manifolds with contact type boundaries,
  \textit{Invent. Math.} {\bf 103} (1991),
  651--671.
%
\bibitem{mcsa98}
  {\sc D. McDuff, D. Salamon},
  \textit{Introduction to symplectic topology},
  Oxford Mathematical Monographs, Second,
  The Clarendon Press Oxford University Press, New York, (1998),
  1--486.
%
\bibitem{mcsa17}
  {\sc D. McDuff, D. Salamon},
  \textit{Introduction to symplectic topology},
  Oxford Graduate Texts in Mathematics, Third,
  Oxford University Press, Oxford, (2017),
  xi+623.
%
\bibitem{mcsa04}
  {\sc D. McDuff, D. Salamon},
  \textit{$J$-holomorphic Curves and Symplectic Topology},
  Amer. Math. Soc. Colloq. Publ. {\bf 52},
  American Mathematical Society, Providence, RI (2004).
%
\bibitem{nie06}
  {\sc K. Niederkr\"uger},
  The plastikstufe -- a generalization of the overtwisted disk to higher
  dimensions,
  \textit{Algebr. Geom. Topol.} \textbf{6} (2006),
  2473--2508.
%
\bibitem{nie13}
  {\sc K. Niederkr{\"u}ger},
  \textit{On fillability of contact manifolds},
  Habilitation \`a diriger des recherches,
  Universit\'e Paul Sabatier - Toulouse III, 2013, tel-00922320.
%
\bibitem{nr11}
  {\sc K. Niederkr{\"u}ger, A. Rechtman},
  The {W}einstein conjecture in the presence of submanifolds
              having a {L}egendrian foliation,
  \textit{J. Topol. Anal.} {\bf 3} (2011),
  405--421.
%
\bibitem{nw11}
  {\sc K. Niederkr{\"u}ger, C. Wendl},
  Weak symplectic fillings and holomorphic curves,
  \textit{Ann. Sci. \'{E}c. Norm. Sup\'{e}r. (4)} {\bf 44} (2011),
  801--853.
%
\bibitem{par19}
  {\sc J. Pardon},
  Contact homology and virtual fundamental cycles,
  \textit{J. Amer. Math. Soc.} {\bf 32} (2019),
  825--919.
%
\bibitem{pol91}
  {\sc L. Polterovich},
  The surgery of {L}agrange submanifolds,
  \textit{Geom. Funct. Anal.} {\bf 1} (1991),
  198--210.
%
\bibitem{rs06}
  {\sc J. Robbin, D. Salamon},
  A construction of the {D}eligne-{M}umford orbifold,
  \textit{J. Eur. Math. Soc. (JEMS)} {\bf 8} (2006),
  611--699.
%
\bibitem{sal99}
  {\sc D. Salamon},
  \textit{ Lectures on {F}loer homology}, in:
  Symplectic geometry and topology ({P}ark {C}ity, {UT}, 1997),
  IAS/Park City Math. Ser. {\bf 7}, Amer. Math. Soc., Providence, RI,
  (1999), 143--229.
%
\bibitem{sal11}
  {\sc D. Salamon},
  \textit{Funktionentheorie},
  Birkh\"auser, Basel, (2011),
  viii + 218.
%
\bibitem{sz17}
  {\sc S. Suhr, K. Zehmisch},
  Polyfolds, cobordisms, and the strong {W}einstein conjecture,
  \textit{Adv. Math.} {\bf 305} (2017),
  1250--1267.
%
\bibitem{wein79}
  {\sc A. Weinstein},
  On the hypotheses of Rabinowitz' periodic orbit theorems,
  \textit{J. Differential Equations} {\bf 33} (1979),
  336--352.
%
\bibitem{wen05}
  {\sc C. Wendl},
  \textit{Finite energy foliations and surgery on transverse links},
  Thesis (Ph.D.)--New York University, (2005),
  xii + 411.
%
\bibitem{wen10}
  {\sc C. Wendl},
  \textit{Lectures on Holomorphic Curves in Symplectic and Contact Geometry},
  preprint, {\tt arXiv:~1011.1690}  
%
\bibitem{zeh15}
  {\sc K. Zehmisch},
  Holomorphic jets in symplectic manifolds,
  \textit{J. Fixed Point Theory Appl.} {\bf 17} (2015),
  379--402.
%
\bibitem{zehm03}
  {\sc K. Zehmisch},
  The Eliashberg--Gromov tightness theorem,
  Diplomarbeit, Universit\"at Leipzig (2003).
\end{thebibliography}
\end{document}